\documentclass[11pt]{amsart}

\usepackage{amsthm,amsmath,amssymb}
\usepackage[T1]{fontenc}
\usepackage[utf8]{inputenc}
\usepackage[english]{babel}
\usepackage{indentfirst}

\usepackage[ruled,vlined]{algorithm2e}
\usepackage{subcaption,comment}
\usepackage{graphicx}
\graphicspath{{figures/}}
\usepackage[dvipsnames]{xcolor}
\usepackage[pagebackref,breaklinks,colorlinks=true,linkcolor=MidnightBlue,citecolor=MidnightBlue]{hyperref}
\usepackage[shortlabels]{enumitem}
\usepackage{tikz}

\numberwithin{equation}{section}

\usepackage{mathFL}

\usepackage{multirow}
\usepackage{float}
\usepackage{subcaption}

\DeclareMathOperator{\cof}{cof}

\DeclareMathOperator{\id}{Id}

\newcommand{\rhomax}{\rho_\mathrm{max}}

\title{The back-and-forth method for\\ Wasserstein gradient flows}
\date{\today}
\author{Matt Jacobs}
\address{UCLA, Los Angeles, CA, USA}
\email{majaco@math.ucla.edu}

\author{Wonjun Lee}
\address{UCLA, Los Angeles, CA, USA}
\email{wlee@math.ucla.edu}

\author{Flavien Léger}
\address{ENS, PSL University, Paris, France}
\email{flavienleger@nyu.edu}

\begin{document}

\begin{abstract}
We present a method to efficiently compute Wasserstein gradient flows.  Our approach is based on a generalization of the back-and-forth method (BFM) introduced in~\cite{jacobslegerbf} to solve optimal transport problems.  We evolve the gradient flow by solving the dual problem to the JKO scheme. In general, the dual problem is much better behaved than the primal problem.  This allows us to efficiently run large scale gradient flows simulations for a large class of internal energies including singular and non-convex energies.  
\end{abstract}

\maketitle



\section{Introduction}

 In this work, we are interested in simulating the evolution of parabolic equations of the form
\begin{equation}\label{eq:pde}
    \begin{split}
        \partial_t \rho - \nabla \cdot ( \rho \nabla \phi) = 0,\\
        \phi = \delta U(\rho).
    \end{split}
\end{equation}
Equation (\ref{eq:pde}), often referred to as Darcy's law or the generalized porous medium equation, describes the evolution of a mass density $\rho$ flowing along a pressure gradient $\nabla \phi$ generated by an internal energy functional $U$.  This class of equations models various physical phenomena  such as fluid flow, heat transfer, aggregation-diffusion, and crowd motion~\cite{vazquez2007porous, otam}. In general, these equations are both stiff and non-linear making them challenging to solve numerically.  For example, in the important special case where $U(\rho) = \frac{1}{m-1} \int \rho^m$ ($m>1$),  equation~\eqref{eq:pde} becomes a non-linear version of the heat equation
 \[
 \partial_t \rho - \Delta (\rho^m) = 0,
\]
known as the porous medium equation (PME). When $U$ is non-differentiable or non-convex, simulation of these equations becomes even more difficult. Thus, in this paper, our goal is to design a method that can efficiently and accurately simulate equation~\eqref{eq:pde} for a wide variety of internal energies $U$.

Our approach to simulating Darcy's law is based on the celebrated interpretation of equation~\eqref{eq:pde} as a gradient flow with respect to the Wasserstein metric~\cite{JKO1998variational,otto2001pme}. This interpretation can be used to create a discrete-in-time approximation scheme known as the JKO scheme~\cite{JKO1998variational}. The scheme constructs approximate solutions by iterating
\begin{equation}\label{eq:variational-scheme}
    \rho^{(n+1)} := \argmin_{\rho} U(\rho) + \frac{1}{2\tau} W^2_2(\rho, \rho^{(n)}).
\end{equation}
Here, $\tau$ plays the role of the time step in the scheme and $W_2(\cdot,\cdot)$ is the $2$-Wasserstein metric from the theory of optimal transportation~\cite{otam} (see Section~\ref{section:background-optimal-transport} for a brief overview of optimal transport and the 2-Wasserstein metric). Thanks to the variational structure of the scheme, the iterates are unconditionally energy stable and one can choose the time step $\tau$ independently from any spatial discretization.  In addition, the JKO scheme retains many desirable properties of the continuum equation, such as comparison and contraction type principles~\cite{jacobskimtongL1, santambrogio_bv, aky2014}.  

In light of the many favorable properties of the JKO scheme, there have been many works devoted to the computation of minimizers for problem (\ref{eq:variational-scheme}),  see~\cite{MR2580954, MR2566595, MR3413589,MR3555350,  MR3565819,MR3635459,  carrillo2019primal, carrillo2020variational,  leclerc2020lagrangian}
to name just a few.  Despite the amount of work on this problem, it remains a challenge to efficiently solve the JKO scheme at a high resolution.   The main difficulty in solving problem (\ref{eq:variational-scheme}) lies in the handling of the Wasserstein distance term.  Indeed, there is not a simple formula that gives the variation of the Wasserstein distance with respect to the density $\rho$.  As such, essentially all methods for solving (\ref{eq:variational-scheme}) are adaptations of algorithms for computing the Wasserstein distance between two fixed densities.  

In this paper, we solve problem (\ref{eq:variational-scheme}) by adapting the back-and-forth method (BFM) introduced in~\cite{jacobslegerbf}.  BFM is a state-of-the-art algorithm for computing optimal transport maps between two fixed densities. 
Instead of directly solving Monge's optimal transportation problem, BFM finds optimal maps by solving the associated Kantorovich dual problem. Building on this approach, rather than directly solving problem (\ref{eq:variational-scheme}), we instead compute solutions to its dual problem.   The dual problem is a concave maximization problem that produces the pressure variable at the next time step $\phi^{(n+1)}$.  The optimal density variable can then easily be recovered from the pressure via the duality relation $\phi^{(n+1)}= \delta U(\rho^{(n+1)})$. 

There are several advantages to solving the dual problem rather than the original primal problem.  The pressure variable $\phi$ has better regularity than the density variable $\rho$.  Indeed, at worst, the pressure gradient must be square integrable.  As a result, the pressure is better suited to discrete approximation schemes.  In addition, there is an explicit formula to compute derivatives of the dual functional, hence one can apply gradient ascent to solve the dual problem (the corresponding gradient descent scheme for the primal problem is much more difficult).  
Finally, the dual approach is very convenient when $U$ encodes hard constraints (such as incompressibility of the density), as the dual problem will be unconstrained.

Leveraging the advantages of the dual problem to (\ref{eq:variational-scheme}) and the special gradient ascent structure of BFM, we are able to rapidly and accurately solve the JKO scheme for a large class of internal energies $U$. 
We show that the algorithm increases the value of the dual problem at every step.   In particular, this analysis holds even in cases where the Hessian of $U$ is singular and our analysis has no dependence on the size of the computational grid.  As a result, we are able to simulate equation (\ref{eq:pde}) on a much larger scale than previous methods, and we are easily able to handle difficult cases like incompressible crowd motion models with obstacles and aggregation-diffusion equations.

\subsection{Overall approach}\label{section:intro-overall-approach}
The back-and-forth method for Wasserstein gradient flows is based on solving the dual problem associated to the JKO scheme. 
The starting point for this analysis is Kantorovich's dual formulation of optimal transport.  Given two measures $\mu$ and $\nu$, the dual formulation of the 2-Wasserstein distance is given by
\begin{equation}\label{eq:kantorovich-dual0}
\frac{1}{2\tau}W_2^2(\mu,\nu)=\sup_{(\phi,\psi)\in\mathcal{C}} \int_{\Omega} \psi(x)\,d\mu(x)-\int_{\Omega}\phi(y)\,d\nu(y),
\end{equation}
where we maximize over the constraint
\begin{equation*}
    \mathcal{C}:=\{(\phi,\psi)\in C(\Omega)\times C(\Omega): \psi(x)-\phi(y)\leq \frac{1}{2\tau}|x-y|^2\}.
\end{equation*}

Using the dual formulation of optimal transport, we can rewrite problem (\ref{eq:variational-scheme}) as
\[
\inf_{\rho}\sup_{(\phi, \psi)\in \mathcal{C}} U(\rho) + \int_{\Omega} \psi(x)\,d\rho^{(n)}(x)-\int_{\Omega}\phi(y)\,d\rho(y).
\]
When $U$ is convex, we can interchange the inf and sup to get an equivalent dual problem to (\ref{eq:variational-scheme}):
\begin{equation}\label{eq:got_dual}
    \sup_{(\phi, \psi)\in \mathcal{C}}  \int_{\Omega}\psi(x)\,d\rho^{(n)}(x)-U^*(\phi),
\end{equation}
where $U^*$ is the convex conjugate of $U,$
\[
U^*(\phi):=\sup_{\rho} \int_{\Omega} \phi(y)\,d\rho(y)-U(\rho).
\]

Problem (\ref{eq:got_dual}) looks difficult due to the constraint encoded by $\mathcal{C}$.  Nevertheless, there is a very convenient way to reformulate the problem.  Because $\rho^{(n)}$ is a nonnegative measure, it is favorable to choose $\psi$ to be pointwise as large as possible.  If we fix $\phi$, it then follows that the corresponding largest possible choice for $\psi$ is given by
\begin{equation}\label{eq:c_transform}
\phi^c(x):=\inf_{y\in\Omega} \phi(y)+\frac{1}{2\tau}|x-y|^2.
\end{equation}
Conversely, $U^*$ is increasing with respect to $\phi$ (see Section~\ref{section:background-optimal-transport}), therefore, we would like to choose $\phi$ to be pointwise as small as possible.  Thus, if we fix $\psi$, then the corresponding smallest choice for $\phi$ is given by
\begin{equation}\label{eq:conj_c_transform}
    \psi^{\bar{c}}(y):=\sup_{x\in\Omega} \psi(x)-\frac{1}{2\tau}|x-y|^2.
\end{equation}

Formulas (\ref{eq:c_transform}) and (\ref{eq:conj_c_transform}) are known as the backward-$c$-transform and forward-$c$-transform respectively. These transforms play an essential role in optimal transport and are integral to our method.  Crucially, we can use these transforms to eliminate the constraint $\mathcal{C}$ and either one of the variables $\phi$ or $\psi$.  More explicitly, problem (\ref{eq:got_dual}) is equivalent to maximizing either one of the following two \emph{unconstrained} functionals:
\begin{equation}\label{eq:J_0}
J(\phi):=\int_{\Omega} \phi^c(x)\, d\rho^{(n)}(x)-U^*(\phi),
\end{equation}
\begin{equation}\label{eq:I_0}
I(\psi):=\int_{\Omega} \psi(x)\, d\rho^{(n)}(x)-U^*(\psi^{\bar{c}}).
\end{equation}
Indeed, if $\phi_*$ is a maximizer of $J$ and $\psi_*$ is a maximizer of $I$, then we must have the relations
\[
\phi_*^c=\psi_*, \quad \psi_*^{\bar{c}}=\phi_*,
\]
and $(\phi_*, \psi_*)$ is a maximizer of (\ref{eq:got_dual}).  The reformulations $I$ and $J$ genuinely simplify the task of finding maximizers.  On a regular discrete grid, the $c$-transform can be computed very efficiently~\cite{DBLP:journals/na/Lucet97, jacobslegerbf}.  As a result, it is much more tractable to maximize $I$ and $J$, rather than trying work with (\ref{eq:got_dual}) directly.

We will find the maximizers $\phi_*$ and $\psi_*$ by building upon the BFM algorithm introduced in~\cite{jacobslegerbf}.  
 The original BFM gives a very efficient scheme for finding the maximizers in the special case where $U^*$ is a linear functional. Rather than focusing on either $I$ or $J$, BFM simultaneously maximizes both functionals.    The method proceeds by hopping back-and-forth between gradient ascent updates on $J$ in $\phi$-space and gradient ascent updates on $I$ in $\psi$-space (hence the name).  
 In between gradient steps, information in one space ($\phi$-space or $\psi$-space)  is propagated back to the other by taking a forward/backward $c$-transform.
 As noted in~\cite{jacobslegerbf}, the advantage of the back-and-forth approach is that certain features of the optimal solution pair $(\phi_*, \psi_*)$ may be easier to build in one space compared to the other.  As a result, the back-and-forth method converges far more rapidly than vanilla gradient ascent methods that operate only on $\phi$-space or only on $\psi$-space.  
 
In order to generalize BFM to the Wasserstein gradient flow case, we need to be able to guarantee the stability of gradient ascent steps on (\ref{eq:J_0}) and (\ref{eq:I_0}) when $U^*$ is nonlinear.  In fact, for many important cases, the Hessian of $U^*$ may have a singular component.  To overcome this difficulty, we perform the gradient ascent steps in an appropriately weighted Sobolev space.  The Sobolev control allows us to use Stokes' Theorem to  convert boundary integrals into integrals over the full space, thus taming the singularities of $U^*$ (see Section~\ref{ssec:h1}). As a result of this continuous analysis, the discretized scheme will have a convergence rate that is independent of the grid size.  The back-and-forth method is summarized in Algorithm~\ref{alg:otn_0}, where $H$ is the aforementioned weighted Sobolev space.

\begin{algorithm}[h] 
\SetAlgoLined
    Given $\rho^{(n)}$ and $\phi_0$, iterate: 
        \begin{equation*}
            \begin{split}
                \phi_{k+\frac{1}{2}} &= \phi_{k} +  \nabla_{\!H} J(\phi_{k})\\
                \psi_{k+\frac{1}{2}} &= (\phi_{k+\frac{1}{2}})^{c}\\
                \psi_{k+1} &= \psi_{k+\frac{1}{2}} +  \nabla_{\!H}I(\psi_{k+\frac{1}{2}})\\
                \phi_{k+1} &= (\psi_{k+1})^{\bar{c}}
            \end{split}
        \end{equation*}
 \caption{The back-and-forth scheme for solving~\eqref{eq:got_dual}}
 \label{alg:otn_0}
\end{algorithm}

Once we have solved the dual problem, we can recover the solution to the original problem (\ref{eq:variational-scheme}).  If $U$ is convex, then the optimal dual variable $\phi_*$ is related to $\rho^{(n+1)}$ through the duality relation $\rho^{(n+1)}=\delta U^*(\phi_*)$ (see Theorem~\ref{thm:dual-gap-zero} in Section~\ref{ssec:duality}). When $U$ is not convex, the connection between (\ref{eq:variational-scheme}) and the dual problem becomes more tenuous.  Luckily, we can circumvent this difficulty using a convexity splitting scheme~\cite{eyre}.  Indeed, if we write $U=U_1+U_0$ where $U_1$ is convex and $U_0$ is concave, then we can replace the JKO scheme (\ref{eq:variational-scheme}) with the modified scheme
\begin{equation}\label{eq:modified_scheme}
    \rho^{(n+1)}=\argmin_{\rho} U_1(\rho)+U_0(\rho^{(n)})+(\delta U_0(\rho^{(n)}),\rho-\rho^{(n)})+\frac{1}{2\tau}W_2^2(\rho,\rho^{(n)}).
\end{equation}
It is well-known that convexity splitting retains the energy stability of a fully implicit scheme. Crucially, the energy term $U_1(\rho)+U_0(\rho^{(n)})+(\delta U_0(\rho^{(n)}),\rho-\rho^{(n)})$ in (\ref{eq:modified_scheme}) is a convex function of the variable $\rho$, and thus, we can apply the duality approach.  All together, our method gives an extremely rapid way to simulate the PDE (\ref{eq:pde}) even when $U$ is non-convex or irregular.

The remainder of the paper is organized as follows. In Section~\ref{section:background}, we review important background information on optimal transport, convex analysis, and optimization. In Section~\ref{section:back-and-forth-method}, we present the back-and-forth algorithm and explain how to guarantee stability and choose step sizes. Lastly, in Section~\ref{section:experiments}, we demonstrate the accuracy, speed, and versatility of the algorithm through a wide suite of numerical experiments. In particular, our experiments include many cases that are well-known to be numerically challenging. 

\section{Background}\label{section:background}

In this section, we will rigorously establish the connection between the primal and dual formulations of the JKO scheme.  Furthermore, we will review key concepts from optimal transport and convex analysis that are needed to compute the gradients $\nabla_{\!H} J, \nabla_{\!H} I$ and establish stability of Algorithm~\ref{alg:otn_0}.  Note that throughout the paper we shall assume that $\Omega\subset \R^d$ is a bounded open set. 

\subsection{The \texorpdfstring{$c$}{c}-transform and optimal transport}\label{section:background-optimal-transport}
Throughout this section the space of continuous functions over $\Om$ will be denoted by $C(\Om)$.

\begin{definition} Given $\phi\in C(\Om)$ its {\it backward $c$-transform} is 
\begin{equation*}
\phi^{c}(x):=\inf_{y\in\Omega} \phi(y)+\frac{1}{2\tau}|x-y|^2.
\end{equation*}
Given $\psi\in C(\Om)$  its {\it forward $c$-transform} is 
\begin{equation*}
\psi^{\bar{c}}(y):=\sup_{x\in \Omega} \psi(x)-\frac{1}{2\tau}|x-y|^2.
\end{equation*}
\end{definition}

\begin{lemma}[\cite{otam}]\label{lem:ccc}
Given $\phi, \psi\in C(\Om)$, we have
\[
\phi^{c\bar{c}}\leq \phi, \quad \psi\leq \psi^{\bar{c}c},
\]
and
\[
\phi^{c\bar{c}c}=\phi^{c},\quad \psi^{\bar{c}c\bar{c}}=\psi^{\bar{c}}.
\]
\end{lemma}

\begin{definition}
Given $\phi, \psi\in C(\Omega)$, we say that $\phi$ is {\it $c$-convex} if $\phi^{c\bar{c}}=\phi$ and we say that $\psi$ is {\it $c$-concave} if $\psi^{\bar{c}c}=\psi$.  Furthermore, if $\phi^c=\psi$ and $\psi^{\bar{c}}=\phi$, then we say the pair $(\phi,\psi)$ is {\it$c$-conjugate}.
\end{definition}

The following two propositions establish the fundamental relationship between optimal transport and the $c$-transform.
\begin{prop}[\cite{gangbo1994elementary, gangbo1995quelques, gangbo_mccann}] \label{prop:c_transform_variation}
If $\phi\colon\Omega\to\R$ is $c$-convex and $\psi\colon\Omega\to\R$ is $c$-concave, then the maps
\begin{equation}\label{eq:T_map}
    T_{\phi}(x):=\argmin_{y\in\Omega} \phi(y)+\frac{1}{2\tau}|x-y|^2
\end{equation}
and
\begin{equation}\label{eq:S_map}
    S_{\psi}(y):=\argmax_{x\in\Omega} \psi(x)-\frac{1}{2\tau}|x-y|^2
\end{equation}
are well-defined and unique almost everywhere.
Furthermore, if $u\in C(\Omega)$, then for almost every $x,y\in\Omega$ we have the following perturbation formulas for the $c$-transform
\begin{equation}\label{eq:c_variation}
\lim_{t\to 0^+}  \frac{(\phi+tu)^c(x)-\phi^c(x)}{t}=u(T_{\phi}(x)),
\end{equation}
\begin{equation}\label{eq:c_bar_variation}
\lim_{t\to 0^+}  \frac{(\psi+tu)^{\bar{c}}(y)-\psi^{\bar{c}}(y)}{t}=u(S_{\psi}(y)).
\end{equation}
Finally, if $\phi$ and $\psi$ are $c$-conjugate, then 
\[
S_{\psi}(y)=y+\tau\nabla \phi(y),
\]
\[
T_{\phi}(x)=x-\tau\nabla \psi(x),
\]
and
$T_{\phi}\big( S_{\psi}(y)\big)=y$, $S_{\psi}\big( T_{\phi}(x)\big)=x$ almost everywhere.
\end{prop}

\begin{prop}[\cite{otam}]\label{lem:c_duality}
If $\mu, \nu\in L^1(\Omega)$ are nonnegative densities with the same mass, then
\[
 \frac{1}{2\tau}W_2^2(\mu,\nu)=\sup_{\phi\in C(\Omega)} \int_{\Omega} \phi^c(x) \,\mu(x)dx -\int_{\Omega} \phi(y)\,\nu(y) dy,
\]
\[
 \frac{1}{2\tau}W_2^2(\mu,\nu)=\sup_{\psi\in C(\Omega)} \int_{\Omega} \psi(x)\, \mu(x)dx -\int_{\Omega} \psi^{\bar{c}}(y)\,\nu(y)dy.
\]
\end{prop}

Now we can state the fundamental result that guarantees the existence and uniqueness of the optimal transport maps.
\begin{theorem}[\cite{brenier_polar,gangbo_habilitation,gangbo_mccann}] \label{thm:fund_ot}
If $\mu, \nu\in L^1(\Omega)$ are nonnegative densities with the same mass, then there exists a $c$-conjugate pair $(\phi_*, \psi_*)$ such that
\[
\phi_*\in\argmax_{\phi\in C(\Omega)} \int_{\Omega} \phi^c(x) \,\mu(x) dx-\int_{\Omega} \phi(y)\,\nu(y) dy,
\]
\[
\psi_*\in\argmax_{\psi\in C(\Omega)} \int_{\Omega} \psi(x) \,\mu(x) dx-\int_{\Omega} \psi^{\bar{c}}(y)\,\nu(y) dy,
\]
\[
\frac{1}{2\tau}W_2^2(\mu,\nu)=\int_{\Omega} \psi_*(x)\,\mu(x) dx-\int_{\Omega} \phi_*(y)\,\nu(y) dy,
\]
and $T_{\phi_*}, S_{\psi_*}$ are the unique optimal transport maps sending $\mu$ to $\nu$ and $\nu$ to $\mu$ respectively, i.e. 
$T_{\phi_*\,\#}\mu=\nu$ and $S_{\psi_*\,\#}\nu=\mu$.
\end{theorem}


\subsection{Convex duality} \label{ssec:duality}
Now that we have developed the basics of optimal transport, we are ready to return to the JKO scheme. 
To iterate the JKO scheme, one must be able to solve  generalized optimal transport (GOT) problems of the form
\begin{equation} \label{eq:got}
\rho_*=\argmin_{\rho\in L^1(\Om)} U(\rho) + \frac{1}{2\tau} W_2^2(\rho, \mu),
\end{equation}
where $\mu\in L^1(\Om)$ is a given nonnegative density.   Our method solves the GOT problem by appealing to its dual formulation.  In the rest of this subsection, we shall derive the dual problem and develop its basic properties.  To obtain a well-behaved dual problem, we shall need the following assumptions on the energy $U$. 
\begin{assumption} \label{assumption:U}
    The internal energy is given by a proper, convex, and lower semicontinuous functional $U\colon L^1(\Omega) \rightarrow \R \cup \{+\infty\}$ such that $U(\rho) = \infty$ if $\rho$ is negative on a set of positive measure.  
\end{assumption}

\begin{assumption} \label{assumption:U2}
    There exists a function $s\colon\R\to \R\cup\{+\infty\} $ with superlinear growth such that 
    \[
    U(\rho)\geq \int_{\Omega} s(\rho(y))\, dy.
    \]
\end{assumption}

\begin{remark}
Assumption~\ref{assumption:U} encodes the fact that $\rho$ must be a nonnegative density,  while Assumption~\ref{assumption:U2} guarantees that for each $B\in \R$ the sets $\{\rho\in L^1(\Omega): U(\rho)<B \}$ are weakly compact.  
\end{remark}
\begin{remark}
Except for the convexity requirement, Assumptions 1 and 2 are very natural in the context of Wasserstein gradient flows.  Note that we will eventually consider non-convex $U$ in Section~\ref{sec:nonconvex}.  
\end{remark}

At the heart of duality is the notion of convex conjugation.
\begin{definition}\label{def:legendre}
Given a functional $U\colon L^1(\Omega)\to\R$ its convex conjugate $U^*\colon L^{\infty}(\Omega)\to\R$ is defined by
\[
    U^*(\phi) := \sup_{\rho\in L^1(\Om)}\int_\Omega \phi(x) \rho(x)\,dx - U(\rho),
\]
\end{definition}

Thanks to Assumption 1, $U^*$ possess an important monotonicity property.

\begin{lemma}\label{lemma:monotonicity-Ustar}
$U^*$ is monotonically increasing, i.e. if $\phi_0, \phi_1\colon\Omega\to\R$ are functions such that $\phi_0\leq \phi_1$ pointwise everywhere, then
\[
U^*(\phi_0) \le U^*(\phi_1).
\]
\end{lemma}
\begin{proof}
By Assumption~\ref{assumption:U} the internal energy is finite only over nonnegative densities, thus, 
\[
    U^*(\phi) = \sup_{\rho \ge 0} \int_\Om \phi(x)\,\rho(x)dx - U(\rho). 
\]
If we take some $\rho\in L^1(\Om)$, with $\rho(x) \ge 0$ a.e., then we have
\[
    \int_\Om \phi_0(x)\,\rho(x)dx - U(\rho) \le \int_\Om \phi_1(x)\,\rho(x)dx - U(\rho). 
\]
Taking a supremum over $\rho\geq 0$ finishes the proof.
\end{proof}
Now we are ready to reintroduce the twin dual functionals $I$ and $J$.
\begin{prop}\label{prop:concavity}
Fix a nonnegative density $\mu\in L^1(\Om)$. The functionals $ I, J$ given by
\[
J(\phi):=\int_{\Omega} \phi^c(x)\,\mu(x) dx-U^*(\phi)
\]
\[
I(\psi):=\int_{\Omega} \psi(x)\,\mu(x) dx-U^*(\psi^{\bar{c}}),
\]
are proper, weakly upper semicontinuous, concave and $\sup_{\phi\in C(\Omega)} J(\phi)=\sup_{\psi\in C(\Omega)} I(\psi)$.  Furthermore, if $\phi$ is $c$-convex and $\psi$ is $c$-concave, then $J$ and $I$ have first variations
\[
\delta J(\phi)=T_{\phi\,\#}\mu-\delta U^*(\phi),
\]
\[
\delta I(\psi)=\mu-S_{\psi\,\#}\delta U^*(\psi^{\bar{c}}),
\]
where $\delta U^*$ is the first variation of $U^*$.
\end{prop}
\begin{proof}
Following the logic in the proof of Lemma~\ref{lemma:monotonicity-Ustar}, we may write 
\[
U^*(\psi^{\bar{c}})=\sup_{\rho\geq 0} \int_{\Omega} \psi^{\bar{c}}(y)\,\rho(y) dy-U(\rho).
\]
Next, let $\mathcal{M}(\Omega\times\Omega)$ denote the space of nonnegative measures on $\Omega\times\Omega$, and for any given density $\rho\geq 0$ define
\[
\Pi(\rho):=\left\{\pi\in\mathcal{M}(\Omega\times\Omega): \iint_{\Omega\times\Omega}f(y)\, d\pi(x,y) = \int_{\Omega}f(y)\,\rho(y) dy\; \textrm{for all} \; f\in C(\Omega)\right\}.
\]
Using the definition of the $c$-transform, we can then write
\[
\int_{\Omega}\psi^{\bar{c}}(y)\,\rho(y) dy=\sup_{\pi\in \Pi(\rho)} \iint_{\Omega\times \Omega} \big(\psi(x)-\frac{1}{2\tau}|x-y|^2\big)\,d\pi(x,y).
\]
Therefore, we have
\[
-U^*(\psi^{\bar{c}})=\inf_{\rho\geq 0}\inf_{\pi\in \Pi(\rho)} U(\rho)-\iint_{\Omega\times\Omega} \big(\psi(x)-\frac{1}{2\tau}|x-y|^2\big)\, d\pi(x,y).
\]
Now it is clear that $I$ can be written as the infimum over a family of linear functionals of $\psi$.  Hence, $I$ must be proper, concave and weakly upper semicontinuous.  An essentially identical argument applies to $J$. 

Since $U^*$ is monotonically increasing, Lemma~\ref{lemma:monotonicity-Ustar} implies that for any $\phi, \psi\in C(\Omega)$
\[
J(\phi)\leq I(\phi^{c}), \quad I(\psi)\leq J(\psi^{\bar{c}}).
\]
Therefore, we must have 
\[
\sup_{\psi\in C(\Omega)} I(\psi)=\sup_{\phi\in C(\Omega)} J(\phi).
\]
When $\phi$ and $\psi$ are $c$-convex/concave respectively, the formulas for the first variations follow directly from Proposition~\ref{prop:c_transform_variation}.
\end{proof}

Finally, we conclude this subsection by stating the essential result linking  the  primal and dual generalized optimal transport problems.  Crucially, this shows how to recover the solution to (\ref{eq:got}) from the maximizers of $I$ and $J$. 
\begin{theorem}[\cite{jacobskimtongL1}]\label{thm:dual-gap-zero} 
If $\mu\in L^1(\Omega)$, $U$ satisfies Assumptions 1, 2, and $\delta U(\mu)$ is not a constant function,  then there 
exists a unique density $\rho_*$ and a
pair of $c$-conjugate functions $(\phi_*, \psi_*)$ such that 
\[
\rho_*=\argmin_{\rho\in L^1(\Omega)} U(\rho)+\frac{1}{2\tau}W_2^2(\rho,\mu),\quad 
\phi_*\in \argmax_{\phi\in C(\Omega)} J(\phi),\quad \psi_*\in\argmax_{\psi\in C(\Omega)} I(\psi),
\]
\[
U(\rho_*)+\frac{1}{2\tau}W_2^2(\rho_*,\mu)=J(\phi_*)= I(\psi_*),
\]
\[
\rho_*\in \delta U^*(\phi_*), \quad \phi_*\in \delta U(\rho_*), \quad \rho_*=T_{\phi_*\,\#}\mu.
\]
\end{theorem}
\begin{remark}
Note that if $\delta U(\mu)$ is constant, then $\mu=\argmin_{\rho\in L^1(\Omega)} U(\rho)+\frac{1}{2\tau}W_2^2(\rho,\mu)$.  Thus, the excluded case is trivial. 
\end{remark}


\subsection{Concave gradient ascent}\label{subsection:gradient-ascent}
Now that we see how to link the JKO scheme to the dual functionals $I$ and $J$, it remains to develop a method to find the maximizers of $I$ and $J$.  To that end, 
in this subsection, we review classical unconstrained gradient ascent. Let us first recall the notion of \emph{gradient}. 
This will require the structure of a real Hilbert space $\mathcal{H}$ with inner product $\langle \cdot, \cdot \rangle_{\mathcal{H}}$ and norm $\norm{\cdot}_{\mathcal{H}}$.

\begin{definition}\label{def:first-variation}
 Given a point $\varphi\in\mathcal{H}$, we say that a bounded linear map $\delta F(\varphi)\colon\mathcal{H}\rightarrow\R$ is the first variation (Fr\'echet derivative) of $F$ at $\varphi$ if
\[
    \lim_{\|h\|_\mathcal{H}\to 0} \frac{\|F(\varphi + h) - F(\varphi) - \delta F(\varphi)(h)\|_\mathcal{H}}{\|h\|_{\mathcal{H}}} = 0.
\]
\end{definition}

\begin{definition}\label{def:H-gradient}
We say that a map $\nabla_\mathcal{H} F\colon\mathcal{H}\rightarrow\mathcal{H}$ is the $\mathcal{H}$-gradient of $F$ (or simply gradient if there is no ambiguity about the space $\calH$) if
\[
\langle \nabla_\mathcal{H} F(\varphi),h \rangle_\calH = \delta F(\varphi) (h)
\]
for all $(\varphi,h) \in \mathcal{H}\times \mathcal{H}$.
\end{definition}

The above identity highlights that  gradients are intimately linked to the 
inner product of the Hilbert space, in contrast to  first variations. Indeed, note that one can define the notion of a first variation over any normed vector space, while the notion of a gradient requires an inner product.

\subsubsection*{Gradient ascent method}
Given a concave functional $J$ over $\calH$, consider the gradient ascent iterations
\begin{equation}\label{eq:ga_standard}
    \phi_{k+1} = \phi_{k} + \nabla_{\mathcal{H}} J(\phi_{k}).
\end{equation}
The gradient ascent scheme (\ref{eq:ga_standard}) can equivalently be written in the variational form
\begin{equation}\label{eq:variational_ga}
\phi_{k+1}=\argmax_{\phi} J(\phi_k)+\delta J(\phi_k)(\phi-\phi_k)-\frac{1}{2}\norm{\phi-\phi_k}_{\calH}^2.
\end{equation}
Note that equations (\ref{eq:ga_standard}) and (\ref{eq:variational_ga}) typically include a step size parameter that controls how far one travels in the gradient direction.  For reasons that will become clear shortly (see equation (\ref{eq:def-metric}) and the subsequent discussion), we prefer to incorporate any parameters into the norm $\norm{\cdot}_{\calH}$ itself.

In order to obtain convergence of the scheme
\[
    J(\phi_k) \xrightarrow[k\to\infty]{} \sup_\phi J(\phi),
\]
with an efficient rate, it is essential to choose the norm $\norm{\cdot}_{\calH}$ properly. If the norm is too weak, then the algorithm may become unstable and fail to converge.  On the other hand, if the norm is too strong, then very little change happens at each step and the algorithm converges slowly.  The following theorem, one of the cornerstones of optimization,  explains how to balance these competing  considerations.

\begin{theorem}[\cite{nesterov2013introductory}] \label{thm:optim-gradient-ascent}
Let $J\colon \calH\to\Real$ be a twice Fréchet-differentiable concave functional with maximizer $\phi^*$.  If
\begin{equation} \label{eq:smoothinthm}
    -\delta^2\!J(\phi)(h,h) \le \norm{h}_{\calH}^2,
\end{equation}
for all $\phi,h\in \calH$ ($J$ is said to be ``$1$-smooth''), then the gradient ascent scheme
\[
    \phi_{k+1}=\phi_k+\nabla_{\!\calH}J(\phi_k)
\]
starting at a point $\phi_0$ satisfies the ascent property
\begin{equation}\label{eq:ascent_property}
    J(\phi_{k+1})\geq J(\phi_k)+\frac{1}{2}\norm{\nabla_{\calH}J(\phi_k)}^2_{\calH},
\end{equation}
and has the convergence rate
\begin{equation} \label{eq:ascentinthm}
J(\phi^*)-J(\phi_{k})\leq  \frac{\norm{\phi^*-\phi_0}^2_{\mathcal{H}}}{2k}.
\end{equation}

\end{theorem}

From Theorem~\ref{thm:optim-gradient-ascent}, we can again see the competing interests of weakening or strengthening the norm $\norm{\cdot}_{\calH}$.  A stronger norm makes it easier to satisfy equation (\ref{eq:smoothinthm}), while a weaker norm gives a better convergence rate in (\ref{eq:ascentinthm}).  Putting these considerations together, we see that it is optimal to choose the weakest possible norm such that (\ref{eq:smoothinthm}) holds.

\subsubsection*{Sobolev norm}
Let $\Om$ be an open bounded convex subset of $\Rd$. Our gradient ascent schemes use a norm $H$ based on the Sobolev space $H^1(\Om)$. For two constants $\Theta_1>0$ and $\Theta_2>0$ we define
\begin{equation} \label{eq:def-metric}
\norm{h}_H^2=\int_\Omega \Theta_2 \abs{\nabla h(x)}^2 + \Theta_1 \abs{h(x)}^2\,dx.    
\end{equation}
The precise value of $\Theta_1$ and $\Theta_2$ will depend on the functional being maximized (see for instance Theorem~\ref{thm:J-smoothness} in Section~\ref{section:back-and-forth-method}).  In many instances, it will be optimal to choose $\Theta_1$ and $\Theta_2$ to have rather different values.  For this reason, we do not wish to reduce these parameters to a single step size value. 
The next lemma describes how to compute gradients with respect to this inner product.

\begin{lemma}\label{lemma:frechet}
Suppose that $F=F(\phi)$ is a Fr\'echet-differentiable functional such that for any $\phi$ the first variation $\delta F(\phi)$ evaluated at any point $h$ can be written as integration against a function $f_\phi$, i.e.
\[
    \delta F(\phi)(h) = \int_\Omega h(x) f_\phi(x) \,dx.
\]
Define $\norm{\cdot}_H$ by~\eqref{eq:def-metric}. Then the $H$-gradient of $F$ can be written
\[
    \nabla_{\!H} F(\phi) = (\Theta_1 \id - \Theta_2 \Delta)^{-1} f_\phi,
\]
where $\id$ is the identity operator and $\Delta$ is the Laplacian operator, taken together with zero Neumann boundary conditions.
\end{lemma}
\begin{proof}
Fix $\phi$ and consider the unique solution to the elliptic equation
\[
\begin{cases}
(\Theta_1 \id - \Theta_2 \Delta) g &= f_\phi \quad\text{in }\Om,\\
n\cdot\nabla g &= 0 \quad\text{on }\partial\Om.
\end{cases}
\]
Then we have the chain of equalities
\begin{align*}
    \delta F(\phi)(h) &= \int_\Omega h(x) f_\phi(x) \,dx\\
    &= \int_\Omega h(x) (\Theta_1 \id - \Theta_2 \Delta) g(x) \,dx\\
    &= \int_\Omega \Theta_1 h(x)g(x)+\Theta_2\nabla h(x)\cdot\nabla g(x)\,dx\\
    &= \bracket{h,g}_H.
\end{align*}
This shows that $g$ is the $H$-gradient of $F$.
\end{proof}

The above result can be restated as follows: the $H$-gradient of $F$ is obtained by ``preconditioning'' $\delta F$ with the inverse operator $(\Theta_1 \id - \Theta_2 \Delta)^{-1}$.


\section{The back-and-forth method}\label{section:back-and-forth-method}

Our goal is to develop an efficient algorithm for solving the JKO scheme
for a large class of interesting energies $U$. We begin in Section~\ref{ssec:bfm} with the case where $U$ is convex with respect to $\rho$.  In this case, the JKO scheme has an equivalent dual problem that we solve using an adaptation of the back-and-forth method from~\cite{jacobslegerbf}.   In Section~\ref{ssec:h1}, we show that the algorithm is gradient stable in a properly weighted $H^1$ space for convex energies of the form
\[
U(\rho)=\int_{\Omega} u_m(\rho(x))+V(x)\rho(x)\, dx,
\]
where $V\colon\Omega\to[0,+\infty]$ is a fixed function, and
\begin{equation}\label{eq:u_m}
u_m(\rho)=
\begin{cases}
\frac{\gamma}{m-1}(\rho^m-\rho) &\textrm{if} \; \rho\geq 0,\\
+\infty &\textrm{otherwise},\\
\end{cases}
\end{equation}
for some constants $\gamma>0$ and $m>1$.
We shall also consider the two limiting cases $m\to 1$ and $m\to\infty$.    Let us note that our analysis can be extended to more general functionals, however, we focus on the (important) special case above for clarity of exposition. 
After we have developed the method for convex energy functionals $U$, in Section~\ref{sec:nonconvex} we show how to generalize the algorithm for non-convex  $U$.

\subsection{The back-and-forth method for convex \texorpdfstring{$U$}{U}}\label{ssec:bfm}
To iterate the JKO scheme, we must be able to solve the generalized optimal transport (GOT) problem 
\begin{equation} \label{eq:got_main}
\rho_*=\argmin_{\rho\in L^1(\Om)} U(\rho) + \frac{1}{2\tau} W_2^2(\rho, \mu),
\end{equation}
for any fixed nonnegative density $\mu\in L^1(\Omega)$.
As we saw in Section~\ref{section:background} (see Theorem~\ref{thm:dual-gap-zero}), when $U$ is convex, the generalized optimal transport problem is in duality with 
the twin functionals $I$ and $J$, i.e.
\[
\inf_{\rho\in L^1(\Om)} U(\rho) + \frac{1}{2\tau} W_2^2(\rho, \mu) = \sup_{\phi} J(\phi)=\sup_{\psi} I(\psi).
\]
 Recall $I$ and $J$ are given by
\begin{equation}\label{eq:J_main}
J(\phi)=\int_{\Omega} \phi^c(x) \,\mu(x) dx-U^*(\phi),
\end{equation}
\begin{equation}\label{eq:I_main}
I(\psi)=\int_{\Omega} \psi(x) \,\mu(x)dx-U^*(\psi^{\bar{c}}).
\end{equation}
Furthermore, the minimizer $\rho_*$ of problem (\ref{eq:got_main}) is related to the maximizers $\phi_*, \psi_*$ through the relations
\begin{equation}\label{eq:duality_relations}
\rho_*=T_{\phi_*\,\#}\mu, \quad \rho_*\in \delta U^* (\phi_*), \quad \phi_*^c=\psi_*.
\end{equation}

Both $I$ and $J$ are unconstrained concave functionals (see Proposition~\ref{prop:concavity}), therefore, it is now clear that one can find the maximizer of either functional via standard gradient ascent methods. On the other hand, choosing to work with solely $I$ or solely $J$ breaks the symmetry of the problem. Thus, rather than focusing on only one of the functionals, the back-and-forth method performs alternating gradient ascent steps on $I$ and $J$.  Although $I$ and $J$ use different variables, we can switch between $\phi$ and $\psi$ by using the $c$-transform.  As noted in~\cite{jacobslegerbf}, the alternating steps on $I$ and $J$ substantially accelerate the convergence rate of the method beyond standard gradient ascent. 

We are now ready to introduce our approach to find the twin dual maximizers $(\phi_*, \psi_*)$ to problem (\ref{eq:got_main}). The method is outlined in Algorithm~\ref{algo:got} and is based on two main ideas:
\begin{enumerate}[1.]
\item 
A back-and-forth update scheme, alternating between gradient ascent steps on $I$ and $J$.
\item 
Gradient ascent steps in an $H^1$-type norm $H$, with
\begin{equation*}
\begin{aligned}
\nabla_{\!H} J(\phi) &= (\Theta_1\id - \Theta_2\laplacian)^{-1} \Big[T_{\phi\,\#}\mu - \delta U^*(\phi)\Big],\\
\nabla_{\!H} I(\psi) &= (\Theta_1\id - \Theta_2\laplacian)^{-1} \Big[\mu - S_{\psi\,\#}(\delta U^*(\psi^{\bar c}))\Big].
\end{aligned}
\end{equation*}
\end{enumerate}

\begin{algorithm}[h] 
\SetAlgoLined
    Given $\mu$ and $\phi_0$, iterate: 
        \begin{equation*}
            \begin{split}
                \phi_{k+\frac{1}{2}} &= \phi_{k} +  \nabla_{\!H} J(\phi_{k})\\
                \psi_{k+\frac{1}{2}} &= (\phi_{k+\frac{1}{2}})^{c}\\
                \psi_{k+1} &= \psi_{k+\frac{1}{2}} +  \nabla_{\!H}I(\psi_{k+\frac{1}{2}})\\
                \phi_{k+1} &= (\psi_{k+1})^{\bar{c}}
            \end{split}
        \end{equation*}
 \caption{The back-and-forth scheme for solving~\eqref{eq:J_main} and \eqref{eq:I_main}}
 \label{algo:got}
\end{algorithm}

Our ultimate goal is to show that each step of Algorithm~\ref{algo:got} increases the value of the functionals $J$ and $I$.  Thanks to Lemmas~\ref{lem:ccc} and~\ref{lemma:monotonicity-Ustar} it is easy to check that 
\[
J(\phi_{k+\frac{1}{2}})\leq I((\phi_{k+\frac{1}{2}})^c), \quad I(\psi_{k+1})\leq J((\psi_{k+1})^{\bar{c}}).
\]
Thus, we see that the alternating steps where we switch between the $\phi$ and $\psi$ variables can only increase the values of the dual problems.  To show that the gradient steps $\phi_{k+\frac{1}{2}} = \phi_{k} +  \nabla_{\!H} J(\phi_{k})$ and $\psi_{k+1} = \psi_{k+\frac{1}{2}} +  \nabla_{\!H}I(\psi_{k+\frac{1}{2}})$ increase the values of $J$ and $I$ respectively requires a more detailed analysis, which will be the main focus of Section~\ref{ssec:h1}.  As we shall see, the enhanced stability provided by the $H^1$ preconditioner $(\Theta_1\id-\Theta_2\Delta)^{-1}$ will be essential to ensure that the gradient steps have the ascent property.

Once the dual problems $I$ and $J$ have been solved to sufficient accuracy, one can recover the optimal density $\rho_*$ in (\ref{eq:got_main}) through the duality relations in (\ref{eq:duality_relations}).  In certain examples, such as incompressible flows, the subdifferential $\delta U^*$ may be multivalued.  When this happens, the relation $\rho_*\in \delta U^*(\phi_*)$ does not uniquely define $\rho_*$.   However, in practice, $\delta U^*$ is typically only multivalued on a single level set of $\phi_*$ which has zero measure.   As a result, for numerical purposes, we can simply identify $\rho_*=\delta U^*(\phi_*)$.   Note that it is advantageous to recover $\rho_*$ in this way as opposed to the pushforward relation $\rho_*=T_{\phi_*\,\#}\mu$.  Indeed, the formula $\rho_*=T_{\phi_*\,\#}\mu$ requires the computation of numerical derivatives of $\phi_*$, while the duality relation $\rho_*\in \delta U^*(\phi_*)$ is derivative free.  

Combining our work, we obtain an algorithm for evolving the JKO scheme. 
\begin{algorithm}[h]
\SetAlgoLined
    Given initial data $\rho^{(0)}$, initialize $\phi^{(0)}=\delta U(\rho^{(0)})$.\\
    \For{ $n=0,\dots,N$}{
    $\phi^{(n+1)}\gets $Run Algorithm~\ref{algo:got}  with $\mu=\rho^{(n)}$ and $\phi_0=\phi^{(n)}$.\\
    $\rho^{(n+1)} = \delta U^*(\phi^{(n+1)})$.
    }
 \caption{Running the JKO scheme}
  \label{algo:jko}
\end{algorithm}

\subsection{\texorpdfstring{$H^1$}{H1} gradient ascent}\label{ssec:h1}

In order to ensure stability of the gradient ascent steps, the gradients of $I$ and $J$ are computed in a metric based on the $H^1$ Sobolev norm. Given two constants $\Theta_1>0$, $\Theta_2>0$, we define the Hilbert norm $H$ by
\begin{equation} \label{eq:defH}
\norm{h}_{H}^2 = \int_\Om \Theta_2\abs{\nabla h(x)}^2 + \Theta_1\abs{h(x)}^2\,dx.
\end{equation}
The main steps of the back-and-forth scheme are the gradient ascent steps in the in the $H$ norm
\[
\phi_{k+\frac{1}{2}} = \phi_{k} +  \nabla_{\!H} J(\phi_{k})
\]
and
\[
\psi_{k+1} = \psi_{k+\frac{1}{2}} +  \nabla_{\!H} I(\psi_{k+\frac{1}{2}}).
\]
In order to obtain convergence of our method, we want these steps to increase the values of the concave functionals $J$ and $I$ respectively. The so-called gradient ascent property
\begin{align*}
    J(\phi_{k+\frac 1 2}) - J(\phi_k) & \ge \frac 1 2\norm{\nabla_{\!H}J(\phi_k)}_H^2,\\
    I(\psi_{k+1}) - I(\psi_{k+\frac 1 2}) & \ge \frac 1 2\norm{\nabla_{\!H} I(\psi_{k+\frac 1 2})}_H^2,
\end{align*}
can be obtained when the Hessian bounds
\begin{equation}\label{eq:1-smooth} 
\begin{aligned}
-\delta^2\!J(\phi)(h,h) \le \norm{h}^2_{H}, \\
-\delta^2\!I(\psi)(h,h) \le \norm{h}^2_{H}
\end{aligned}
\end{equation}
are  satisfied (c.f. Theorem~\ref{thm:optim-gradient-ascent} in Section~\ref{subsection:gradient-ascent}). When~\eqref{eq:1-smooth} holds, $I$ and $J$ are said to be ``$1$-smooth'' with respect to $H$.

We shall devote the rest of this subsection to obtaining inequalities of the form~\eqref{eq:1-smooth}. Specifically, we shall show how to choose the constants $\Theta_1$ and $\Theta_2$ in equation (\ref{eq:defH}) to ensure that $I$ and $J$ are 1-smooth (under regularity assumptions on $\phi$ and $\psi$) when $U$ has the form
\begin{equation} \label{eq:convex-U}
U(\rho)=\int_\Om u_m(\rho(x))\,dx + \int_\Om V(x)\rho(x)\,dx,
\end{equation}
where $u_m$ is defined in (\ref{eq:u_m}) and $V\colon\Omega\to[0,+\infty]$ is some given function.  

Crucially, we will give upper bounds on $\Theta_1$ and $\Theta_2$ that can be efficiently computed from the data.  
Obtaining tight bounds for $\Theta_1$ and $\Theta_2$ is important as they essentially control the step size of the algorithm (note that small values of $\Theta_1$ and $\Theta_2$ correspond to large gradient steps).  As we explained in Section~\ref{subsection:gradient-ascent}, it is optimal to choose the smallest values of $\Theta_1$ and $\Theta_2$ such that (\ref{eq:1-smooth}) holds.  This analysis is actually practical, as our numerical experiments confirm that the convergence of BFM can be substantially accelerated by making good choices for $\Theta_1$ and $\Theta_2$.

Those who are interested in the analysis of these bounds can continue reading this section, otherwise, one can immediately jump to the statements of Theorems~\ref{thm:J-smoothness} and~\ref{thm:I-smoothness},  which give approximately optimal values of $\Theta_1$ and $\Theta_2$ for the functionals $I$ and $J$.

\subsubsection{Hessian bound analysis}
It turns out that the Hessian bound analysis is nearly identical for $I$ and $J$. Therefore, we will primarily focus on the analysis for $J$, and we will later explain  how to deal with $I$ in a similar fashion. 
To obtain Hessian bounds on $J(\phi)=\int_\Om\phi^c\mu - U^*(\phi)$, we first derive bounds on the $c$-transform term 
\begin{equation} \label{eq:defF}
F(\phi):=\int_\Om\phi^c(x)\,\mu(x)dx,
\end{equation}
and then on the internal energy term $U^*(\phi)$. 
Let us begin by providing an expression for $\delta^2 F(\phi)$, the Hessian of $F$ at a point $\phi$ that is $c$-convex.

\begin{lemma}[Hessian bounds on the $c$-transform] \label{lemma:hessian-F}
Let $F$ be the functional defined in \eqref{eq:defF}. If $\phi$ is a $c$-convex function, then 
the Hessian of $F$ at $\phi$ can be written as
\begin{equation*}\delta^2F(\phi)(h,h) = -\tau\, \int_\Om \nabla h(y)\cdot \cof(I_{d\times d}+\tau D^2\phi(y))\nabla h(y)\, \mu(y+\tau\nabla\phi(y))\,dy,\end{equation*}
where $\cof(I_{d\times d}+\tau D^2\phi(y))$ denotes the cofactor matrix of $I_{d\times d}+\tau D^2\phi(y)$.
Furthermore, if the eigenvalues of $I_{d\times d}+\tau D^2\phi(y)$ are bounded above by some constant $\Lambda$ for every $y\in\Omega$, then we have the bound
\begin{equation} \label{eq:upperboundF}
-\delta^2F(\phi)(h,h) \leq \tau \normlinf{\mu} \Lambda^{d-1} \normltwo{\nabla h}^2.
\end{equation}

\end{lemma}

The proof of Lemma~\ref{lemma:hessian-F} can be found in the appendix. To gain some insight into the bound~\eqref{eq:upperboundF}, note that given a positive definite symmetric matrix $M\in \R^{d\times d}$ with eigenvalues $\{\lambda_1,\ldots, \lambda_d\}$, the eigenvalues of $\cof(M)$ are $\{\frac{\det(M)}{\lambda_1},\ldots, \frac{\det(M)}{\lambda_d}\}$.  This produces the $d-1$ degree scaling of $\Lambda^{d-1}$.  To understand the meaning of $\Lambda$ itself better,   recall that the optimal primal variable $\rho_*$ is given by $T_{\phi_*\,\#}\mu=\mu(y+\tau\nabla \phi_*(y))\det(I_{d\times d}+\tau D^2\phi(y)).$  Hence, the eigenvalues of $I_{d\times d}+\tau D^2\phi$ roughly measure how concentrated the mass of $\rho_*$ is compared to $\mu$.  Since one expects the difference between $\rho_*$ and $\mu$ to be on the order of $\tau$, it is reasonable to expect that $\Lambda$ will be close to 1.

We now turn our attention to bounding the Hessian of the internal energy term $U^*(\phi)$. 
When $U$ takes the form~\eqref{eq:convex-U}, its convex conjugate can be written as
\begin{equation*}
U^*(\phi)=\int_\Om u^*_m(\phi(x)-V(x))\,dx,
\end{equation*}
where 
\[
u^*_m(p)=\gamma^{-\frac{1}{m-1}}\Big(\frac{(m-1)p+\gamma}{m}\Big)_+^{\frac{m}{m-1}}
\]
and $(\cdot)_+=\max(\cdot, 0)$.
Now it is clear that the Hessian of $U^*$ is given by
\begin{equation} \label{eq:Hessian-Ustar}
    \delta^2U^*(\phi)(h,h)=\int_\Om (u_m^*)''\big(\phi(x)-V(x)\big)\abs{h(x)}^2\,dx.    
\end{equation}

When $1\leq m\leq 2$, the bounds are straightforward as $(u_m^*)''(p)$ is increasing with respect to $p$.  Hence, in this case, we have
\[
\delta^2 U^*(\phi)(h,h)=\int_{\Omega} (u_m^*)''(\phi(x)-V(x))|h(x)|^2\, dx\leq B\norm{h}_{L^2(\Omega)}^2,
\]
where $B=\sup_{x\in\Omega} (u_m^*)''(\phi(x)-V(x))$.  It was shown in~\cite{jacobskimtongL1} that the maximizer $\phi_*$ of $J$ obeys a maximum type principle in the sense that
\[
\phi_*(x) \le M:=\sup_{x\in\Om} \delta U(\mu)(x).
\]
It is therefore natural to assume that $\phi$ will be bounded above by $M$ throughout the algorithm (the gradient steps tend to diffuse pressure in the regions of highest concentration). Assuming $V(x)\geq 0$ everywhere, it now follows that 
\[
\delta^2 U^*(\phi)(h,h)\leq (u_m^*)''(M)\norm{h}_{L^2(\Omega)}^2.
\]

The aforementioned maximum principle on the pressure, $\phi(x)\le M$, can be used again to write the upper bound in terms of density instead of pressure. Indeed note that
\[
\rho(x)=(u_m^*)'(\phi(x)-V(x))\le (u_m^*)'(\phi(x)) \le (u_m^*)'(M).
\]
Therefore the quantity
\begin{equation} \label{eq:def-rhomax}
    \rhomax := (u_m^*)'(M)
\end{equation}
acts a natural upper bound on the densities. Furthermore writing $(u_m^*)''(M)=(u_m^*)''\big(u_m'(\rhomax)\big)=u_m''(\rhomax)^{-1}$, we obtain
\[
\delta^2 U^*(\phi)(h,h)\leq u_m''(\rhomax)^{-1}\norm{h}_{L^2(\Omega)}^2.
\]

The case $m>2$ is substantially more complicated.
When $m>2$,  $(u^*_m)''$ is singular at zero.  Hence, the integrand may be unbounded near points where $\phi(x)=V(x)$.  In this case, it may not be possible to bound (\ref{eq:Hessian-Ustar}) in terms of the $L^2$ norm of $h$.  To understand this better, let us focus on the most difficult model we consider in this paper: the incompressible limit $m\to\infty$. When $m\to\infty$, the energy $u_m$ encodes a hard ceiling constraint on the density values, i.e. 
\[
u_{\infty}(\rho)=
\begin{cases}
0& \textrm{if} \quad 0\leq \rho\leq 1,\\
+\infty &\textrm{otherwise}.
\end{cases}
\]
Hence, the dual energy $u^*_{\infty}$ is given by
\[
u^*_{\infty}(p)=
\begin{cases}
0\quad\text{if } p<0\\
p\quad\text{if } p\ge 0.\\
\end{cases}
\]
We pause here to point out that $u^*_{\infty}$ has much better regularity than $u_{\infty}$, for instance  $u^*_{\infty}$ is continuous over $\R$ while $u_{\infty}$ is discontinuous at $0$ and $1$. This illustrates once more the advantage of working with dual quantities. Nevertheless, $u^*_{\infty}$ is clearly not smooth in the convex sense, as there is a jump of derivative at $0$. In fact, we have $(u^*_{\infty})''=\delta_0$, where $\delta_0$ denotes the Dirac delta function at $0$. 

Luckily,  even though $U^*$ is built from $u^*_{\infty}$ which is not smooth, it is possible to bound the Hessian of $U^*$ as long as the singularity only occurs on a small set. Indeed, if we make the assumption that $|\nabla\phi(x)-\nabla V(x)|$ stays away from zero on the surface $\{\phi=V\}$, i.e. there exists a constant $\Gamma_0>0$ such that
\[
\sup_{\{x\in\Omega: \phi(x)=V(x)\}}\frac{1}{|\nabla \phi(x)-\nabla V(x)|}\leq \Gamma_0
\]
(note this is a quantitative way of saying that $\{\phi=V\}$ is a lower dimensional set), then we can use the coarea formula to rewrite equation~\eqref{eq:Hessian-Ustar} as
\begin{equation}\label{eq:coarea_bound}
\begin{aligned}
\delta^2U^*(\phi)(h,h) &=\int_{\R} (u_{\infty}^*)''(\alpha)\int_{\{x\in \Omega: \phi(x)-V(x)=\alpha\}}\frac{|h(x)|^2}{|\nabla \phi(x)-\nabla V(x)|}\, ds(x)\,d\alpha\\
&=\int_{\{\phi=V\}} \frac{\abs{h(x)}^2}{\abs{\nabla \phi(x)-\nabla V(x)}}\,ds(x) \\ &\leq \Gamma_0\int_{\set{\phi=V}} \abs{h(x)}^2\, ds(x),
\end{aligned}
\end{equation}
where $ds$ is the usual surface measure.  Due to the fact that the integration occurs over a surface, we cannot bound the right hand side of (\ref{eq:coarea_bound}) in terms of $\norm{h}_{L^2}$. However, we can use \emph{trace inequalities} from PDE theory to bound surface integrals by volume integrals involving a higher derivative~\cite{evans_book} (this can be essentially viewed as an inequality version of Stokes' Theorem). More precisely, there exist constants $C_1, C_2$ depending on the surface $\{\phi=V\}$, but independent of $h$ such that
\begin{equation*} 
    \int_{\{\phi=V\}} \abs{h(x)}^2\,ds(x) \le  C_2\norm{\nabla h}_{L^2(\Omega)}^2+C_1\norm{h}_{L^2(\Omega)}^2.   
\end{equation*}
From there we can immediately deduce that $U^*$ is $H$-smooth, since
\begin{equation*}
    \Gamma_0\int_{\{\phi=V\}} \abs{h(x)}^2\,ds(x)\leq \norm{h}_{H}^2    
\end{equation*}
as long as we choose $\Theta_i \geq C_i\Gamma_0$, $i=1,2$.

Now that we have seen how to obtain Hessian bounds in the most singular case $m\to\infty$, we are ready to return to the case $2<m<\infty$.  Note that in this case, $(u_m^*)''(p)$ is zero if $p<0$, singular at zero, and decreasing for $p> 0$.  Hence, if we choose some value $\lambda>0$ and let \[
A_{\lambda}=\{x\in\Omega: 0\leq \phi(x)-V(x)\leq \lambda\},
\]then we immediately have the bound
\[
\delta^2 U^*(\phi)(h,h)\leq (u_m^*)''(\lambda)\norm{h}_{L^2(\Omega)}^2+\int_{A_{\lambda}} (u_m^*)''(\phi(x)-V(x))|h(x)|^2\, dx.
\]
To estimate the second term, we proceed along the same lines as the case $m=\infty$. For any $\alpha\in\R$ let $\{\phi-V=\alpha\}=\set{x\in\Omega: \phi(x)-V(x)=\alpha}$.
As long as we have a constant $\Gamma_{\lambda}$ and trace inequality constants $C_1(\alpha), C_2(\alpha)$ such that
\begin{equation}\label{eq:def-Gamma}
\sup_{x\in A_{\lambda}} \frac{1}{ \abs{\nabla \phi(x)-\nabla V(x)}}\leq \Gamma_{\lambda} 
\end{equation}
and 
\begin{equation} \label{eq:trace_constants}
     \int_{\{\phi-V=\alpha\}} |h(x)|^2\, d s(x)\leq C_2(\alpha)\norm{\nabla h}_{L^2(\Omega)}^2  +C_1(\alpha)\norm{h}_{L^2(\Omega)}^2, 
\end{equation}
then we can replicate the argument from above.
 Combining the coarea formula and trace inequality, we get the following string of inequalities
\[
\int_{A_{\lambda}} (u_m^*)''(\phi(x)-V(x))|h(x)|^2\, dx
\]
\[
\leq \Gamma_{\lambda}\int_0^{\lambda} (u_m^*)''(\alpha)\int_{\{\phi-V=\alpha\}} |h(x)|^2\, ds(x)\, d\alpha
\]

\[
\leq(u^*_m)'(\lambda)\Gamma_{\lambda}\Big( C_{2,\lambda}\norm{\nabla h}_{L^2(\Omega)}^2  +C_{1,\lambda}\norm{h}_{L^2(\Omega)}^2  \Big),
\]
where 
\begin{equation}\label{eq:c_i_def}
C_{i,\lambda}=\max_{0\leq\alpha\leq \lambda} C_i(\alpha).
\end{equation}
Thus, $-\delta^2 U^*(h,h)$ is bounded by $\norm{h}_H^2$ as long as we choose 
\[
\Theta_1\geq (u^*_m)''(\lambda)+(u^*_m)'(\lambda)\,\Gamma_{\lambda}\,C_{1,\lambda}
\]
and 
\[
\Theta_2\geq (u^*_m)'(\lambda)\,\Gamma_{\lambda}\,C_{2,\lambda}
\]
where we have the freedom to choose the precise value of $\lambda.$

Our above computations are now summarized in the following lemma.
\begin{lemma}[Hessian bound on the internal energy] \label{lemma:hessian-Ustar}
    Define $\rhomax$, $\Gamma_{\lambda}$ and $C_{i,\lambda}$ by~\eqref{eq:def-rhomax},~\eqref{eq:def-Gamma} and~\eqref{eq:c_i_def}. 
    \begin{enumerate}[1.]
        \item Case $1\le m\le 2$. We have
        \[
            \delta^2U^*(\phi)(h,h) \le \frac{1}{\gamma m}(\rhomax)^{2-m} \normltwo{h}^2.
        \]
        
        \item Case $2<m<\infty$. For any $\lambda>0$,
        \begin{multline*}
            \delta^2U^*(\phi)(h,h) \le (\gamma m')^{1-m'} C_{2,\lambda}\,\Gamma_{\lambda}\,\normltwo{\nabla h}^2 + \\
            (\gamma m')^{1-m'} \Big(C_{1,\lambda}\,\Gamma_{\lambda}\lambda^{m'-1}+(m'-1)\lambda^{m'-2}\Big) \normltwo{h}^2,
        \end{multline*}
        where $m'=\frac{m}{m-1}$.
        
        \item Case $m=\infty$. We have
        \[
            \delta^2U^*(\phi)(h,h) \le C_{2,0}\,\Gamma_0\normltwo{\nabla h}^2+C_{1,0}\,\Gamma_0\normltwo{ h}^2.
        \]
    \end{enumerate}
    
\end{lemma}

Combining Lemma~\ref{lemma:hessian-F} and~\ref{lemma:hessian-Ustar}  we directly obtain the main theorem of this section.

\begin{theorem}[$1$-smoothness of $J$] \label{thm:J-smoothness}
Let $1\le m\le \infty$ and $U(\rho)=\int_{\Om} u_m(\rho(x))+V(x)\rho(x)\,dx$, where $u_m$ is defined by~\eqref{eq:u_m}. Then $J(\phi):=\int_\Om \phi^c(x)\,\mu(x)dx-U^*(\phi)$ satisfies the Hessian bound
\begin{equation*}
-\delta^2\!J(\phi)(h,h) \le \Theta_2\normltwo{\nabla h}^2 + \Theta_1 \normltwo{h}^2,
\end{equation*}
where $\Theta_1$ and $\Theta_2>0$ are given by the table below. 

As in Lemma~\ref{lemma:hessian-F}, $\Lambda$ is an upper bound on the eigenvalues of $I_{d\times d}+\tau D^2\phi(y)$ uniformly in $y$. Additionally $\lambda>0$ is a parameter to choose and $\rhomax$, $\Gamma_{\lambda}$ and $C_{i,\lambda}$ are defined by~\eqref{eq:def-rhomax},~\eqref{eq:def-Gamma} and~\eqref{eq:c_i_def}.

\begin{center}
\renewcommand{\arraystretch}{1}
\begin{tabular}{|c|c|c|}
\hline
$m$ & $\Theta_1$ & $\Theta_2$ \\
\hline
\rule{0pt}{4ex}
$m=1$ & $\displaystyle\frac{\rhomax}{\gamma}$ & $\tau \Lambda^{d-1}\normlinf{\mu}$ \\[2ex]
\hline
\rule{0pt}{4ex}
$1<m<2$ & $\displaystyle\frac{\rhomax^{\ \ \ \ 2-m}}{\gamma m}$ & $\tau \Lambda^{d-1}\normlinf{\mu}$ \\[2ex]
\hline
\rule{0pt}{4ex}
$m=2$ & $\displaystyle\frac{1}{2\gamma}$ & $\tau \Lambda^{d-1}\normlinf{\mu}$\\[2ex]
\hline
\rule{0pt}{4ex}
$m>2$& $\displaystyle(\gamma m')^{1-m'} \Big(\lambda^{m'-1}C_{1,\lambda}\,\Gamma_{\lambda}+\frac{m'-1}{\lambda^{2-m'}}\Big)$ & $(\gamma m')^{1-m'} C_{2,\lambda}\,\Gamma_{\lambda} +\tau  \Lambda^{d-1}\normlinf{\mu}$ \\[2ex]
\hline
\rule{0pt}{4ex}
$m=\infty$ & $C_{1,0}\,\Gamma_{0}$ & $C_{2,0}\,\Gamma_{0} + \tau \Lambda^{d-1}\normlinf{\mu}$\\[2ex]
\hline
\end{tabular}
\end{center}

\end{theorem}

In order to use Theorem~\ref{thm:J-smoothness} in the case $m>2$, we need to be able to compute $\Gamma_{\lambda}$ and $C_{i,\lambda}$ and we need to choose a value for $\lambda$ when $m\in (2,\infty)$.  On a discrete grid with $n$ points, one can easily compute $\Gamma_{\lambda}$ for all $\lambda$ in $O(n)$ operations.  On the other hand, it requires $O(n)$ operations to compute $C_1(\alpha)$ and $C_2(\alpha)$ for a single value of $\alpha$ (c.f. Section 4.1).  Thus, for the case $m=\infty$, we can compute the constants explicitly in $O(n)$ operations.  The case $2<m<\infty$ is harder, since we cannot efficiently compute $C_{i,\lambda}=\max_{0\leq\alpha\leq\lambda}C_{i}(\alpha)$.  To overcome this difficulty, we typically choose $\lambda$ by minimizing 
\[
\lambda^*=\argmin_{\lambda\geq 0}\;\displaystyle(\gamma m')^{1-m'} \Big(\lambda^{m'-1}\Gamma_{\lambda}C_{1}(0)\,+\frac{m'-1}{\lambda^{2-m'}}\Big),
\]
which gives a reasonable estimate for the optimal choice of $\lambda$ to make $\Theta_1$ as small as possible.
We then estimate $\max_{0\leq \alpha\leq\lambda^*}C_i(\alpha)$ by simply taking the max over $C_i(0)$ and $C_i(\lambda^*)$, which appears to work well in practice. 


To conclude this discussion we turn our attention to the other functional $I$ for which a similar analysis can be made. First we define
\[
p(x) = (\psi^{\bar c}-V)(T_{\psi^{\bar c}}(x)).
\]
Next, for $\lambda>0$, we define
\begin{equation} \label{eq:def-tildeGamma}
    \tilde{\Gamma}_\lambda = \sup_{x : 0 \le p(x)\le \lambda} \frac{1}{\abs{\nabla p(x)}}.
\end{equation}
Finally, we define trace constants $\tilde C_{i}(\alpha)$ such that
\[
    \int_{\{p=\alpha\}} \abs{h(x)}^2\,ds(x) \le \tilde C_2(\alpha) \normltwo{\nabla h}^2 + \tilde C_1(\alpha)\normltwo{h}^2,
\]
and then set
\begin{equation} \label{eq:def-tildectr}
\tilde C_{i,\lambda}=\sup_{0\leq\alpha\leq \lambda} \tilde{C}_i(\alpha).
\end{equation}
Now we can state our result bounding the Hessian of $I$.
\begin{theorem}\label{thm:I-smoothness}
Let $I(\psi)=\int_\Om \psi(x)\,\mu(x)dx-U^*(\psi^{\bar c})$, with $U(\rho)=\int_\Om u_m(\rho(x))+V(x)\rho(x)\,dx$, $u_m$ is defined by~\eqref{eq:u_m} and $1\le m\le \infty$. The Hessian of $I$ can be written
\begin{multline*}
    -\delta^2 I(\psi)(h,h) = \delta^2U^*(\psi^c)(h\circ S_\psi,h\circ S_\psi) + \\
    \tau \int_\Om \nabla h(x)\cdot\cof(I_{d\times d}-\tau D^2\psi(x))\nabla h(x)\,\delta U^*(\psi^c)(x-\tau\nabla\psi(x))\,dx.
\end{multline*}
It satisfies the bound
\begin{equation*}
-\delta^2\!I(\psi)(h,h) \le \Theta_2\normltwo{\nabla h}^2 + \Theta_1 \normltwo{h}^2,
\end{equation*}
where $\Theta_1$ and $\Theta_2>0$ are given by the table below. Here $\Lambda$ is an upper bound on the eigenvalues of $I_{d\times d}-\tau D^2\psi(x)$ uniformly in $x$. 
Additionally  $\lambda>0$ is a parameter to choose and $\rhomax$ is defined by~\eqref{eq:def-rhomax}, $\tilde\Gamma_{\lambda}$ by~\eqref{eq:def-tildeGamma} and $\tilde C_{i,\lambda}$ by~\eqref{eq:def-tildectr}.

\begin{center}
\renewcommand{\arraystretch}{1}
\begin{tabular}{|c|c|c|}
\hline
$m$ & $\Theta_1$ & $\Theta_2$ \\
\hline
\rule{0pt}{4ex}
$m=1$ & $\displaystyle\frac{\Lambda^d\rhomax}{\gamma}$ & $\tau \Lambda^{d-1}\rhomax$ \\[2ex]
\hline
\rule{0pt}{4ex}
$1<m<2$ & $\displaystyle\frac{\Lambda^d(\rhomax)^{2-m}}{\gamma m }$ & $\tau \Lambda^{d-1}\rhomax$ \\[2ex]
\hline
\rule{0pt}{4ex}
$m=2$ & $\displaystyle\frac{\Lambda^d}{2\gamma}$ & $\tau \Lambda^{d-1}\rhomax$\\[2ex]
\hline
\rule{0pt}{4ex}
$m>2$& $\displaystyle \Lambda^d(\gamma m')^{1-m'} \Big(\tilde C_{1,\lambda}\,\tilde \Gamma_{\lambda}\lambda^{m'-1}+\frac{m'-1}{\lambda^{2-m'}}\Big)$ & $\Lambda^d(\gamma m')^{1-m'} \tilde C_{2,\lambda}\,\tilde \Gamma_{\lambda} +\tau  \Lambda^{d-1}\rhomax$ \\[2ex]
\hline
\rule{0pt}{4ex}
$m=\infty$ & $\Lambda^d \tilde C_{1,0}\,\tilde\Gamma_0$ & $\Lambda^d \tilde C_{2,0}\,\tilde\Gamma_0+\tau \Lambda^{d-1}\rhomax $\\[2ex]
\hline
\end{tabular}
\end{center}
\end{theorem}

\subsection{Back-and-forth for non-convex \texorpdfstring{$U$}{U}}\label{sec:nonconvex}

In this section, we will discuss how to extend our method when $U$ is not convex with respect to $\rho$.  The trick is to appeal to convexity splitting~\cite{eyre}, a well-known technique for simulating gradient flows with non-convex energies.    The idea behind convexity splitting is to write $U$ as a sum of a convex function and a concave function, i.e.
\[
U(\rho)=U_1(\rho)+U_0(\rho),
\]
where $U_1$ is convex and $U_0$ is concave.  Thanks to the concavity of $U_0$, given any fixed density $\bar{\rho}$, we have the inequality
\begin{equation}\label{eq:splitting_inequality}
U(\rho)\leq U_1(\rho)+U_0(\bar{\rho})+(\delta U_0(\bar{\rho}), \rho-\bar{\rho}).
\end{equation}
 Crucially, the right-hand-side of equation (\ref{eq:splitting_inequality}) is a convex function.     As such, if we replace the JKO scheme with the relaxed scheme
\begin{equation}\label{eq:relaxed_GOT}
\rho^{(n+1)}=\argmin_{\rho}\; U_1(\rho)+U_0(\rho^{(n)})+(\delta U_0(\rho^{(n)}), \rho-\rho^{(n)})+\frac{1}{2\tau}W_2^2(\rho, \rho^{(n)}),
\end{equation}
then we obtain a convex variational problem. 
The beauty of convexity splitting is that the relaxed scheme is still unconditionally energy stable. Combining (\ref{eq:splitting_inequality}) and (\ref{eq:relaxed_GOT}) we have the string of inequalities
\[
U(\rho^{(n+1)})+\frac{1}{2\tau}W_2^2(\rho^{(n+1)}, \rho^{(n)})\leq
\]
\[
U_1(\rho^{(n+1)})+U_0(\rho^{(n)})+(\delta U_0(\rho^{(n)}), \rho^{(n+1)}-\rho^{(n)})+\frac{1}{2\tau}W_2^2(\rho^{(n+1)}, \rho^{(n)})\leq 
\]
\[
\inf_{\rho}\; U_1(\rho)+U_0(\rho^{(n)})+(\delta U_0(\rho^{(n)}), \rho-\rho^{(n)})+\frac{1}{2\tau}W_2^2(\rho, \rho^{(n)}).
\]
By choosing $\rho=\rho^{(n)}$ in the last line, we can conclude that
\[
U(\rho^{(n+1)})+\frac{1}{2\tau}W_2^2(\rho^{(n+1)}, \rho^{(n)})\leq U(\rho^{(n)}).
\]
Thus, we see that the energy is still decreasing along the iterates of the relaxed scheme. 

 Now let us turn to solving the relaxed problem (\ref{eq:relaxed_GOT}).  Since the energy term in
 (\ref{eq:relaxed_GOT}) is convex, we can solve the problem using the dual approach outlined above. 
The twin dual problems associated to (\ref{eq:relaxed_GOT}), which we shall denote as $\tilde{J}$ and $\tilde{I}$, are given by 
\begin{equation}\label{eq:split_J}
    \tilde{J}(\phi):=\int_{\Omega} \phi^c(x)\,\rho^{(n)}(x)dx - \tilde{U}^*(\phi),
\end{equation}
\begin{equation}\label{eq:split_I}
    \tilde{I}(\psi):=\int_{\Omega}\psi(x)\,\rho^{(n)}(x)dx-\tilde{U}^*(\psi^{\bar{c}}),
\end{equation}
where 
\[
\tilde{U}^*(\phi):=U_1^*\big(\phi-\delta U_0(\rho^{(n)})\big)+(\delta U_0(\rho^{(n)}), \rho^{(n)})-U_0(\rho^{(n)})
\]
is the convex conjugate of $U_1(\rho)+U_0(\rho^{(n)})+(\delta U_0(\rho^{(n)}), \rho-\rho^{(n)})$.  We can then find the dual maximizers $(\phi^{(n+1)}, \psi^{(n+1)})$ of (\ref{eq:split_J}) and (\ref{eq:split_I}) using Algorithm~\ref{algo:got} along with the Hessian bounds developed in the previous subsection.   As before, one can recover the solution $\rho^{(n+1)}$ of (\ref{eq:relaxed_GOT}) through the duality relation $\rho^{(n+1)}=\delta \tilde{U}^*(\phi^{(n+1)})$.

\section{Numerical implementation and experiments}\label{section:experiments}

\subsection{Implementation details}
In this section, we use the back-and-forth method to numerically simulate equation (\ref{eq:pde})  for a wide variety of internal energies $U$.  Throughout this section we will assume that the domain $\Omega = [-1/2,1/2]^2$ is the unit square in $\mathbb{R}^2$, discretized using a regular rectangular grid.  The numerical simulations in this section were coded in C++ and were run on 2019 MacBook Pro with $2.6$ GHz 6-core and $16$ GB RAM.

\quad Following the approach in~\cite{jacobslegerbf}, we will compute the forward and backward $c$-transforms using the fast Legendre transform (FLT) algorithm~\cite{DBLP:journals/na/Lucet97}.  On a regular rectangular grid with $n$ points, the FLT algorithm can be used to compute either the forward or backward $c$-transform in $O(n)$ operations. See~\cite{jacobslegerbf} for more detail on the equivalence of the $c$-transform and the Legendre transform. 

When computing gradients with respect to the weighted norm (\ref{eq:def-metric}), we will need to solve a Poisson equation with zero Neumann boundary condition.  We will solve this equation numerically via the fast Fourier transform (FFT).  All FFTs were calculated using the free FFTW C++ library.

To compute the gradients of $I$ and $J$, we will also need to compute pushforwards.  Given a density $\mu$ and an invertible map $Z\colon\Omega\to\Omega$ we can compute the pushforward $Z_{\#}\mu$ via the Jacobian formula
\[
Z_{\#}\mu(x)=\frac{\mu\big( Z^{-1}(x)\big)}{|\det\big(DZ( Z^{-1}(x))\big)|}=\mu\big( Z^{-1}(x)\big)|\det\big(D(Z^{-1})(x)\big)|.
\]
In our case, we will only need to compute pushforwards with respect to the maps $T_{\phi}$ and $S_{\psi}$ that are induced by the forward and backward $c$-transforms respectively.  Thanks to the structure of BFM, we only need to compute $T_{\phi\,\#}\rho^{(n)}$ and $S_{\psi\,\#}\delta U^*(\psi^{\bar{c}})$ when $\phi$ and $\psi$ are $c$-convex and $c$-concave respectively.  As a result, we have the simple formulas $T_{\phi}^{-1}(y)=y+\tau\nabla \phi(y)$ and $S_{\psi}^{-1}(x)=x-\tau\nabla \psi(x)$.    Therefore, 
\[
    T_{\phi\,\#}\rho^{(n)}(y) =  \rho^{(n)}\big(y+\tau\nabla\phi(y)\big) \det\big(I_{d\times d}+\tau D^2\phi(y) \big),
\]
and
\[
    S_{ \psi \# } \delta U^*(\psi^{\bar{c}}) (x)=\Big( \delta U^*\big(\psi^{\bar{c}}\big)\circ \big(x-\tau\nabla\psi(x)\big)\Big) \det \big(I_{d\times d}-\tau D^2\psi(x)\big).
\]
When implementing our algorithm, we compute these quantities using a simple centered difference scheme. 

Finally, let us briefly explain how to compute the trace inequality constants $C_i(\alpha)$ defined in equation (\ref{eq:trace_constants}).  From Lemma~\ref{lem:trace_dist} and Corollary~\ref{cor:trace} in the Appendix, we see that $C_{i}(\alpha)$ can be computed from the solution $u$ to the Eikonal equation
\[
\begin{cases}
|\nabla u(x)|=1 &\textrm{if}\quad \phi(x)-V(x)\neq \alpha,\\
u(x)>0 &\textrm{if} \quad \phi(x)-V(x)<\alpha,\\
u(x)<0 &\textrm{if}\quad \phi(x)-V(x)>\alpha.\\
\end{cases}
\]
Note that 
\[
|u(x)|^2=\min_{\{y: \phi(y)-V(y)=\alpha\} } |x-y|^2,
\]
which is nothing but a $c$-transform of the indicator function 
\[
\chi_{\alpha}(y)=
\begin{cases}
0 & \textrm{if} \quad \phi(y)-V(y)=\alpha,\\
+\infty &\textrm{else}.\\
\end{cases}
\]
Therefore, $|u|^2$ can be computed in $O(n)$ operations using the Fast Legendre transform, and from there one can recover $u$.  Once one has $u$, it is straightforward to compute the constants in Corollary~\ref{cor:trace} in $O(n)$ operations.
\subsection{Experiments}
\quad We present four sets of numerical experiments. In the first set of experiments, we demonstrate the speed and accuracy of our method by comparing to the so-called Barrenblat solutions, a special case of equation (\ref{eq:pde}) where  closed-form solutions are available.
In the next set of experiments, we simulate the porous media equation $\partial_t \rho=\Delta (\rho^m)+\nabla \cdot (\rho \nabla V)$ for various interesting functions $V\colon\Omega\to\R\cup\{+\infty\}$ and values of $m$.   Note that if $V$ takes the value $+\infty$ on some closed set $E\subset \Omega$, then $\rho$ can never enter $E$.  Hence, this is equivalent to solving (\ref{eq:pde}) on the more complicated domain $\Omega\setminus E$.  
In the third set of experiments, we use the splitting scheme from Section~\ref{sec:nonconvex} to simulate (\ref{eq:pde}) when $U$ is nonconvex.   In this case, the non-convexity will come from an interaction energy of the form $\mathcal{W}(\rho)=\int_{\Omega}\int_{\Omega} W(x-y)\rho(x)\rho(y)\, dy\, dx $. 
Finally, in the last set of experiments, we study incompressible flows where $U$ encodes the hard constraint $\rho\leq 1$ everywhere.  In this case, the dual energy $U^*$ will have a very singular Hessian at the boundary of the support of $\rho$. Nonetheless, we are still able to simulate the evolution even on very fine grids.      

\subsubsection{Accuracy: Barenblatt solutions}
In this experiment, we use our back-and-forth algorithm to solve the PME, 
\begin{equation}
\partial_t \rho = \gamma\Delta(\rho^m),
\end{equation}
with the initial data
\[
    \rho(0,x) = M \delta_{0}(x).
\]
Here, $\gamma>0$ is a constant that controls the speed of the diffusion, $M>0$ is the total initial mass and $\delta_{0}$ is the standard Dirac distribution centered at zero.  When $m>1$, this equation is the Wasserstein gradient flow  of the energy $U(\rho)=\int_{\Omega} \frac{\gamma}{m-1}\rho(x)^m\, dx$.
Thanks to the simplicity of the initial data, on the domain $\R^2$ the equation has a closed form solution, known as the Barenblatt solution~\cite{barenblatt1996, barenblatt2003},
\begin{equation}\label{eq:barenblatt}
\rho(t,x) = \Biggl(
\left(\frac{M}{4\pi m t \gamma}\right)^{\frac{m-1}{m}}
- \frac{(m-1)}{4m^2 t \gamma} |x|^2
\Biggr)_+^{\frac{1}{m-1}},
\end{equation}
where $(\cdot)_+=\max(\cdot,0)$.
The Barenblatt solution is compactly supported, therefore, it agrees with the solution on the square $[-1/2, 1/2]^2$ up until the time $t_c=\frac{m-1}{16m^2\gamma}(\frac{\pi(m-1)  }{4mM})^{m-1} $ when the mass hits the boundary of the square.

Using the Barenblatt solution as a benchmark, we can test the accuracy and efficiency of our scheme.  We will simulate the equation for the exponents $m=2, 4, 6$.   Since the Dirac delta function is challenging to work with numerically, we shall instead fix a height $h_0>0$ and start the flow at a time $t_0>0$, where $t_0$ is chosen so that $\norm{\rho(t_0,\cdot)}_{L^{\infty}}=h_0$.   Note that the value of $t_0$ will depend on the exponent $m$, and can be found explicitly from equation (\ref{eq:barenblatt}).  In addition, we will only consider the flow within in the time interval $[t_0,t_c]$, since the Barenblatt solution is only valid on the unit square up to time $t_c$.

In all of our benchmark experiments, we shall set $M=0.5$,  $h_0=15$ and $\gamma=10^{-3}$.  Note that the small value of $\gamma$ is just a time rescaling to ensure that the flow occurs on a macroscopic time interval.    We will compute the evolution between the times $t_0 \leq t \leq 2+t_0$ with different step sizes $\tau = 0.4, 0.2, 0.1, 0.05, 0.025$  (one can check that with our parameter choices $t_0+2< t_c$ for $m=2,4,6$).  Running the experiments with various time step sizes allows us to verify that the scheme becomes more accurate as the time step is decreased.  We shall measure the accuracy of the solution using the $L^1$ norm, which is very natural in the context of Wasserstein gradient flows (see for instance~\cite{jacobskimtongL1}). The precise formula for our error estimate is
\begin{equation}\label{eq:l1_error}
\mathrm{error} = \frac{1}{N_{\tau}} \sum_{n=0}^{N_{\tau}} \int_\Om \abs{\rho(n\tau+t_0,x)-\rho^{(n)}(x)}\,dx,
\end{equation}
where $N_{\tau}=\lfloor\frac{2}{\tau}\rfloor$,  $\rho(n\tau+t_0,x)$ is the Barenblatt solution and $\rho^{(n)}$ is the $n^{th}$ JKO iterate starting from the initial data $\rho^{(0)}(x)=\rho(t_0,x)$.  When solving for $\phi^{(n+1)}$, we will run Algorithm~\ref{algo:got} until the residual $\norm{T_{\phi}-\delta U^*(\phi)}_{L^1(\Omega)}$ is less than $\epsilon= 10^{-3}$.

The results of these experiments are displayed in Table~\ref{table:barenblatt} and Figure~\ref{fig:comparison-m-246}.   Table~\ref{table:barenblatt} displays the error (\ref{eq:l1_error}) and the total computation time for all of the aforementioned experiments. In Figure ~\ref{fig:comparison-m-246}, we plot a cross section of our solutions and the exact solution at various time snapshots.  The cross section is taken along the horizontal line $\{(x_1, 0): x_1\in [-1/2,1/2]\}$.   One can see that as the time step is decreased, our solution is in excellent agreement with the exact solution for all exponents $m=2,4,6$.  Figure~\ref{fig:comparison-m-246} also shows that our method correctly captures the discontinuity of $\nabla \rho$ at the boundary of the support of $\rho$.  This is notable as most other numerical methods smooth out the discontinuity.  The reason that we are able to correctly capture the discontinuity is due to the fact that we recover the density through the duality relation $\rho^{(n+1)}=\delta U^*(\phi^{(n+1)})=\Big(\frac{m-1}{m\gamma}\max(\phi,0)\Big)^{\frac{1}{m-1}}$.  The function $s(x)=\max(x,0)^{\frac{1}{m-1}}$ has discontinuous derivatives at zero, therefore even when $\phi^{(n+1)}$ is smooth, $\nabla \rho$ will still have a discontinuity at the boundary of its support.

\begin{table}
\center{\caption{Barenblatt solution test case (grid size $512 \times 512$) \label{table:barenblatt}}}
 \begin{tabular}{||c c c c c c c c||} 
 \hline
 \multirow{2}{*}{$\tau$} & \multirow{2}{*}{$N_{\tau}$} &  \multicolumn{2}{ c }{$m=2$} &  \multicolumn{2}{ c }{$m=4$} &  \multicolumn{2}{ c ||}{$m=6$} \\ [0.5ex] 
 &  & Error & Time (s) & Error & Time (s) & Error & Time (s) \\ [0.5ex] 
 \hline\hline
 $0.4$ & $5$ & $6.35\times 10^{-2}$ & $14.54$ & $1.19\times 10^{-1}$ & $23.11$ & $1.13\times 10^{-1}$ & $22.02$\\
 \hline
 $0.2$ & $10$ & $3.72\times 10^{-2}$ & $22.16$ & $7.95\times 10^{-2}$ & $30.34$ & $7.48 \times 10^{-1}$ & $30.41$\\
 \hline
 $0.1$ & $20$ & $2.08\times 10^{-2}$ & $36.57$ & $5.03\times 10^{-2}$ & $48.41$ & $4.74 \times 10^{-2}$ & $43.95$\\
 \hline
 $0.05$ & $40$ & $1.18\times 10^{-2}$ & $55.64$ & $3.06\times 10^{-2}$ & $77.03$  & $2.90\times 10^{-2}$ & $80.10$\\
 \hline
 $0.025$ & $80$ & $8.26\times 10^{-3}$ & $77.67$ & $1.89\times 10^{-2}$ & $140.38$ & $1.79\times 10^{-2}$ & $164.89$\\
 \hline
\end{tabular}
\end{table}
\vspace{0.2cm}

\begin{figure}[h]
      \centering
      \includegraphics[width=1.0\textwidth]{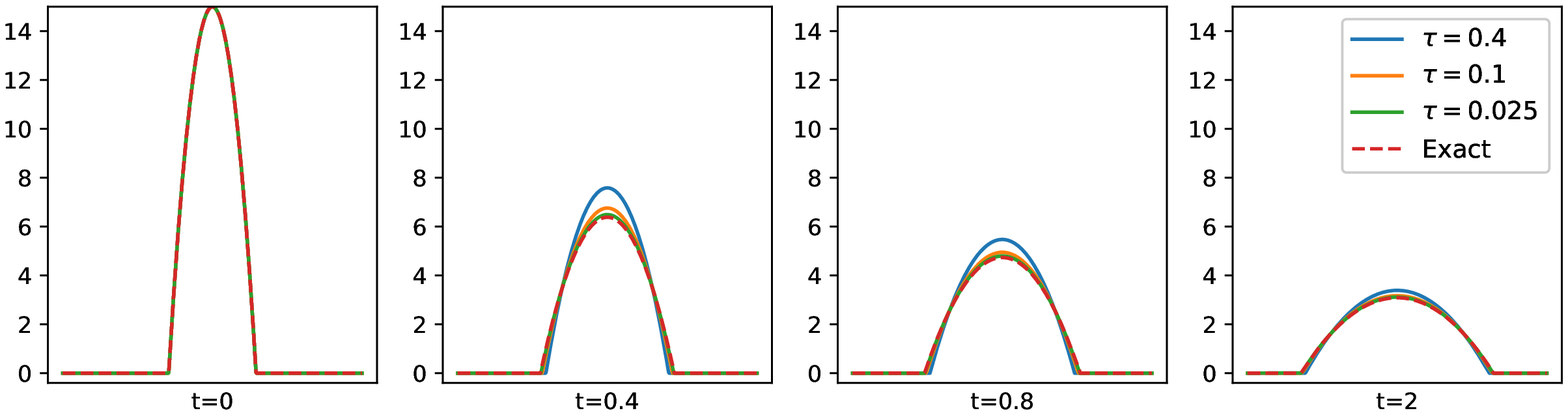}\hfil
      \includegraphics[width=1.0\textwidth]{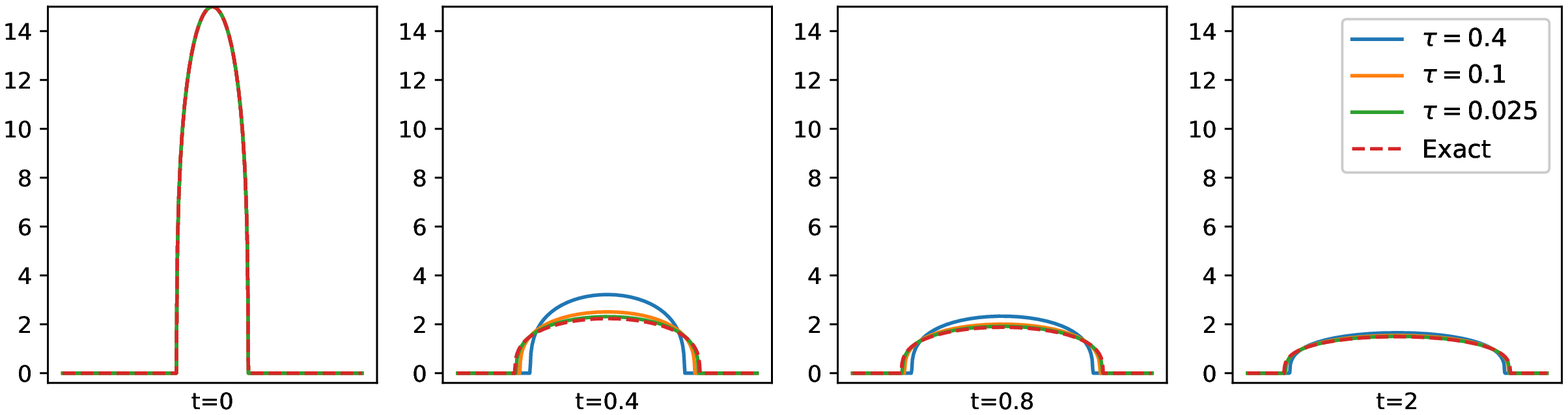}\hfil
      \includegraphics[width=1.0\textwidth]{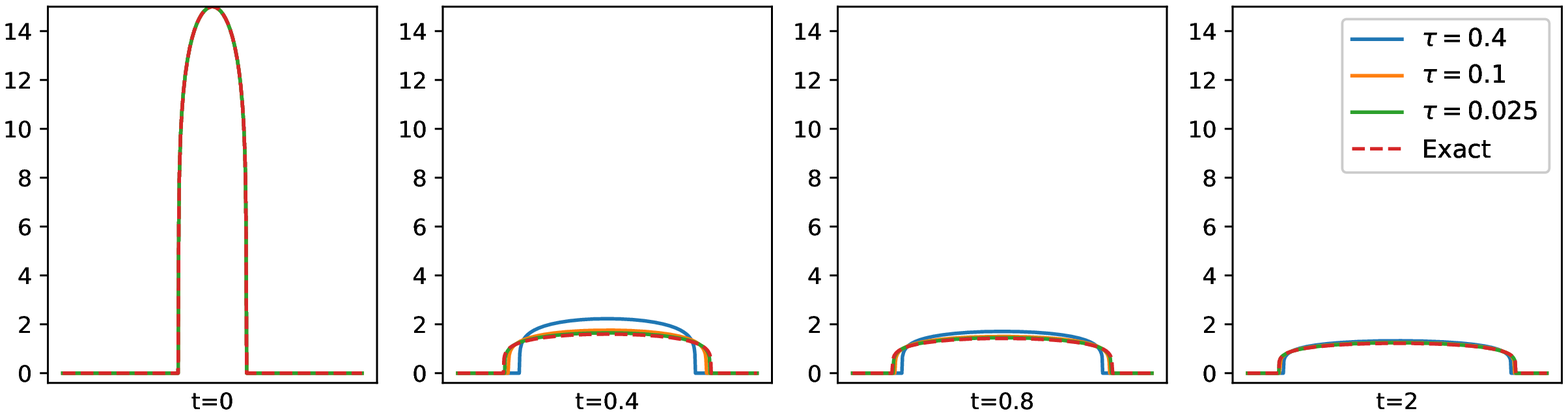}\hfil
    \caption{\footnotesize Cross sections of our computed solutions and the exact Barenblatt solution at times $t=t_0,t_0+0.4,t_0+0.8, t_0+2$ along the horizontal line $\{(x_1,0):x_1\in [-1/2,1/2]\}$. Row~1: $m=2$, Row 2: $m=4$, Row 3: $m=6$.}
    \label{fig:comparison-m-246}
\end{figure}

\subsubsection{Slow diffusion with drifts and obstacles.}
In our next set of experiments, we add spatially varying potentials to the energy functional. The resulting equations are a type of drift-diffusion equations. The energy takes the specific form
\[
    U(\rho) =  \int_\Omega \frac{\gamma}{m-1} \rho^m(x) + V(x) \rho(x) dx,
\]
where $V$ is a given function. 

In the first set of experiments, we consider an example where the initial density is the characteristic function of a star shaped region normalized to have mass $1$, and we use the fixed potential function
\begin{equation}\label{eq:sine}
    V_1(x) =  1 - \sin(5\pi x_1) \sin(3\pi x_2).
\end{equation}
The initial data and the potential $V_1$ are shown in  Figure~\ref{fig:slow-initial-sine-potential}.

\begin{figure}[h]
    \begin{subfigure}[b]{0.3\linewidth}
      \centering
      \includegraphics[width=.99\linewidth]{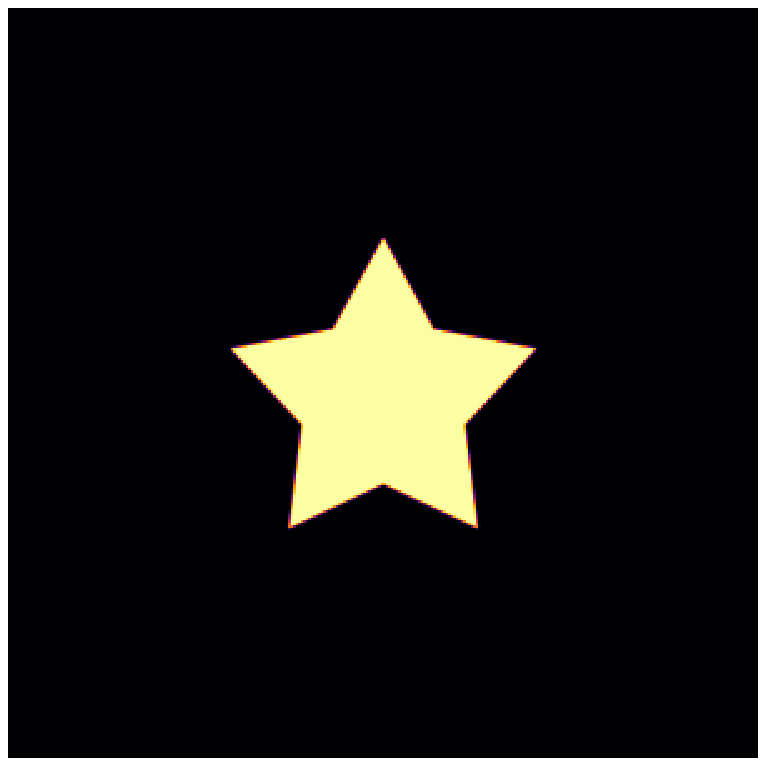}
      \caption*{\footnotesize Initial density}
    \end{subfigure}
    \begin{subfigure}[b]{0.3\linewidth}
      \centering
      \includegraphics[width=.99\linewidth]{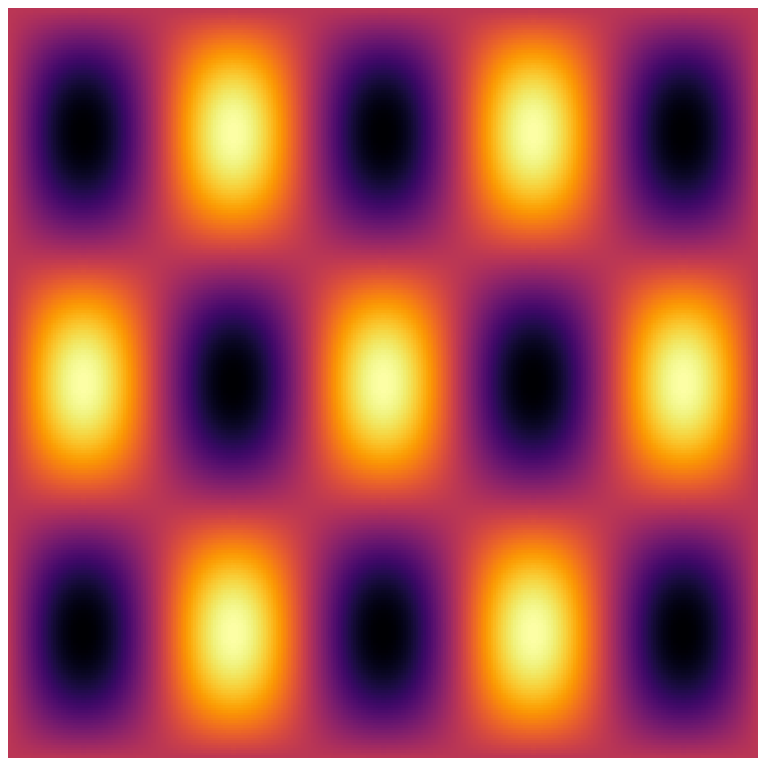}
      \caption*{\footnotesize The potential $V_1$}
    \end{subfigure}\hfil
    \caption{\footnotesize Higher values are depicted with brighter pixels.}
    \label{fig:slow-initial-sine-potential}
\end{figure}

Using this setup, we run two different experiments, one where $m=2$ and another where $m=4$. In both cases, we set $\gamma=0.1$ and use the time step $\tau=0.001$.  We run the equations until we reach a state that is essentially stationary.    The flow for $m=2$ is run from time $t=0$ to time $t=5$, and the flow for $m=4$ is run from time $t=0$ to time $t=2$.       The flow for the $m=2$ case is shown in Figure~\ref{fig:slow-diffusion-sine-example-m-2} and the $m=4$ case is shown in Figure~\ref{fig:slow-diffusion-sine-example-m-4}. The solutions show the density is drawn to regions where the potential is small, while avoiding concentration due to the $\rho^m$ term.  Notice that the steady state for $m=4$ is much more diffuse than the steady state for $m=2$, this is because $\rho^4$ penalizes concentration much more than $\rho^2$.

\begin{figure}[h]
    \begin{minipage}[b]{0.2\linewidth}
      \centering
      \includegraphics[width=.99\linewidth]{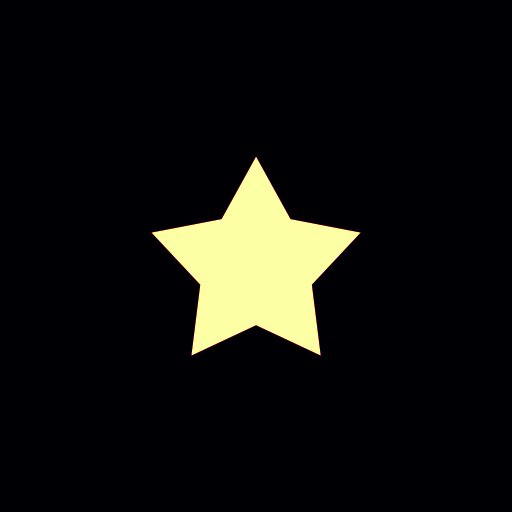}
    \end{minipage}\hfil
    \begin{minipage}[b]{0.2\linewidth}
      \centering
      \includegraphics[width=.99\linewidth]{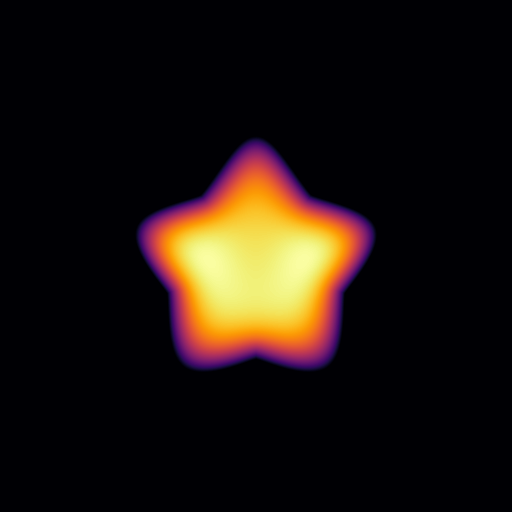}
    \end{minipage}\hfil
    \begin{minipage}[b]{0.2\linewidth}
      \centering
      \includegraphics[width=.99\linewidth]{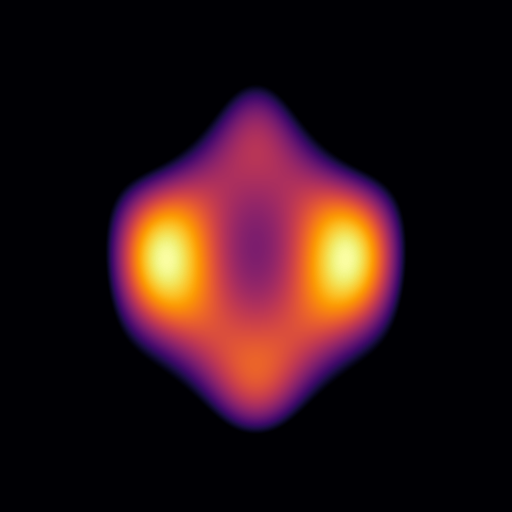}
    \end{minipage}\hfil
    \begin{minipage}[b]{0.2\linewidth}
      \centering
      \includegraphics[width=.99\linewidth]{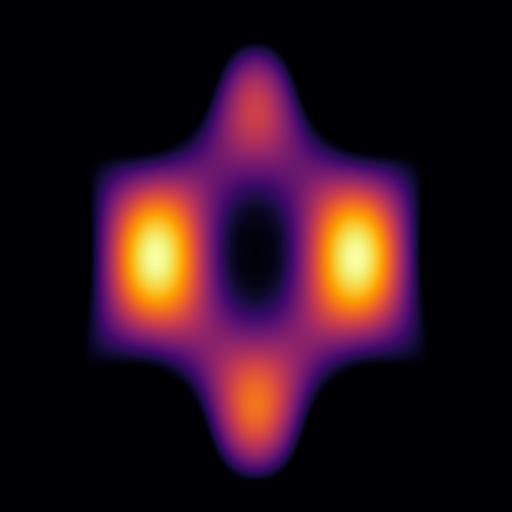}
    \end{minipage}\hfil
    \begin{minipage}[b]{0.2\linewidth}
      \centering
      \includegraphics[width=.99\linewidth]{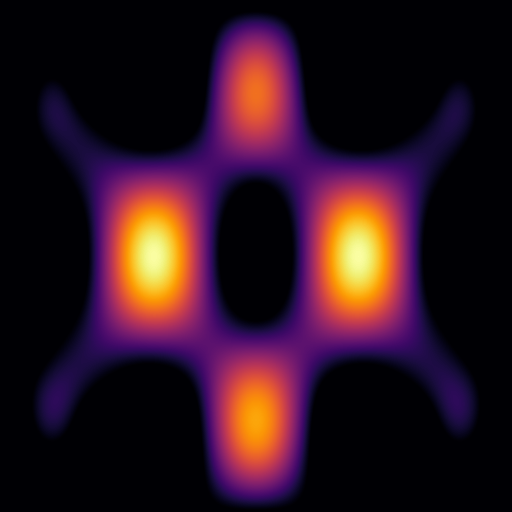}
    \end{minipage}\hfil
    \medskip
    \begin{minipage}[b]{0.2\linewidth}
      \centering
      \includegraphics[width=.99\linewidth]{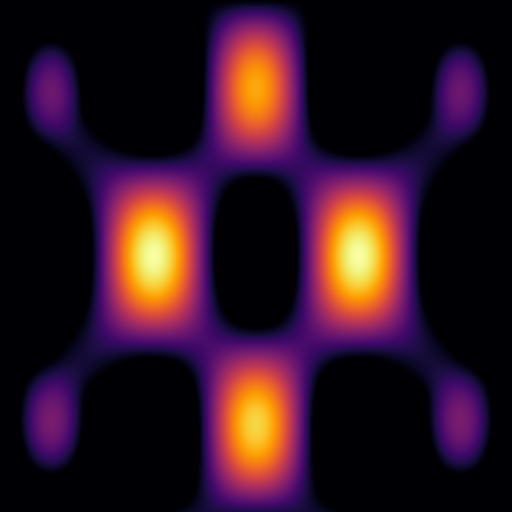}
    \end{minipage}\hfil
    \begin{minipage}[b]{0.2\linewidth}
      \centering
      \includegraphics[width=.99\linewidth]{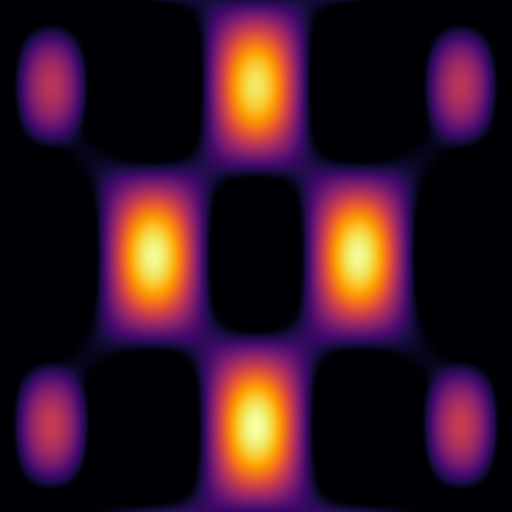}
    \end{minipage}\hfil
    \begin{minipage}[b]{0.2\linewidth}
      \centering
      \includegraphics[width=.99\linewidth]{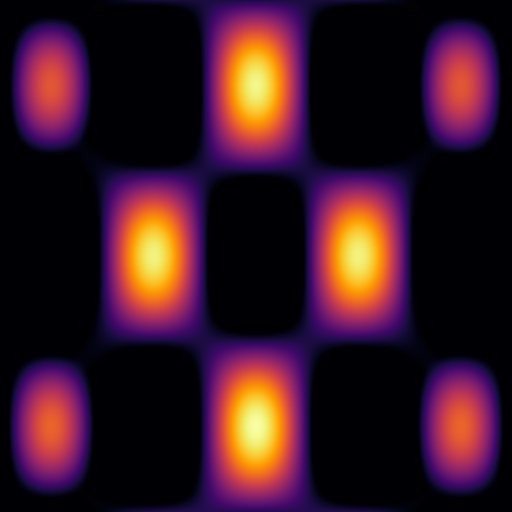}
    \end{minipage}\hfil
    \begin{minipage}[b]{0.2\linewidth}
      \centering
      \includegraphics[width=.99\linewidth]{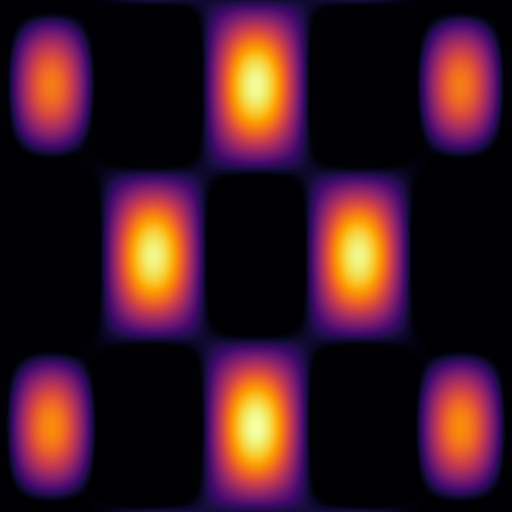}
    \end{minipage}\hfil
    \begin{minipage}[b]{0.2\linewidth}
      \centering
      \includegraphics[width=.99\linewidth]{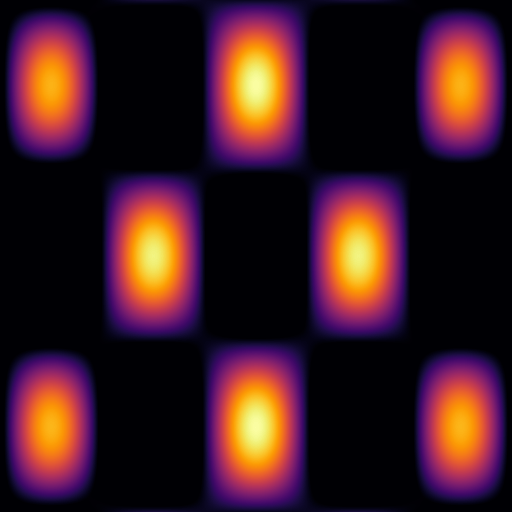}
    \end{minipage}\hfil
    \caption{\footnotesize PME with exponent $m=2$ and potential given by (\ref{eq:sine}).  The images show the evolution from time $t=0$ to $t=5$ (top left to bottom right).  The final image is the approximate steady state. Images are $512\times 512$ pixels. Brighter pixels indicate larger density values.}
    \label{fig:slow-diffusion-sine-example-m-2}
\end{figure}
\begin{figure}[h]
    \begin{minipage}[b]{0.2\linewidth}
      \centering
      \includegraphics[width=.99\linewidth]{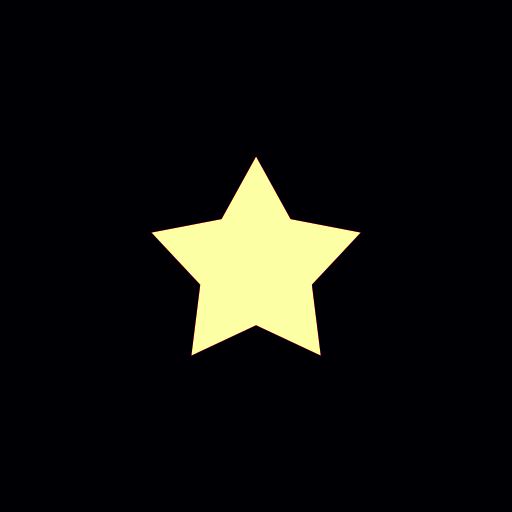}
    \end{minipage}\hfil
    \begin{minipage}[b]{0.2\linewidth}
      \centering
      \includegraphics[width=.99\linewidth]{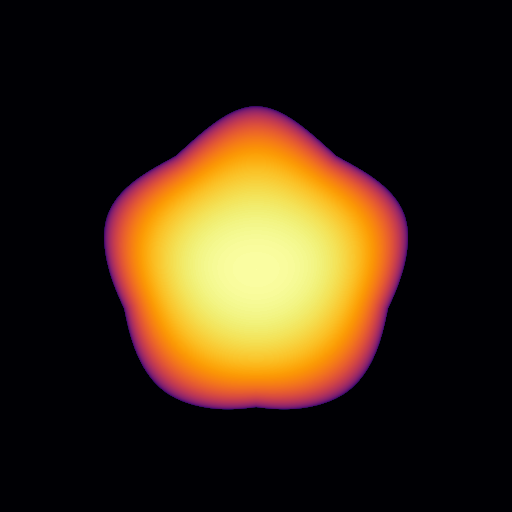}
    \end{minipage}\hfil
    \begin{minipage}[b]{0.2\linewidth}
      \centering
      \includegraphics[width=.99\linewidth]{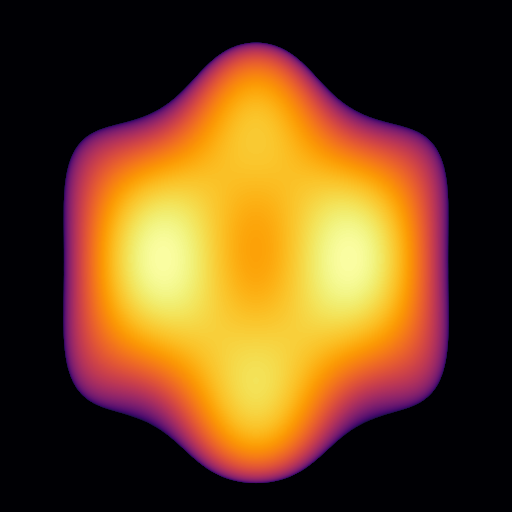}
    \end{minipage}\hfil
    \begin{minipage}[b]{0.2\linewidth}
      \centering
      \includegraphics[width=.99\linewidth]{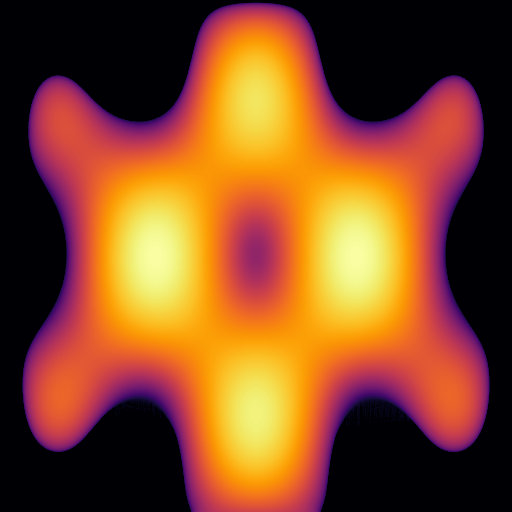}
    \end{minipage}\hfil
    \begin{minipage}[b]{0.2\linewidth}
      \centering
      \includegraphics[width=.99\linewidth]{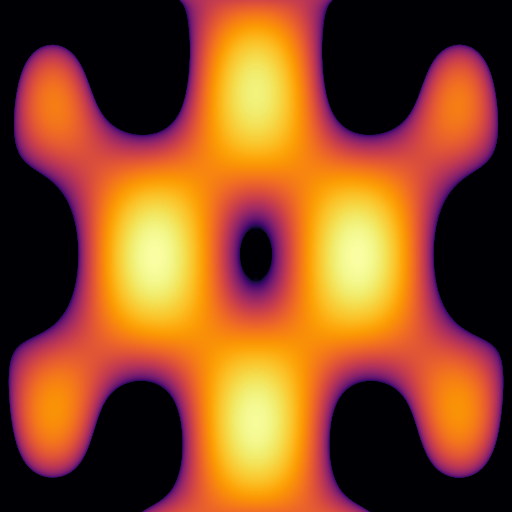}
    \end{minipage}\hfil
    \medskip
    \begin{minipage}[b]{0.2\linewidth}
      \centering
      \includegraphics[width=.99\linewidth]{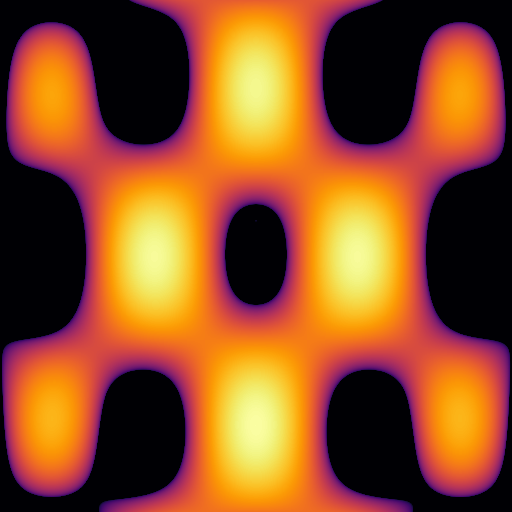}
    \end{minipage}\hfil
    \begin{minipage}[b]{0.2\linewidth}
      \centering
      \includegraphics[width=.99\linewidth]{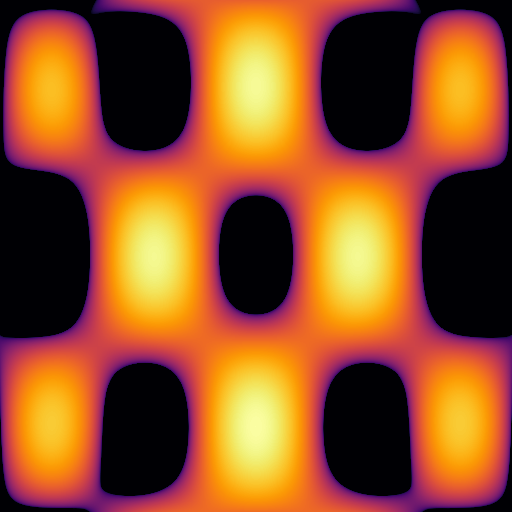}
    \end{minipage}\hfil
    \begin{minipage}[b]{0.2\linewidth}
      \centering
      \includegraphics[width=.99\linewidth]{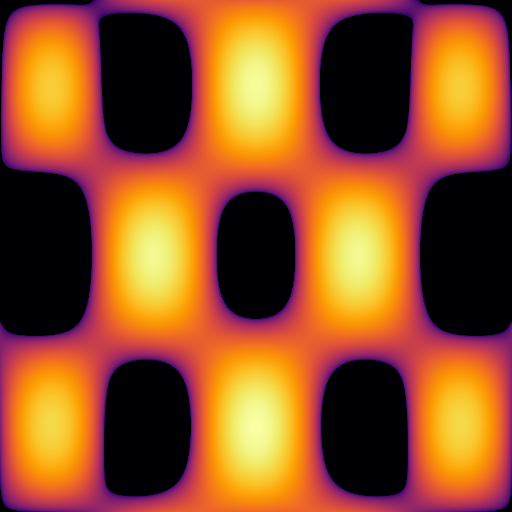}
    \end{minipage}\hfil
    \begin{minipage}[b]{0.2\linewidth}
      \centering
      \includegraphics[width=.99\linewidth]{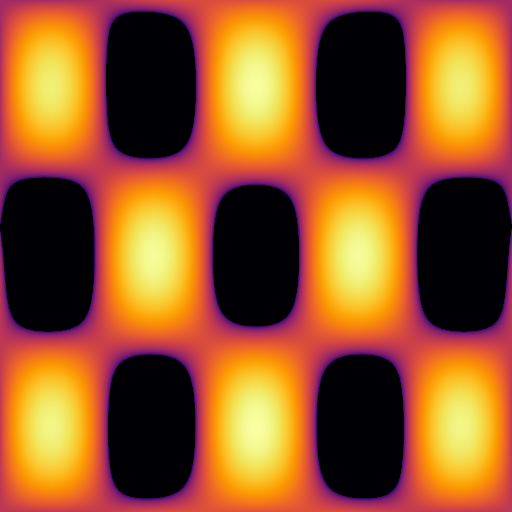}
    \end{minipage}\hfil
    \begin{minipage}[b]{0.2\linewidth}
      \centering
      \includegraphics[width=.99\linewidth]{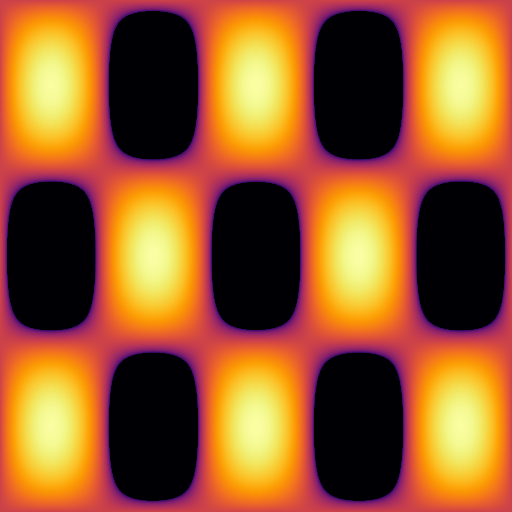}
    \end{minipage}\hfil
    \caption{\footnotesize PME with exponent $m=4$ and potential given by (\ref{eq:sine}).  The images show the evolution from time $t=0$ to $t=2$ (top left to bottom right).  The final image is the approximate steady state. Images are $512\times 512$ pixels. Brighter pixels indicate larger density values.}
    \label{fig:slow-diffusion-sine-example-m-4}
\end{figure}


Next, we consider a different potential function:
\begin{equation}\label{eq:obstacle_potential}
    V_2(x) = 10 \left((x_1-0.4)^2 + (x_2-0.4)^2\right) + \iota_{\Omega\setminus E}(x)
\end{equation}
where $E$ is a given subset of $\Omega$ and $\iota_{\Omega\setminus E}:\Omega\rightarrow \R \cup \{+\infty\}$ is the indicator function
\begin{equation*}
    \iota_{\Omega\setminus E}(x) =
    \begin{cases}
        0 & \text{if } x \in \Omega\backslash E\\
        +\infty & \text{if } x \in E.
    \end{cases}
\end{equation*}
With this setup, the set $E$ represents an obstacle that the density is not allowed to penetrate.   During the flow, the density diffuses and drifts towards the lower level sets of $V_2$, all while avoiding the set $E$.

In Figure~\ref{fig:m-4-obstacle-contour} and Figure~\ref{fig:m-4-obstacle-contour-star}, we display two different experiments with different obstacles $E$, but the same diffusion exponent $m=4$.  In both experiments, the starting density is the characteristic function of a square centered at $(-0.3,-0.3)$ with side length $0.2$ renormalized to have unit mass.  
In  Figure~\ref{fig:m-4-obstacle-contour}, the obstacle is a disc with radius $0.2$ centered at the origin, and in  Figure~\ref{fig:m-4-obstacle-contour-star}, the obstacle is a star shaped region centered at the origin. In both experiments, we set $\tau=0.001$, $\gamma=0.0075$ and we run the flow until time $t=2$.
An interesting difference between the two flows is that the non-convexity of the star shaped obstacle results in some mass being trapped between the arms of the star.  It is not entirely clear if the mass eventually escapes as time goes to infinity. This is because the PME allows for compactly supported solutions (in contrast to say the behavior of the heat equation).
\begin{figure}[h]
    \begin{minipage}[b]{0.2\linewidth}
      \centering
      \includegraphics[width=.99\linewidth]{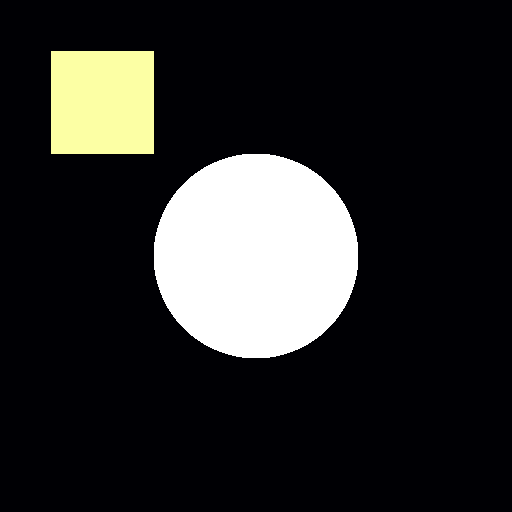}
    \end{minipage}\hfil
    \begin{minipage}[b]{0.2\linewidth}
      \centering
      \includegraphics[width=.99\linewidth]{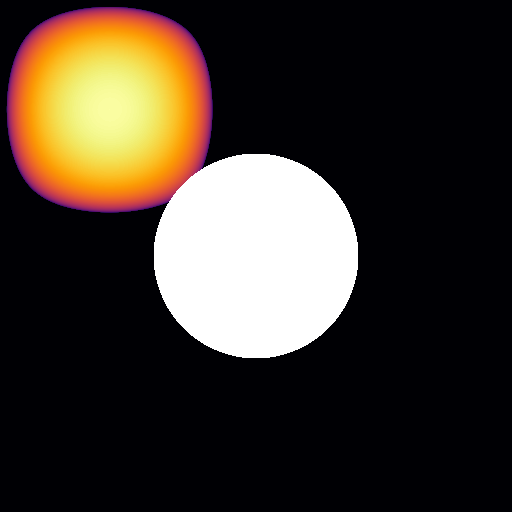}
    \end{minipage}\hfil
    \begin{minipage}[b]{0.2\linewidth}
      \centering
      \includegraphics[width=.99\linewidth]{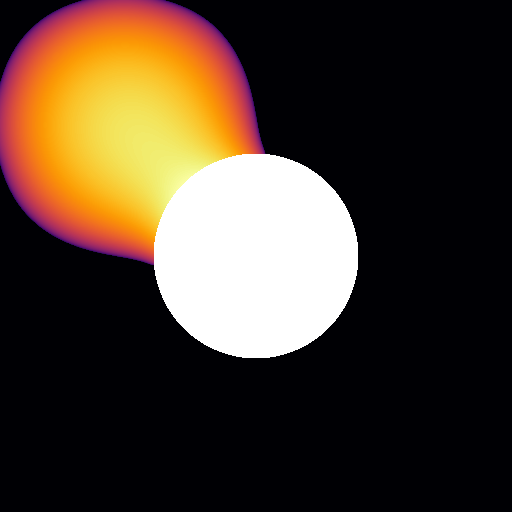}
    \end{minipage}\hfil
    \begin{minipage}[b]{0.2\linewidth}
      \centering
      \includegraphics[width=.99\linewidth]{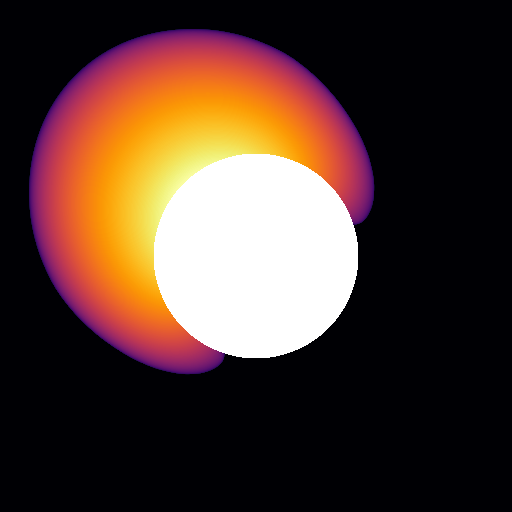}
    \end{minipage}\hfil
    \begin{minipage}[b]{0.2\linewidth}
      \centering
      \includegraphics[width=.99\linewidth]{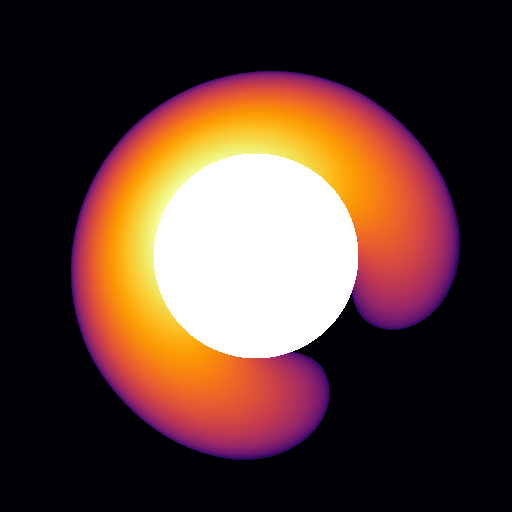}
    \end{minipage}\hfil
    \medskip
    \begin{minipage}[b]{0.2\linewidth}
      \centering
      \includegraphics[width=.99\linewidth]{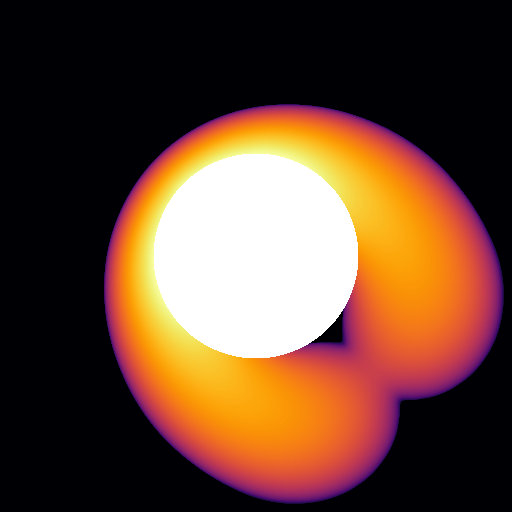}
    \end{minipage}\hfil
    \begin{minipage}[b]{0.2\linewidth}
      \centering
      \includegraphics[width=.99\linewidth]{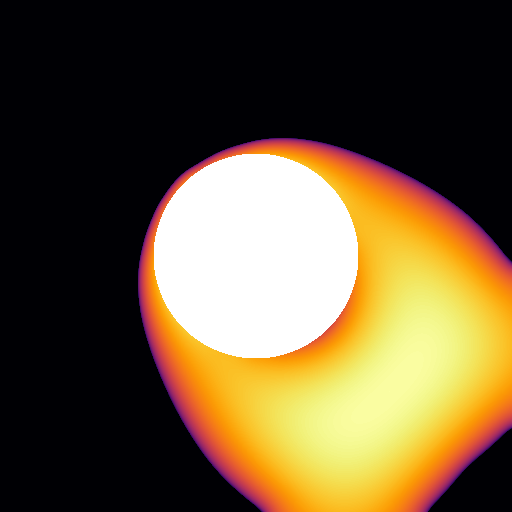}
    \end{minipage}\hfil
    \begin{minipage}[b]{0.2\linewidth}
      \centering
      \includegraphics[width=.99\linewidth]{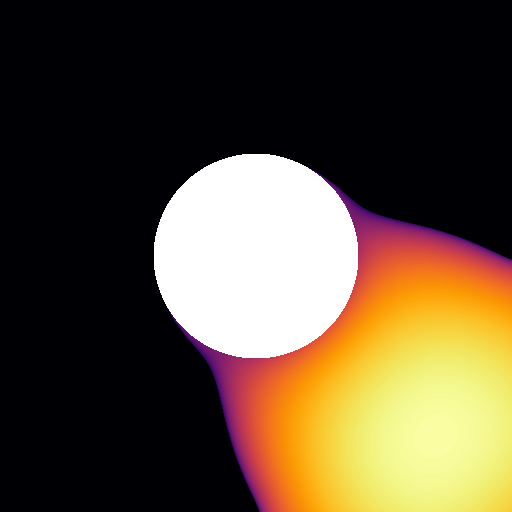}
    \end{minipage}\hfil
    \begin{minipage}[b]{0.2\linewidth}
      \centering
      \includegraphics[width=.99\linewidth]{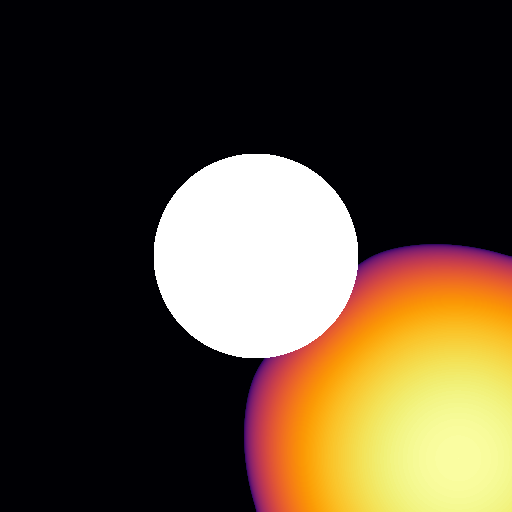}
    \end{minipage}\hfil
    \begin{minipage}[b]{0.2\linewidth}
      \centering
      \includegraphics[width=.99\linewidth]{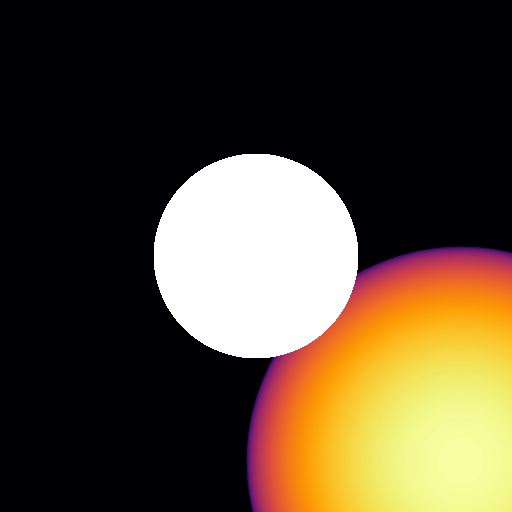}
    \end{minipage}\hfil
    \medskip
    \caption{\footnotesize PME with exponent $m=4$, $\gamma=.0075$ and potential given by (\ref{eq:obstacle_potential}).  The obstacle $E$ is represented by the white region.  The images show the evolution from time $t=0$ to $t=2$ (top left to bottom right).  Images are $512\times 512$ pixels. With the exception of the obstacle, brighter pixels indicate larger density values.}
    \label{fig:m-4-obstacle-contour}
\end{figure}

\begin{figure}[h]
    \begin{minipage}[b]{0.2\linewidth}
      \centering
      \includegraphics[width=.99\linewidth]{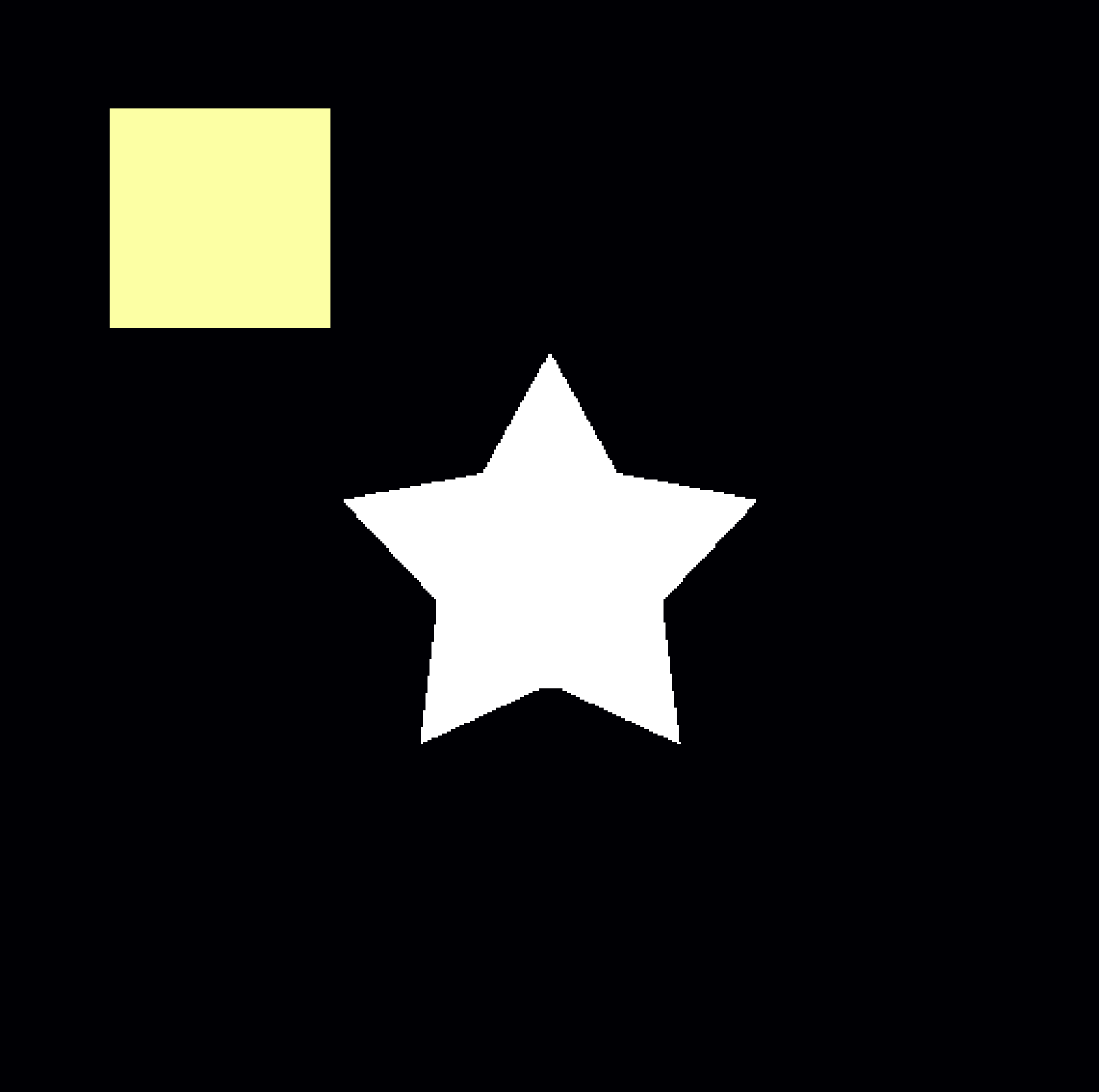}
    \end{minipage}\hfil
    \begin{minipage}[b]{0.2\linewidth}
      \centering
      \includegraphics[width=.99\linewidth]{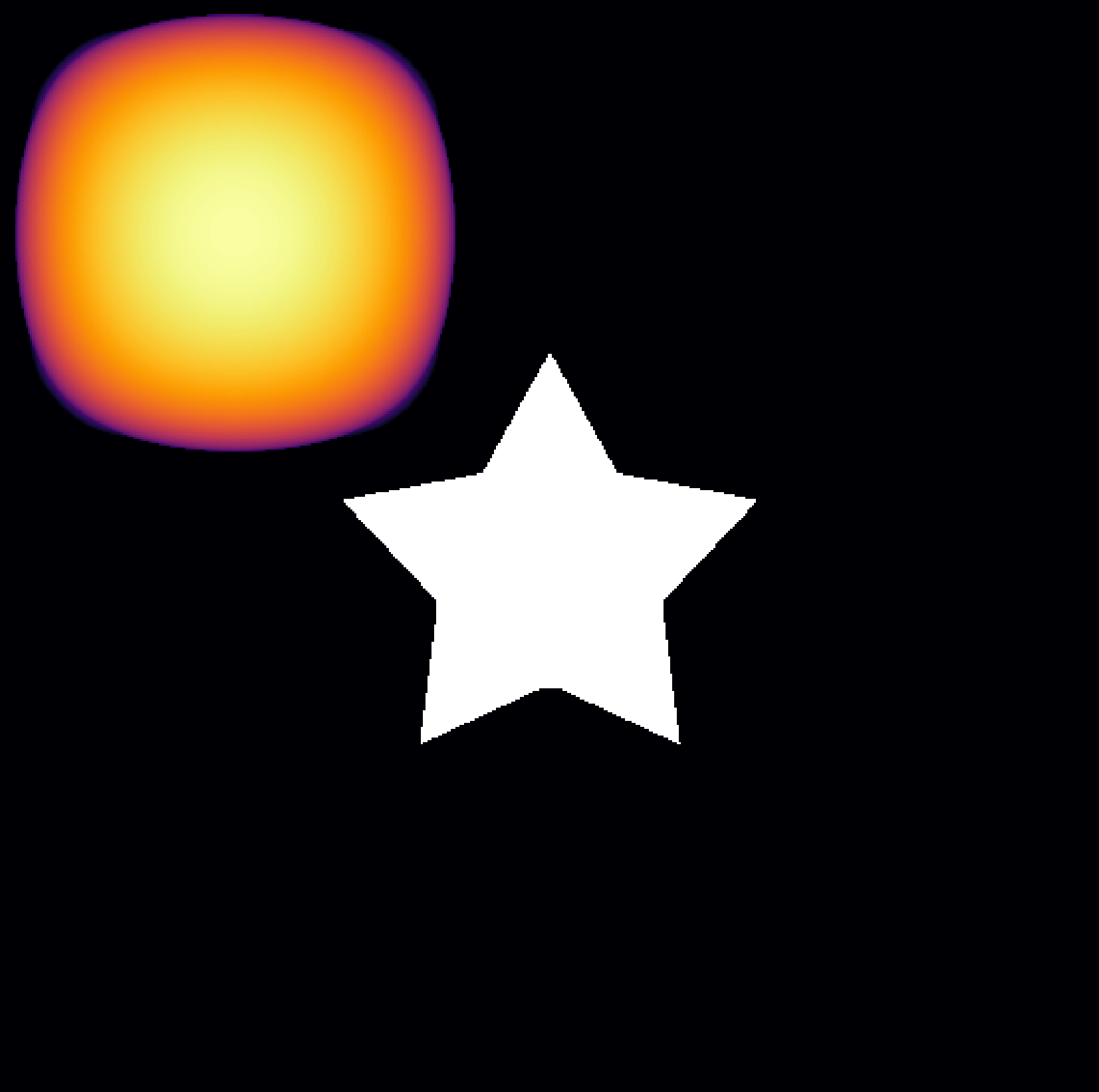}
    \end{minipage}\hfil
    \begin{minipage}[b]{0.2\linewidth}
      \centering
      \includegraphics[width=.99\linewidth]{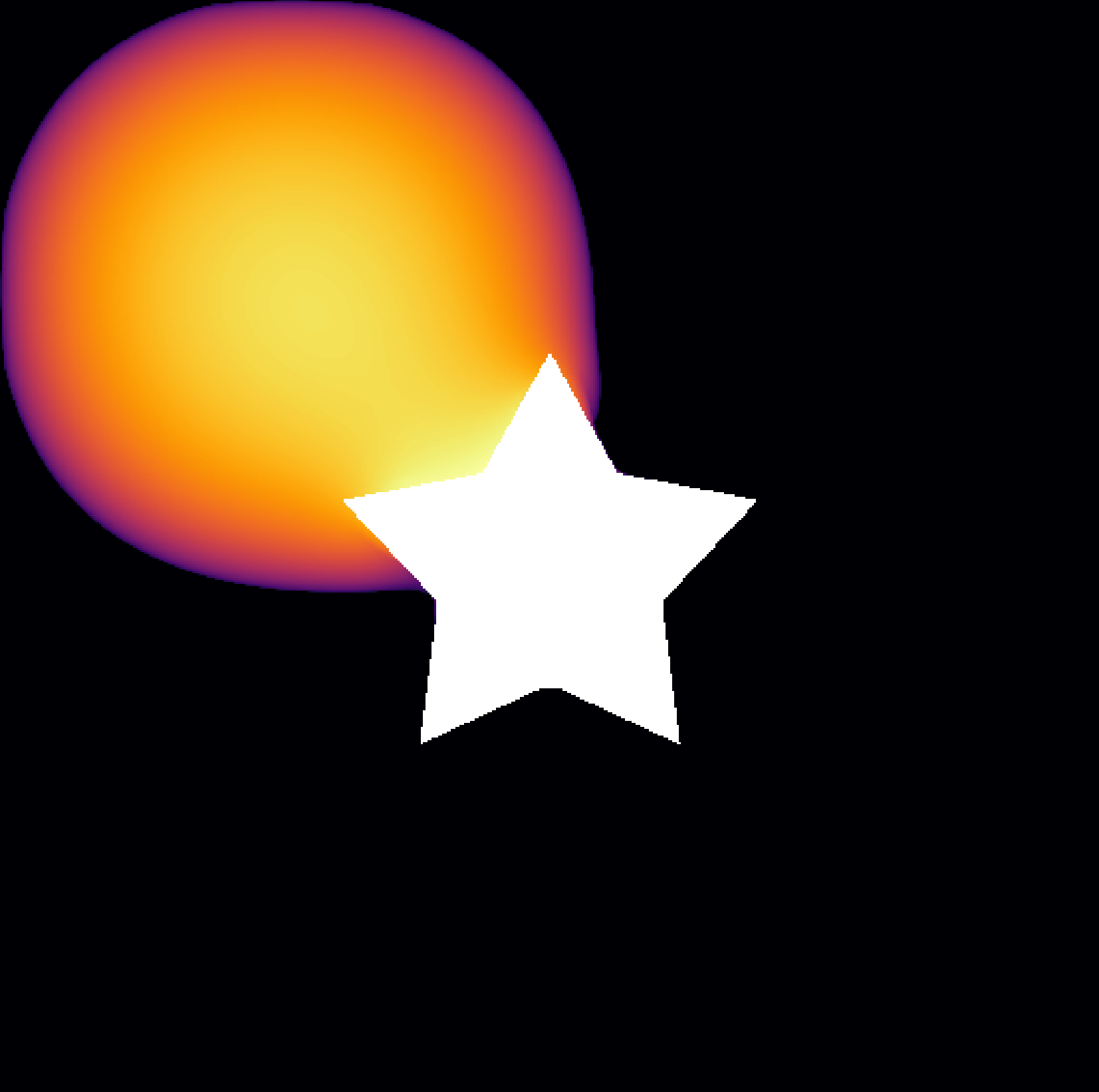}
    \end{minipage}\hfil
    \begin{minipage}[b]{0.2\linewidth}
      \centering
      \includegraphics[width=.99\linewidth]{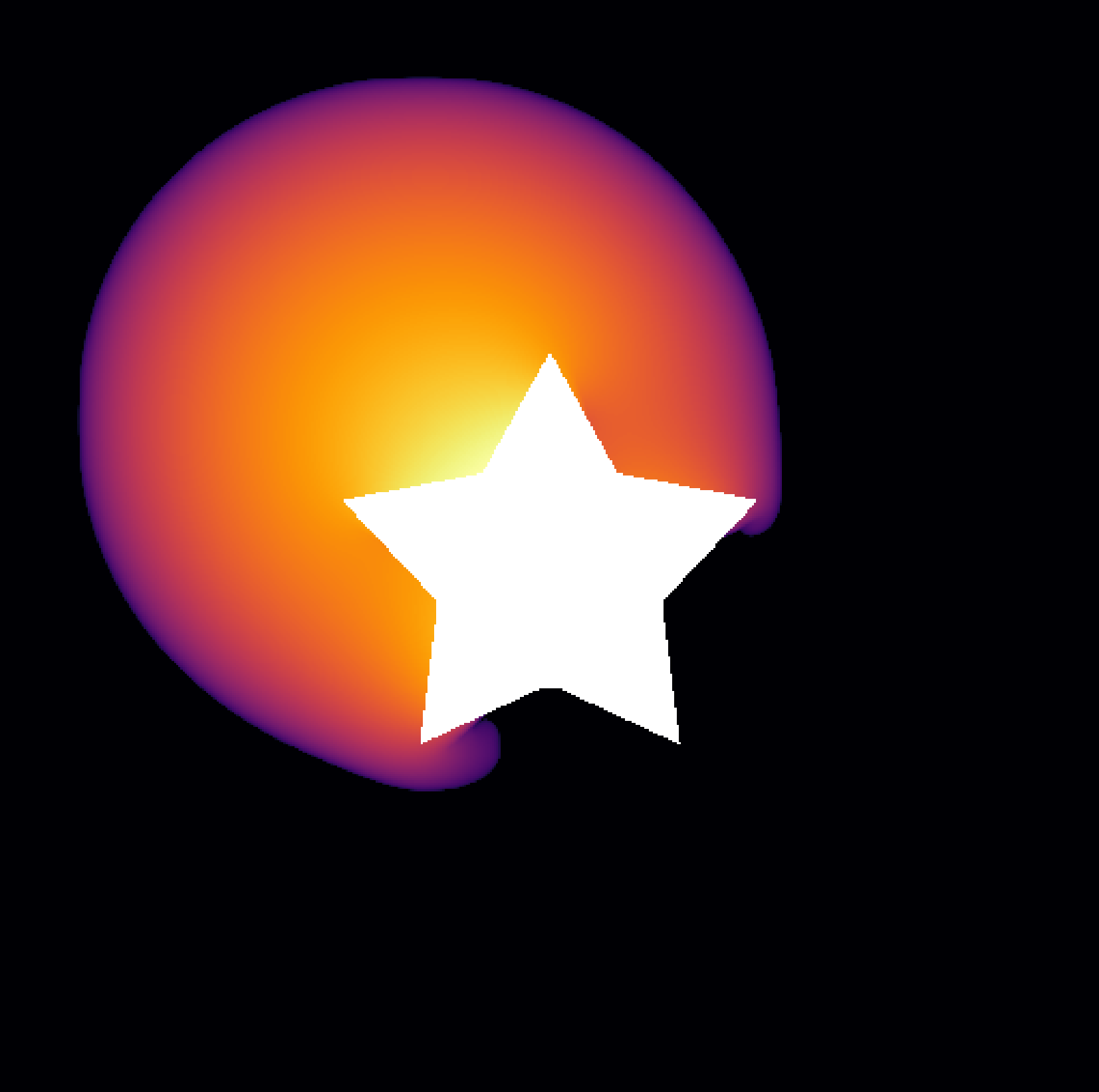}
    \end{minipage}\hfil
    \begin{minipage}[b]{0.2\linewidth}
      \centering
      \includegraphics[width=.99\linewidth]{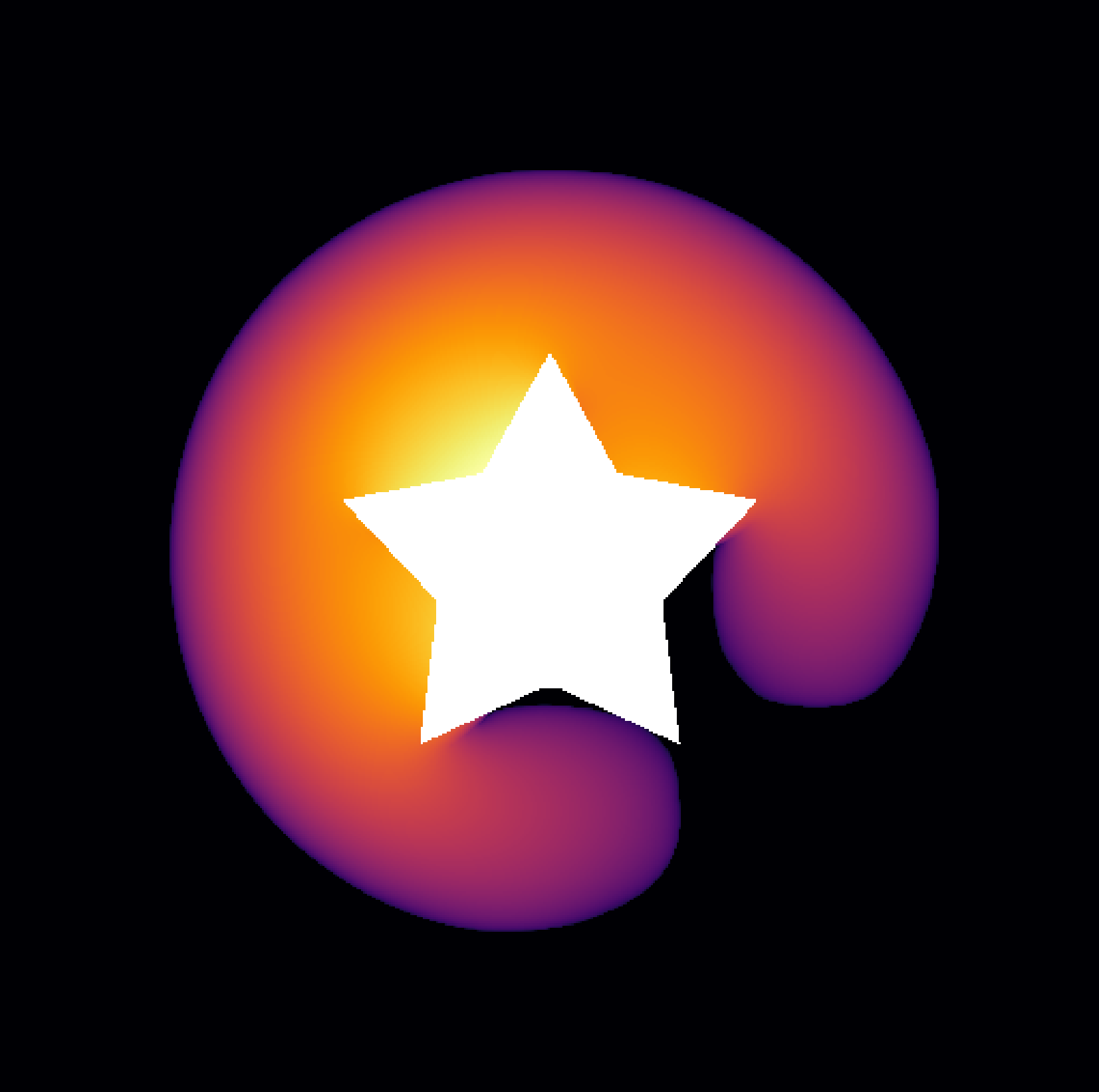}
    \end{minipage}\hfil
    \medskip
    \begin{minipage}[b]{0.2\linewidth}
      \centering
      \includegraphics[width=.99\linewidth]{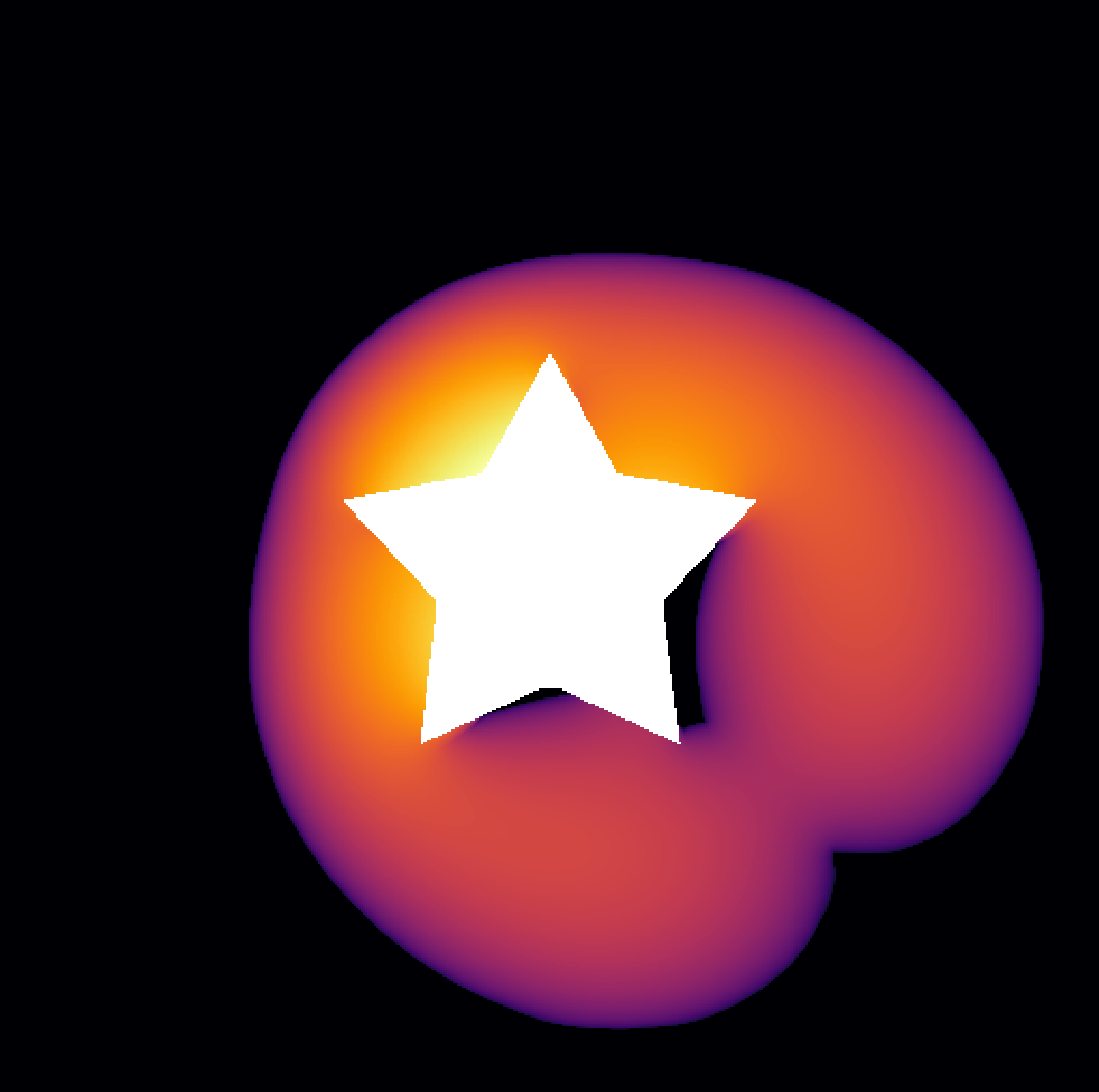}
    \end{minipage}\hfil
    \begin{minipage}[b]{0.2\linewidth}
      \centering
      \includegraphics[width=.99\linewidth]{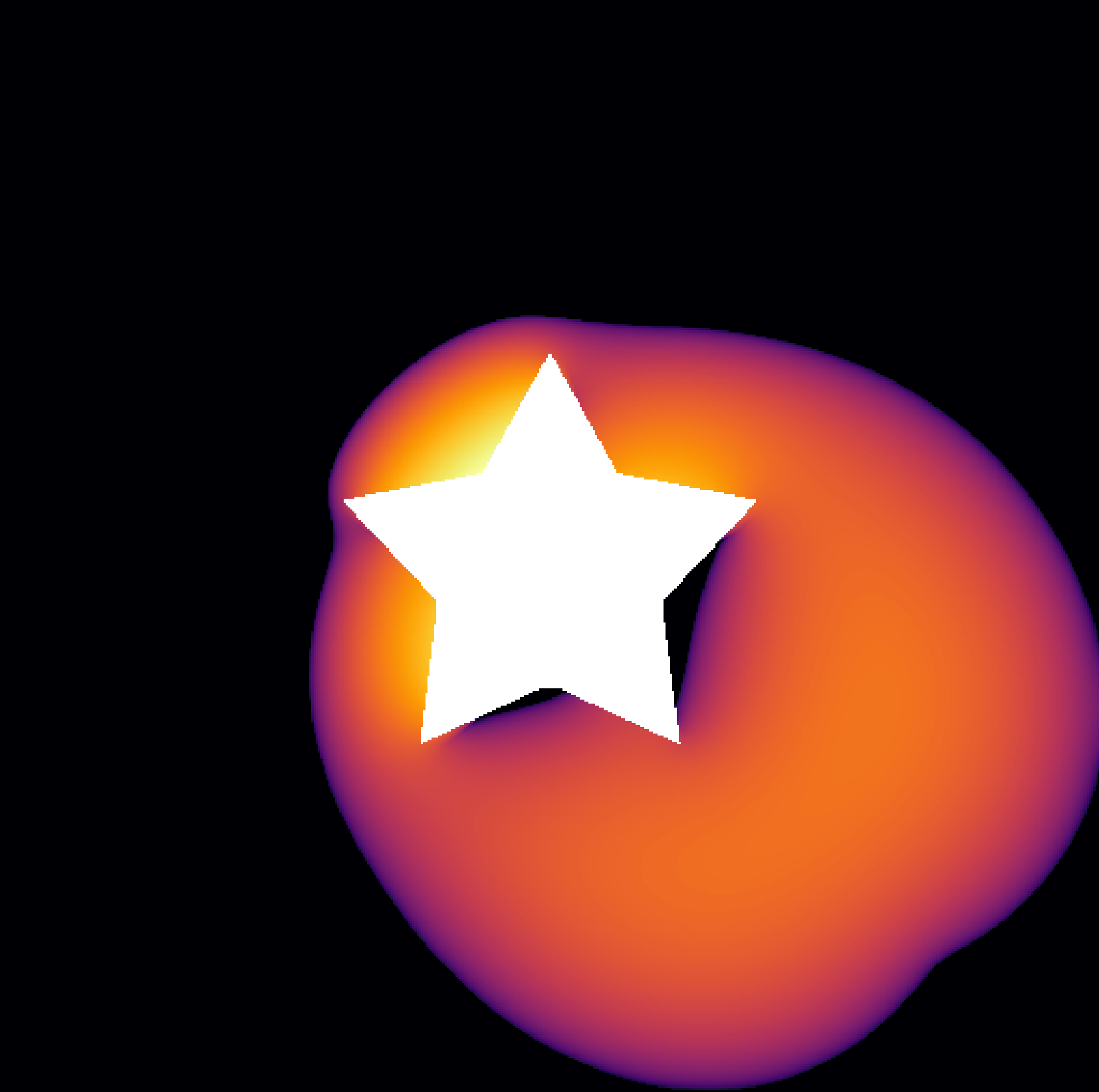}
    \end{minipage}\hfil
    \begin{minipage}[b]{0.2\linewidth}
      \centering
      \includegraphics[width=.99\linewidth]{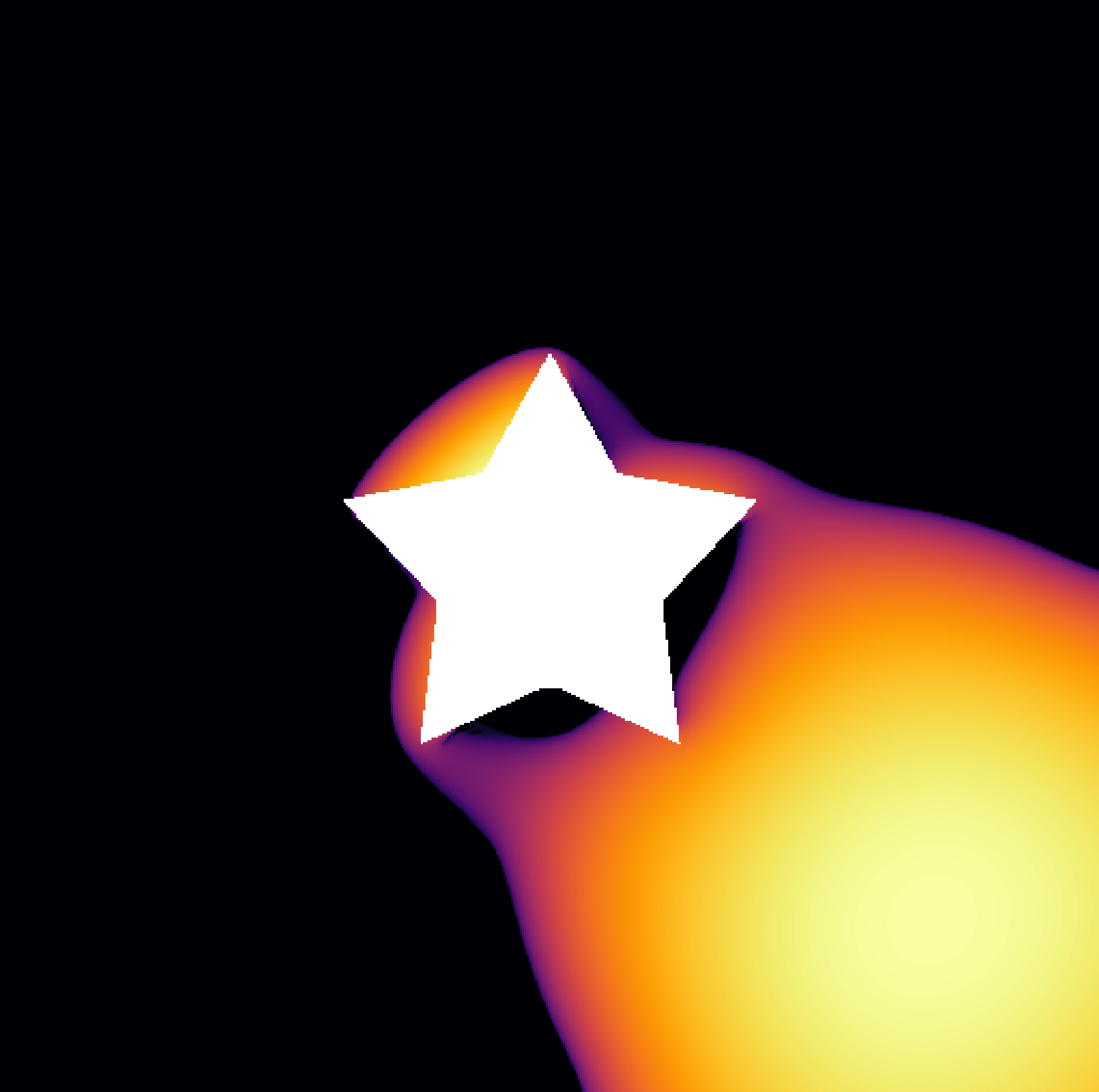}
    \end{minipage}\hfil
    \begin{minipage}[b]{0.2\linewidth}
      \centering
      \includegraphics[width=.99\linewidth]{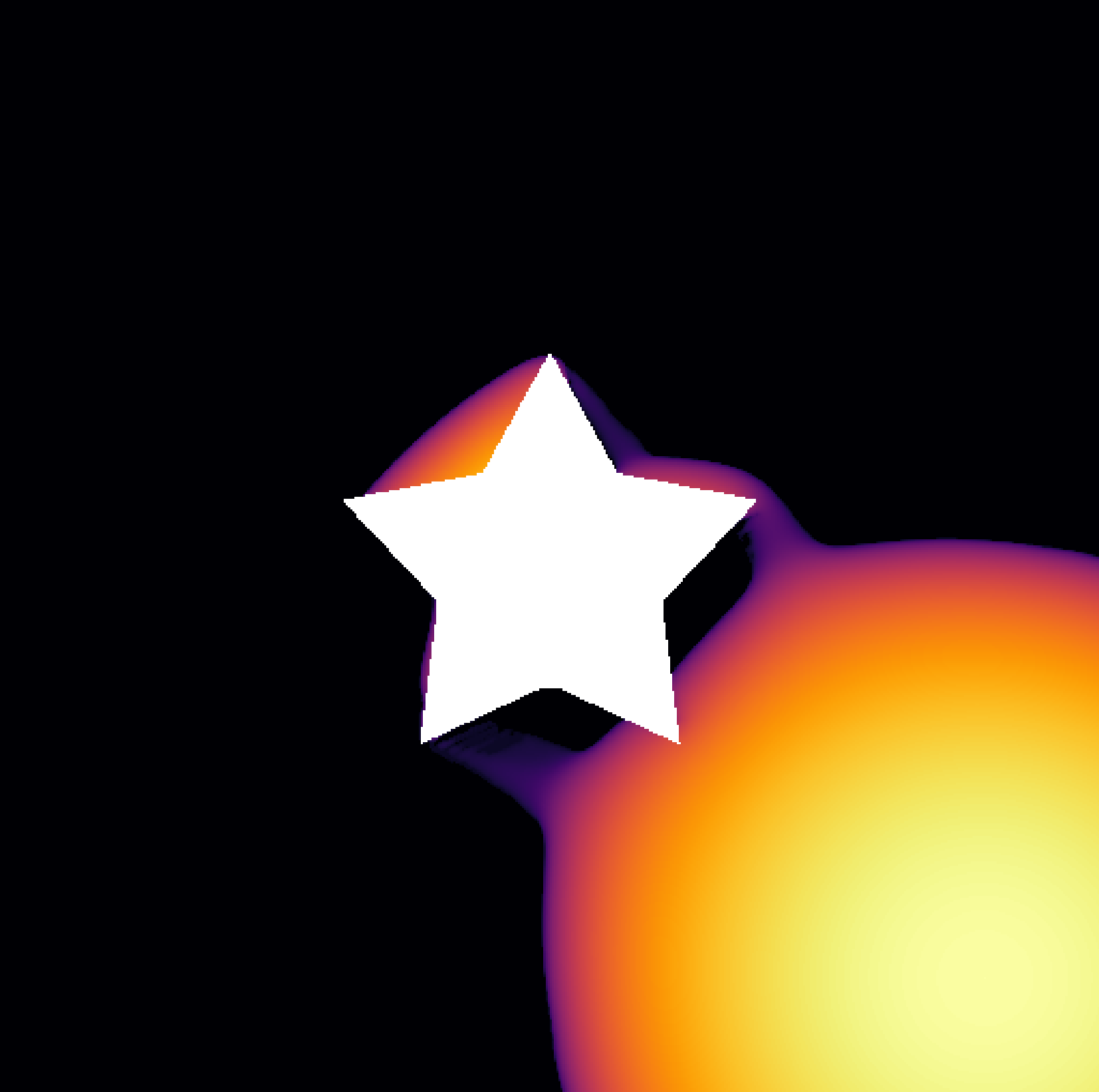}
    \end{minipage}\hfil
    \begin{minipage}[b]{0.2\linewidth}
      \centering
      \includegraphics[width=.99\linewidth]{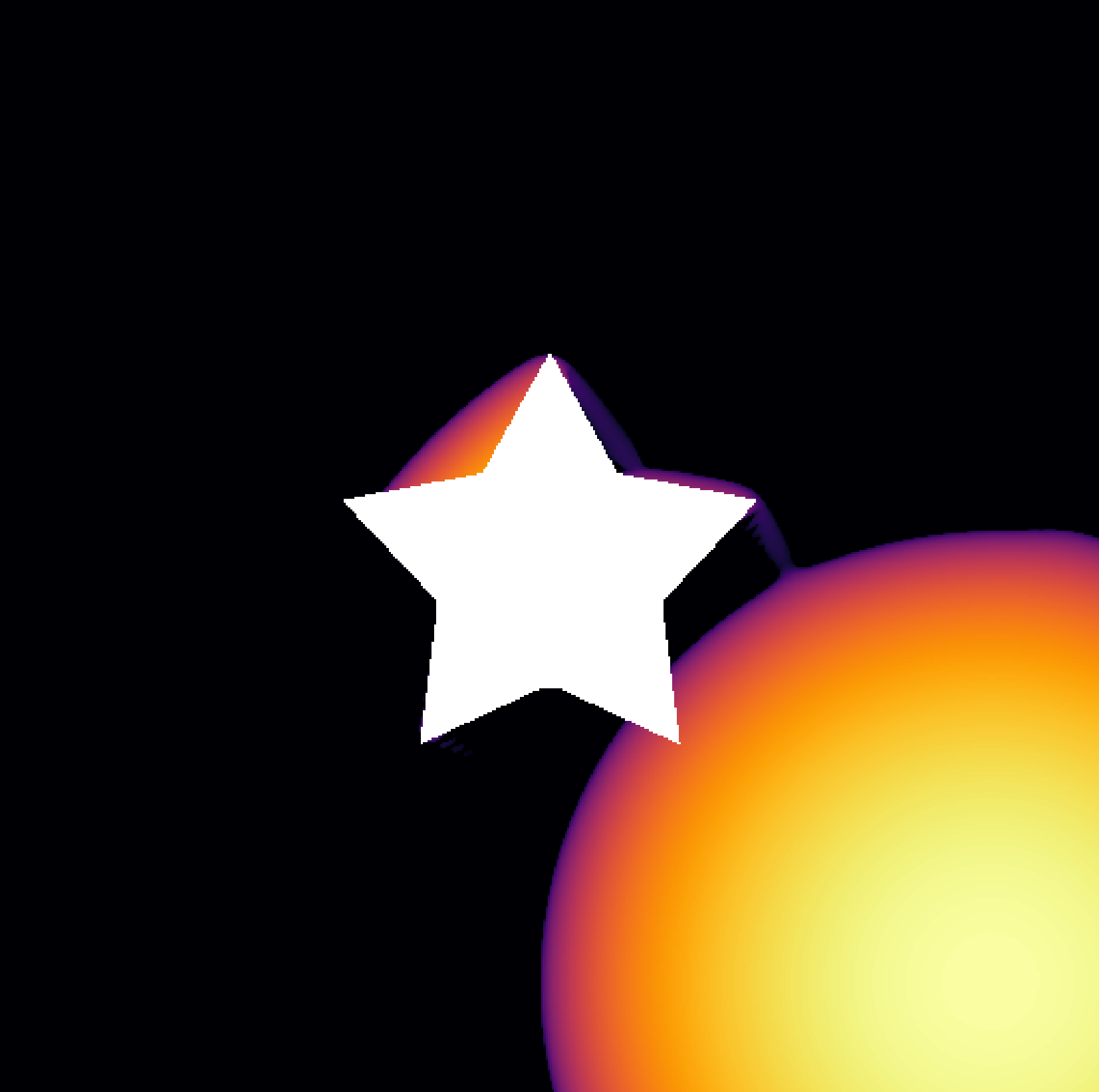}
    \end{minipage}\hfil
    \medskip
    \caption{\footnotesize PME with exponent $m=4$, $\gamma=.0075$ and potential given by (\ref{eq:obstacle_potential}).  The obstacle $E$ is represented by the white region.  The images show the evolution from time $t=0$ to $t=2$ (top left to bottom right). Images are $512\times 512$ pixels. With the exception of the obstacle, brighter pixels indicate larger density values.}
    \label{fig:m-4-obstacle-contour-star}
\end{figure}

\subsubsection{Non-convex \texorpdfstring{$U$}{U} (aggregation-diffusion)}

In this experiment, we simulate (\ref{eq:pde}) with an energy functional $U$ that is not a convex with respect to $\rho$.  Specifically, we consider the energy
\begin{equation}\label{eq:nonconvex_example}
    U(\rho) =\mathcal{W}(\rho)+ \int_\Omega \frac{1}{60} \rho^3(x)\, dx, 
\end{equation}
where 
\[
\mathcal{W}(\rho):= \frac{1}{2}\int_{\Omega}\int_{\Omega}|x-y|^2\rho(x)\rho(y)\, dy \, dx.
\]
By separating out the square, one can check $\mathcal{W}$ is concave with respect to $\rho$.

While convex energies $U$ encourage mass diffusion, non-convex energies allow for both aggregation and diffusion phenomena.  Indeed, one can see that $\mathcal{W}(\rho)$ encourages the density to concentrate while the $\rho^3$ term encourages the density to diffuse.  Due to the convolution, $\mathcal{W}$ can be viewed as a ``lower order'' term as compared to $\rho^3$.  However, since the coefficients in front of the convolution is much larger than the coefficient in front of the $\rho^3$ term, the aggregation effect will dominate until the density reaches a certain saturation level. 

Here we run a single experiment starting with an initial density that is the sum of the characteristic function of four squares with side lengths $0.2$ centered at each combination of $(\pm 0.3, \pm 0.3)$ and  renormalized to have total mass equal to one.   We set $\tau=.005$ and run the flow from time $t=0$ to $t=10$, at which time the evolution appears to have reached a steady state.

The results of the experiment are displayed in Figures~\ref{fig:m-3-attraction-contour} and~\ref{fig:m-3-attraction-surface}.  
Figure~\ref{fig:m-3-attraction-contour} displays a heat map of the density evolution, while Figure~\ref{fig:m-3-attraction-surface} gives a 3 dimensional plot showing the height of the density.  Throughout the evolution, one can see the competing effects of aggregation and diffusion.  The heights of the four densities decrease due to diffusion, however aggregation pulls the four separate components together towards the center of the domain.

\begin{figure}[h]
    \begin{minipage}[b]{0.2\linewidth}
      \centering
      \includegraphics[width=.99\linewidth]{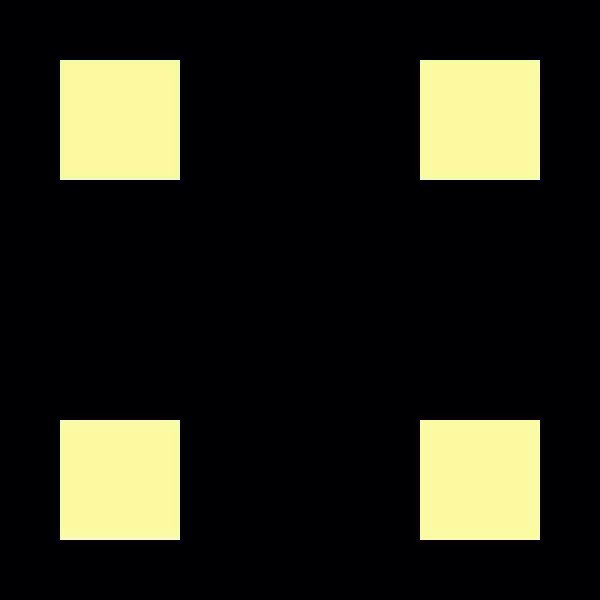}
    \end{minipage}\hfil
    \begin{minipage}[b]{0.2\linewidth}
      \centering
      \includegraphics[width=.99\linewidth]{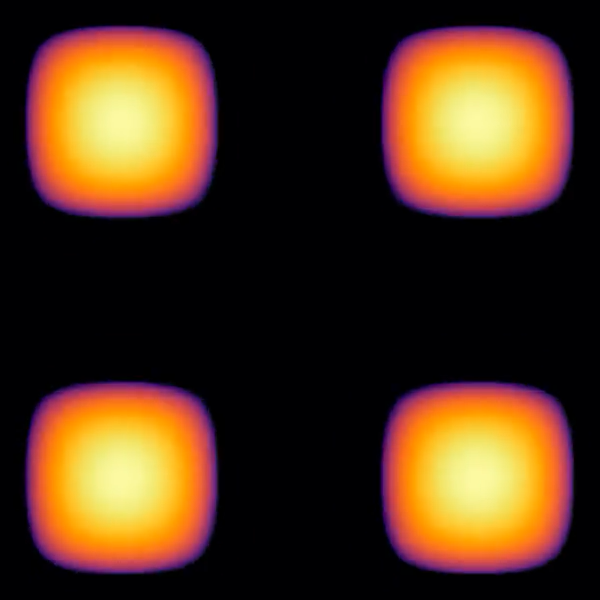}
    \end{minipage}\hfil
    \begin{minipage}[b]{0.2\linewidth}
      \centering
      \includegraphics[width=.99\linewidth]{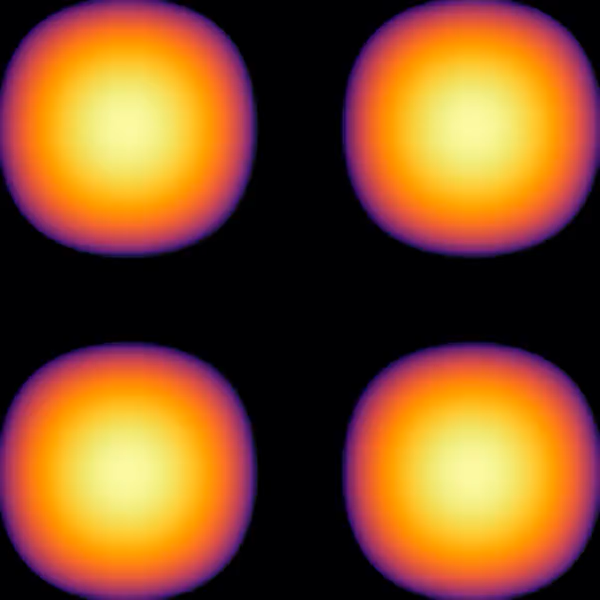}
    \end{minipage}\hfil
    \begin{minipage}[b]{0.2\linewidth}
      \centering
      \includegraphics[width=.99\linewidth]{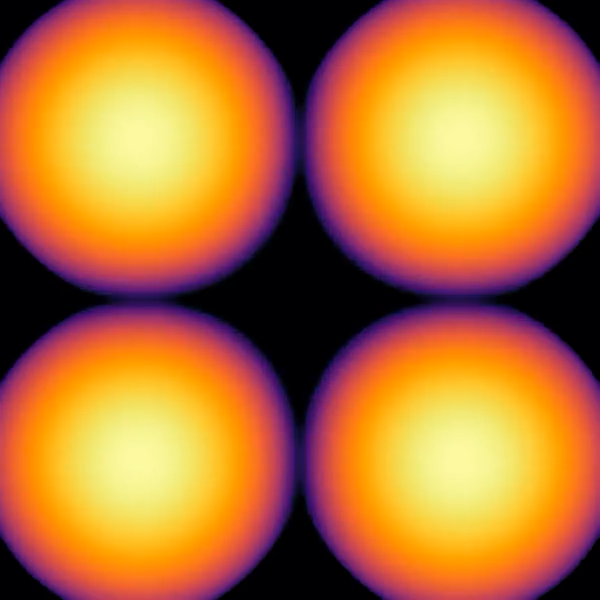}
    \end{minipage}\hfil
    \begin{minipage}[b]{0.2\linewidth}
      \centering
      \includegraphics[width=.99\linewidth]{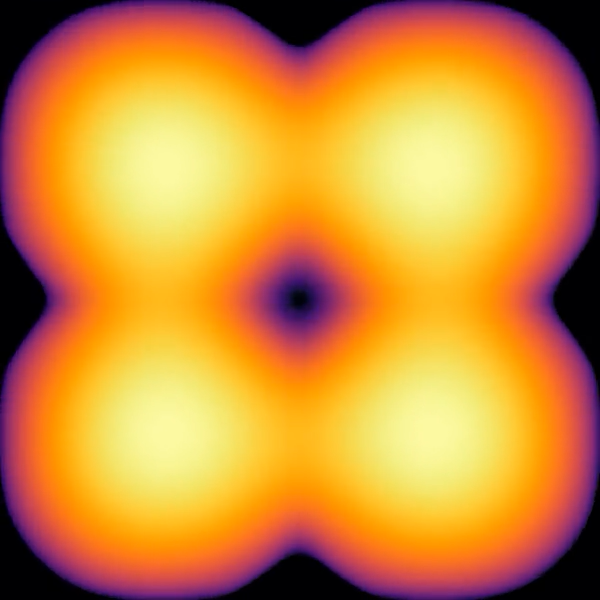}
    \end{minipage}\hfil
    \medskip
    \begin{minipage}[b]{0.2\linewidth}
      \centering
      \includegraphics[width=.99\linewidth]{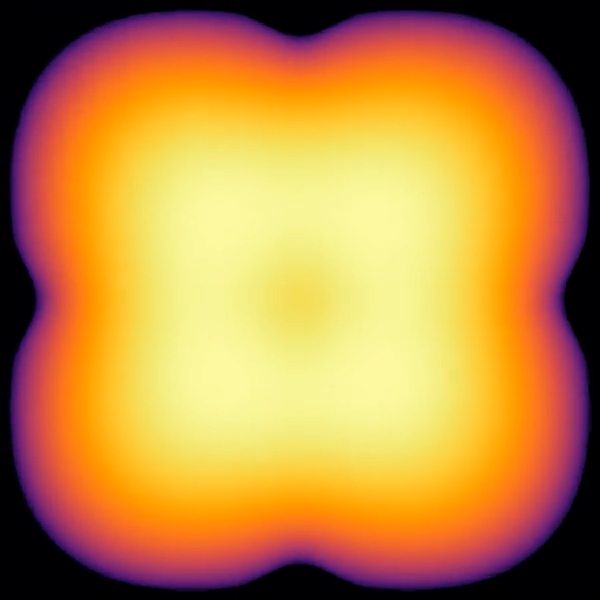}
    \end{minipage}\hfil
    \begin{minipage}[b]{0.2\linewidth}
      \centering
      \includegraphics[width=.99\linewidth]{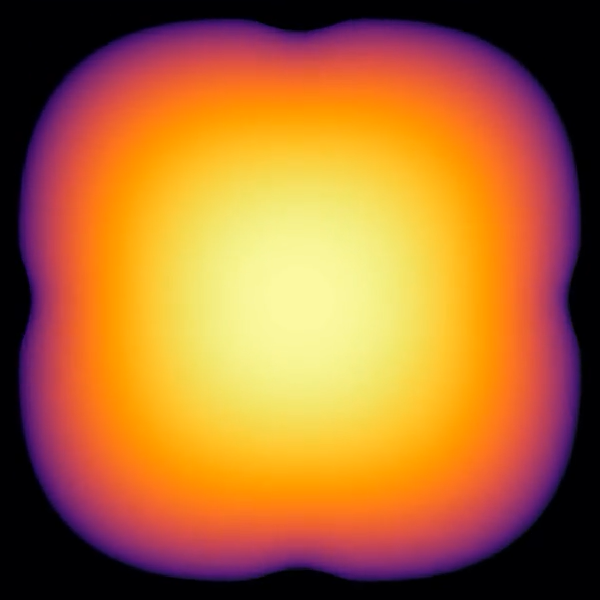}
    \end{minipage}\hfil
    \begin{minipage}[b]{0.2\linewidth}
      \centering
      \includegraphics[width=.99\linewidth]{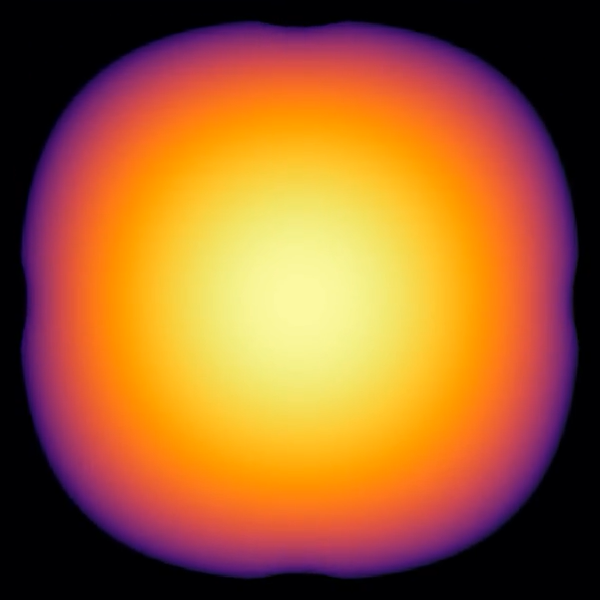}
    \end{minipage}\hfil
    \begin{minipage}[b]{0.2\linewidth}
      \centering
      \includegraphics[width=.99\linewidth]{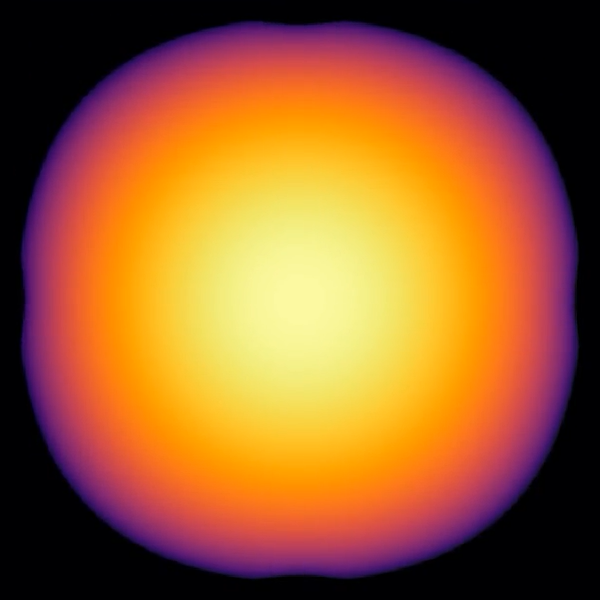}
    \end{minipage}\hfil
    \begin{minipage}[b]{0.2\linewidth}
      \centering
      \includegraphics[width=.99\linewidth]{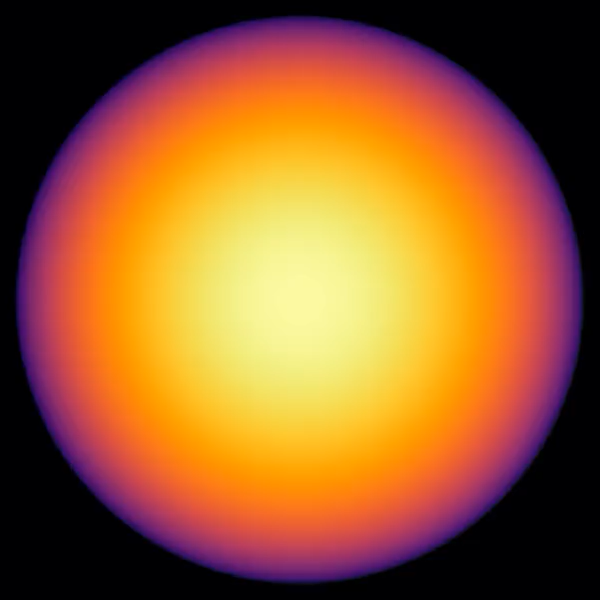}
    \end{minipage}\hfil
    \medskip
    \caption{\footnotesize Aggregation-diffusion equation with an energy given by (\ref{eq:nonconvex_example}).  The images show the evolution from time $t=0$ to $t=10$ (top left to bottom right).  The final image is the approximate steady state. Images are $512\times 512$ pixels. Brighter pixels indicate larger density values.}
    \label{fig:m-3-attraction-contour}
\end{figure}

\begin{figure}[h]
    \begin{minipage}[b]{0.25\linewidth}
      \centering
      \includegraphics[width=.99\linewidth]{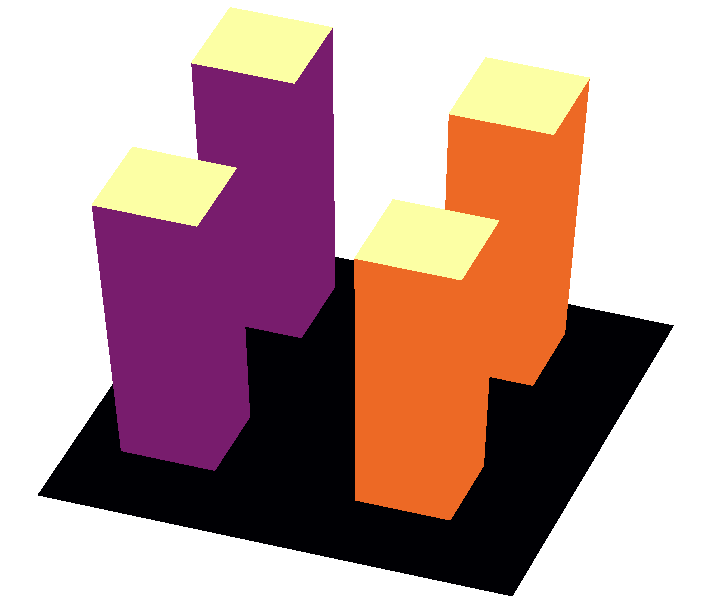}
    \end{minipage}\hfil
    \begin{minipage}[b]{0.25\linewidth}
      \centering
      \includegraphics[width=.99\linewidth]{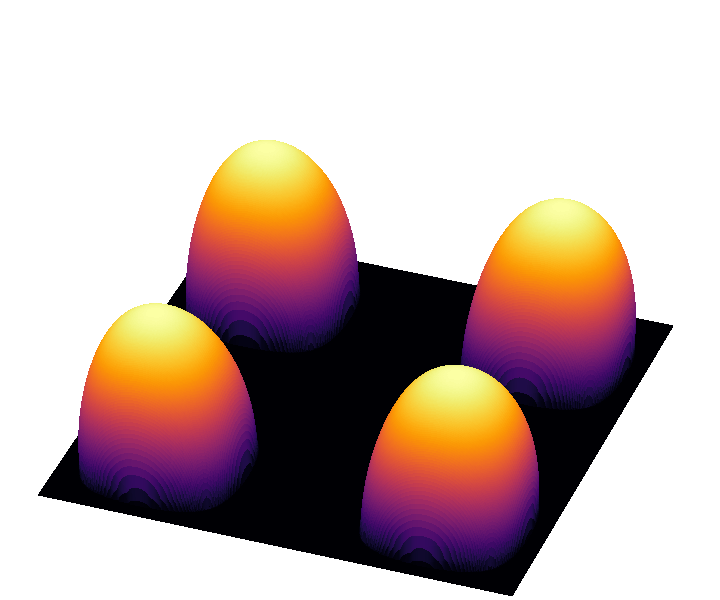}
    \end{minipage}\hfil
    \begin{minipage}[b]{0.25\linewidth}
      \centering
      \includegraphics[width=.99\linewidth]{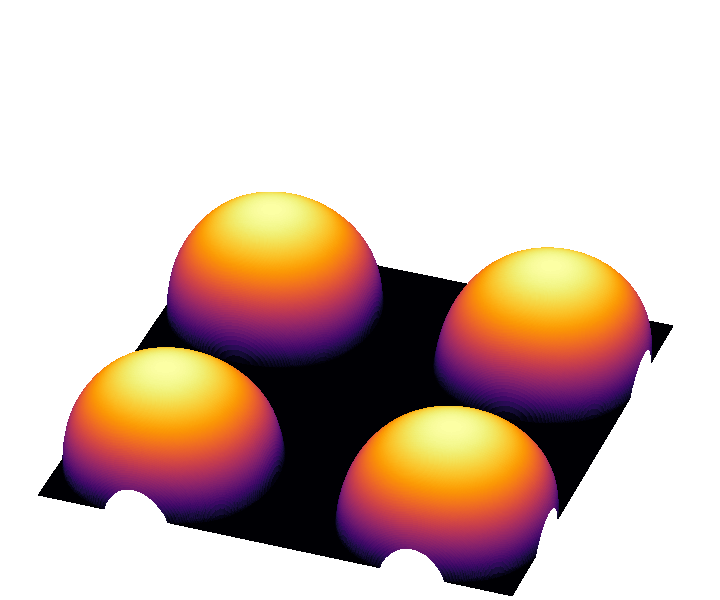}
    \end{minipage}\hfil
    \begin{minipage}[b]{0.25\linewidth}
      \centering
      \includegraphics[width=.99\linewidth]{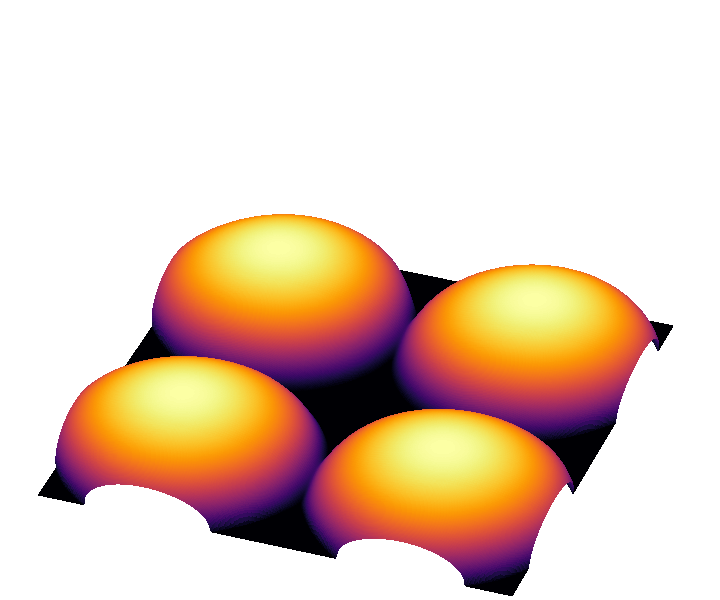}
    \end{minipage}\hfil
    \medskip
    \begin{minipage}[b]{0.25\linewidth}
      \centering
      \includegraphics[width=.99\linewidth]{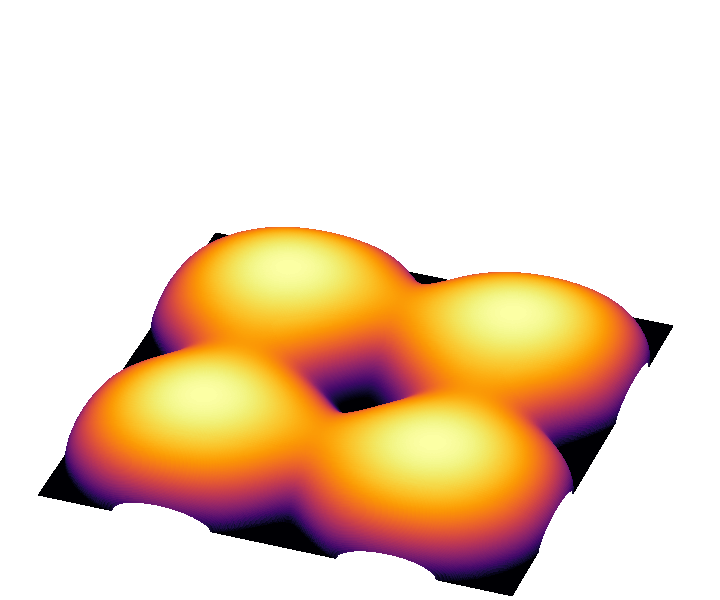}
    \end{minipage}\hfil
    \begin{minipage}[b]{0.25\linewidth}
      \centering
      \includegraphics[width=.99\linewidth]{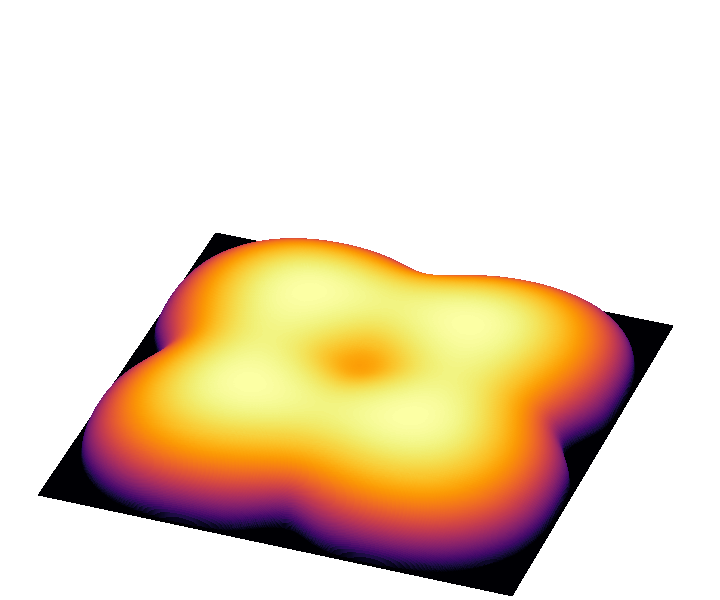}
    \end{minipage}\hfil
    \begin{minipage}[b]{0.25\linewidth}
      \centering
      \includegraphics[width=.99\linewidth]{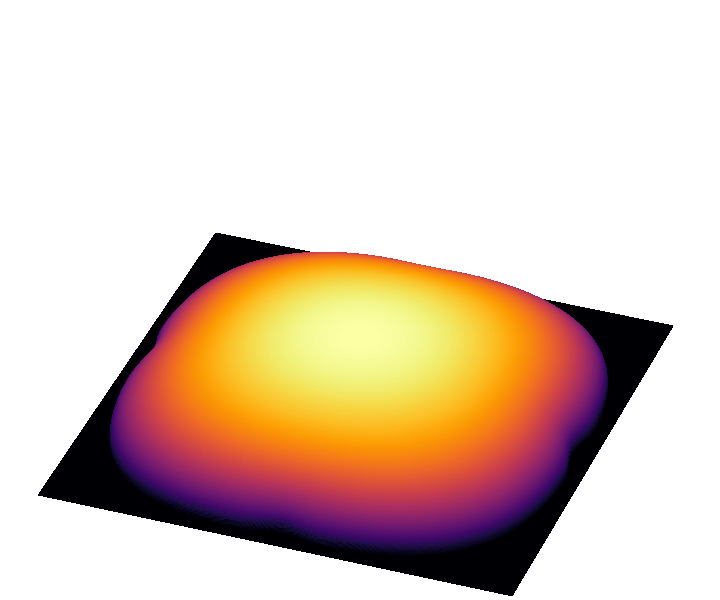}
    \end{minipage}\hfil
    \begin{minipage}[b]{0.25\linewidth}
      \centering
      \includegraphics[width=.99\linewidth]{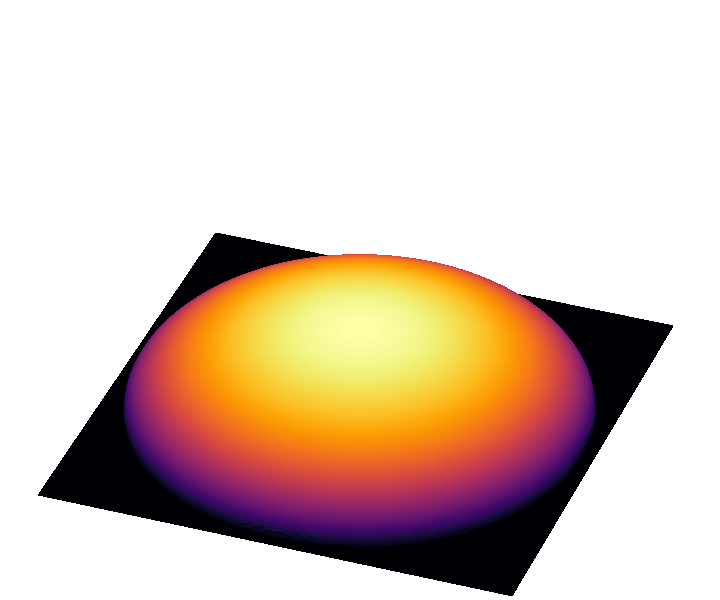}
    \end{minipage}\hfil
    \caption{\footnotesize Aggregation-diffusion equation with an energy given by (\ref{eq:nonconvex_example}).  The images show a 3-d surface plot of the evolution from time $t=0$ to $t=10$ (top left to bottom right).  The final image is the approximate steady state. Images are $512\times 512$ pixels.}
    \label{fig:m-3-attraction-surface}
\end{figure}

\vspace{0.5cm}

\subsubsection{Incompressible projections and flows}\label{subsection:incompressible-exp} In our last set of experiments, we consider incompressible flows, which have applications to crowd motion models and fluid mechanics. Here the energy takes the form
\begin{equation}\label{eq:incompressible_energy}
    U(\rho) = s_\infty(\rho) + \int_\Omega V(x) \rho(x) \, dx,
\end{equation}
where
\begin{equation*}
    s_\infty(\rho) = \begin{cases}
        0 &\text{if } 0\leq \rho(x) \leq 1 \text{ for a.e. } x\in\Omega,\\
        \infty & \text{otherwise},\\
    \end{cases}
\end{equation*}
and $V$ is a fixed potential function.  Note that $s_\infty(\rho)$ can be seen as the limit of the energy
\[
   s_m(\rho)= \frac{1}{m-1} \int_\Omega  \rho^m(x) \, dx
\]
as $m\rightarrow \infty$.

We will run our experiments, using the potential energy
\begin{equation}\label{eq:incompressible_potential}
    V(x) = \frac{1}{2} \bigl( (x_1-\frac{3}{10})^2 + (x_2-\frac{3}{10})^2 \bigr) + \iota_{\Omega\setminus E}(x)
\end{equation}
where $E$ is a closed set that represents an impenetrable obstacle. We run two simulations using two different obstacles
\[
 E_1 = B_{\frac{1}{4}}(\frac{1}{5},-\frac{1}{5}) \cup B_{\frac{1}{4}}(-\frac{1}{5},\frac{1}{5}) 
 \]
 and
 \[
 E_2 = B_{\frac{1}{10}}(0,\frac{1}{5}) \cup B_{\frac{1}{10}}(0,-\frac{1}{5}) \cup B_{\frac{1}{10}}(\frac{1}{5},0) \cup B_{\frac{1}{10}}(-\frac{1}{5},0),
\]
 where $B_r(x_1,x_2)$ denotes the closed ball of radius $r$ centered at $(x_1,x_2)$. 
In both experiments, we choose an initial density $\rho^{(0)}$, which equals $1$ on a ball of a radius $0.15$ centered at $(-0.3,-0.3)$ and is equal to $0$ elsewhere.   

The results of our experiments are displayed in Figures~\ref{fig:incompressible-case} and~\ref{fig:incompressible-case-around-ball}.  Figure~\ref{fig:incompressible-case} uses the obstacle $E_1$, while Figure~\ref{fig:incompressible-case-around-ball} uses the obstacle $E_2$.  
In the figures, the yellow pixels represent the density $\rho^{(n)}$ and white pixels represents the obstacle. In both experiments we use a time step $\tau = 0.05$ and run the evolution from time $t=0$ to time $t=20$.   Both experiments are conducted on $1024 \times 1024$ pixel grids.

Notably, in both of the simulations depicted in Figures~\ref{fig:incompressible-case} and~\ref{fig:incompressible-case-around-ball},  there is a sharp interface separating the regions $\rho=1$ and $\rho=0$.  This matches the expected behavior of the flow with our chosen potentials.  In general, it is difficult for numerical methods to correctly capture sharp interfaces.  Again, the reason that our method is able to do so is because of our dual approach.   By recovering the density through the duality relation $\rho^{(n+1)}\in \delta U^*(\phi^{(n+1)})$ we automatically produce a discontinuity at the level set $\{y\in\Omega: \phi^{(n+1)}(y)-V(y)=0\}$.

\begin{figure}[h]
    \begin{minipage}[b]{0.2\linewidth}
      \centering
      \includegraphics[width=.99\linewidth]{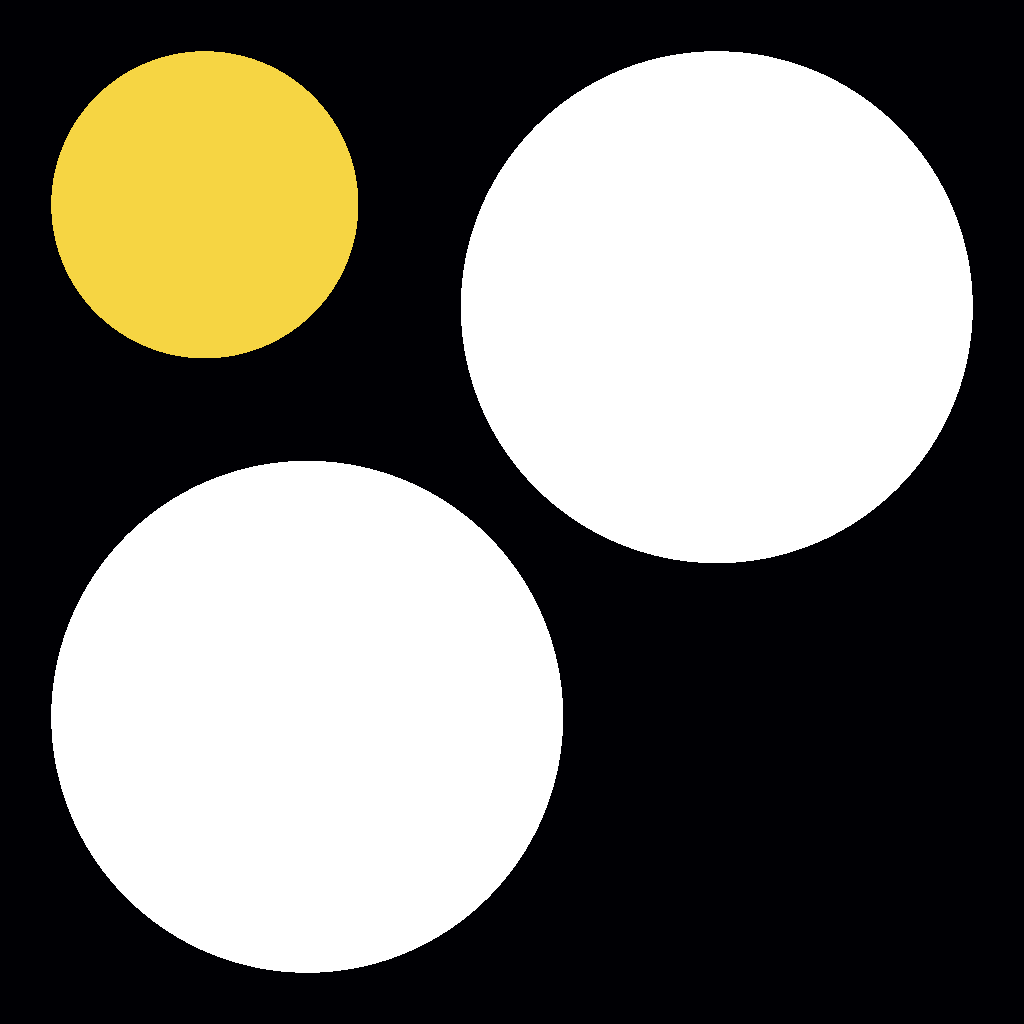}
    \end{minipage}\hfil
    \begin{minipage}[b]{0.2\linewidth}
      \centering
      \includegraphics[width=.99\linewidth]{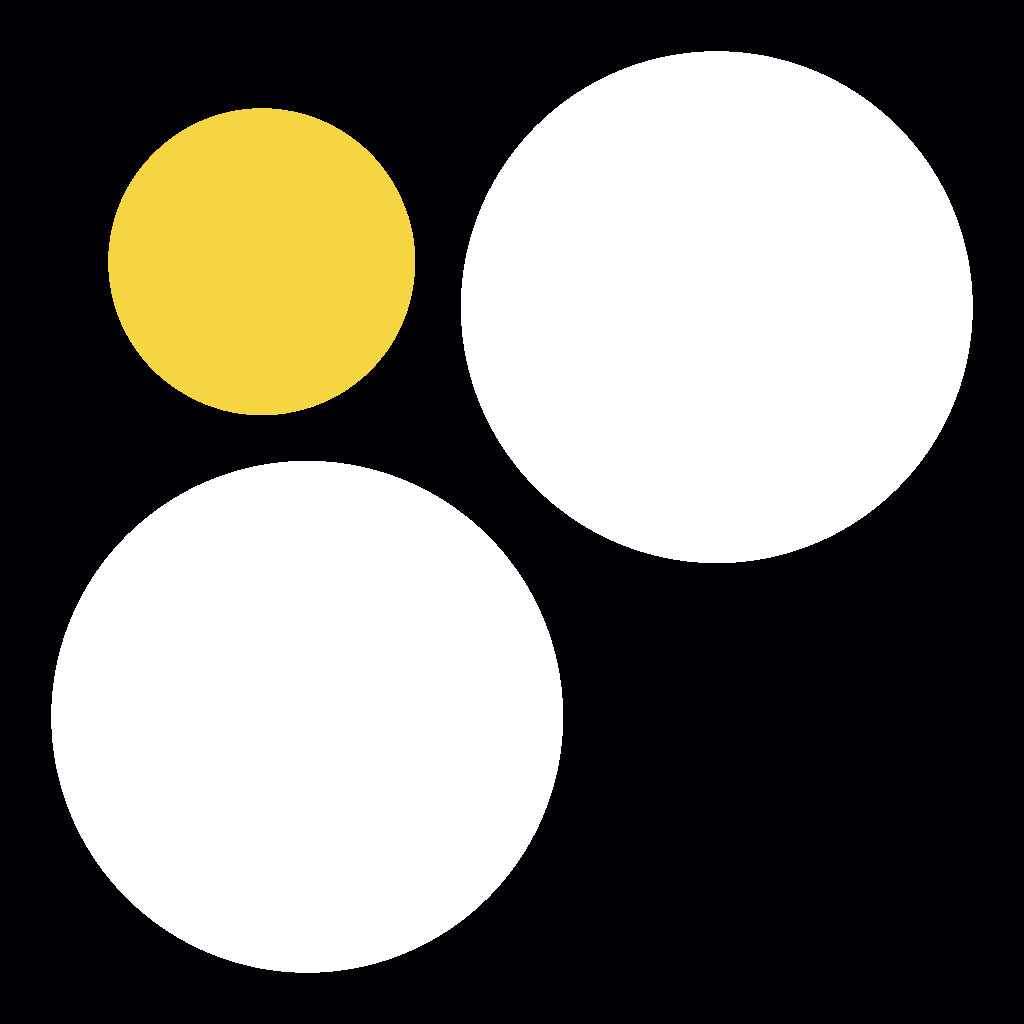}
    \end{minipage}\hfil
    \begin{minipage}[b]{0.2\linewidth}
      \centering
      \includegraphics[width=.99\linewidth]{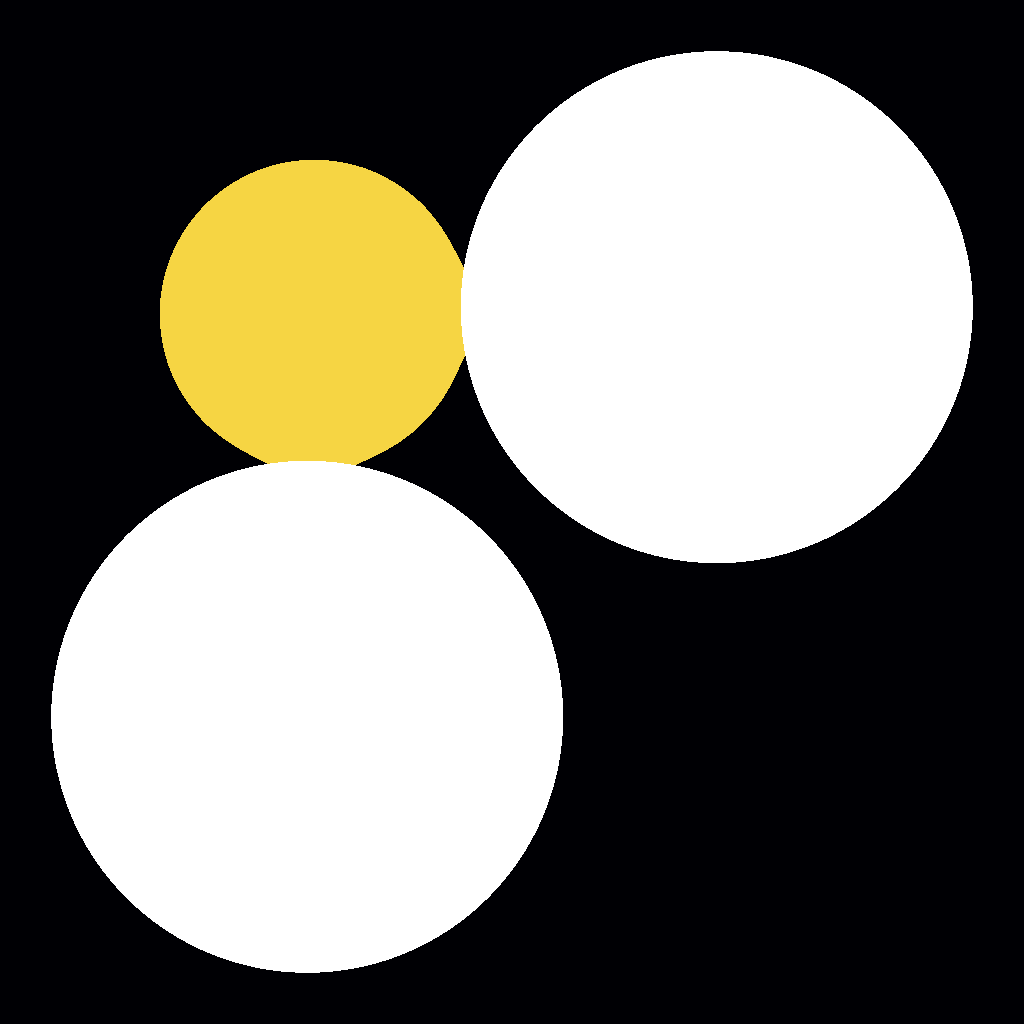}
    \end{minipage}\hfil
    \begin{minipage}[b]{0.2\linewidth}
      \centering
      \includegraphics[width=.99\linewidth]{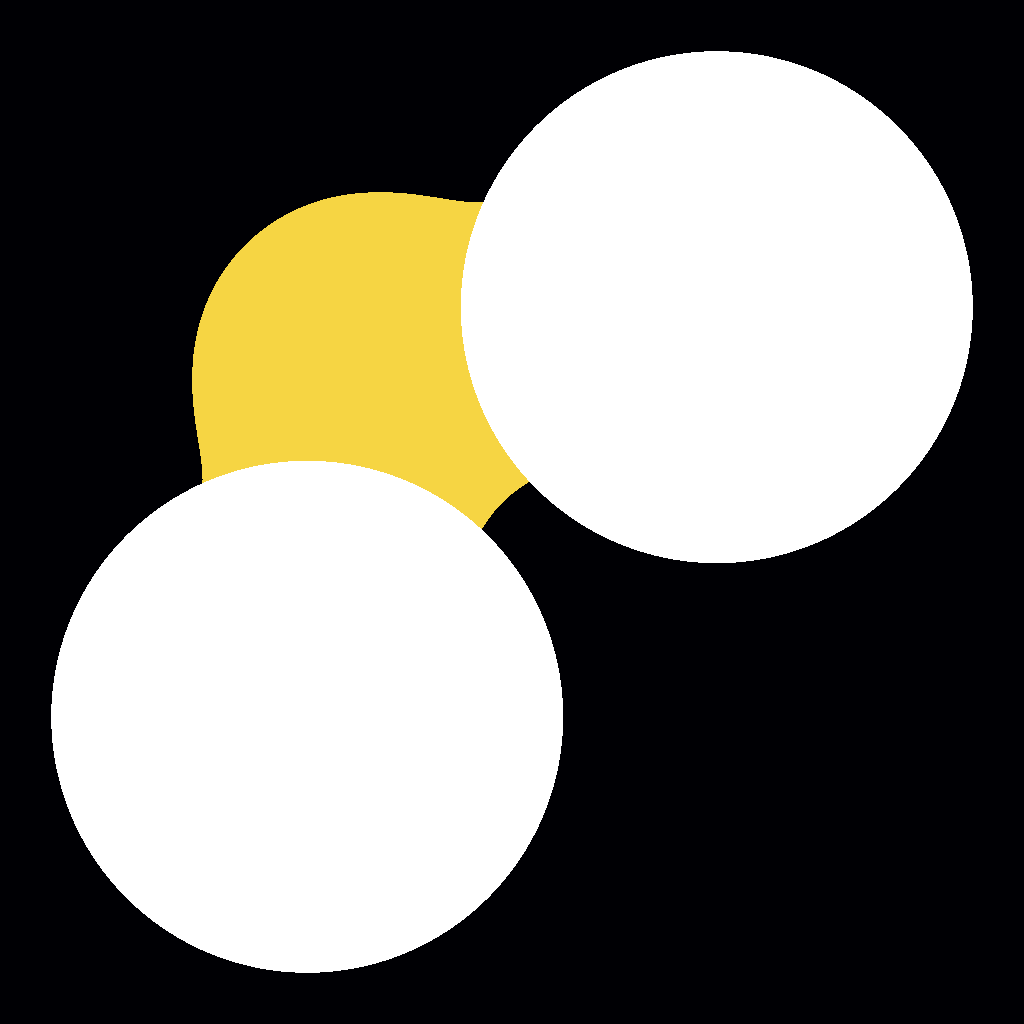}
    \end{minipage}\hfil
    \begin{minipage}[b]{0.2\linewidth}
      \centering
      \includegraphics[width=.99\linewidth]{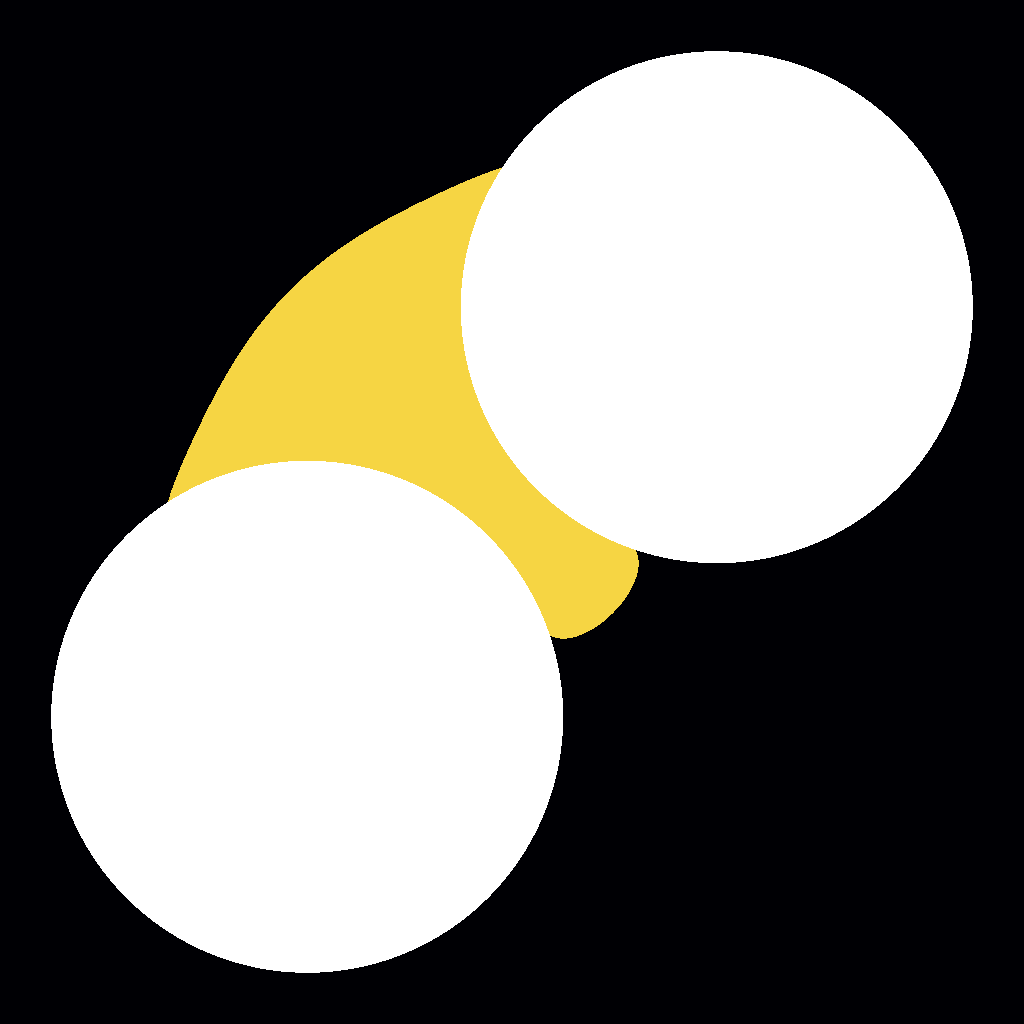}
    \end{minipage}\hfil
    \medskip
    \begin{minipage}[b]{0.2\linewidth}
      \centering
      \includegraphics[width=.99\linewidth]{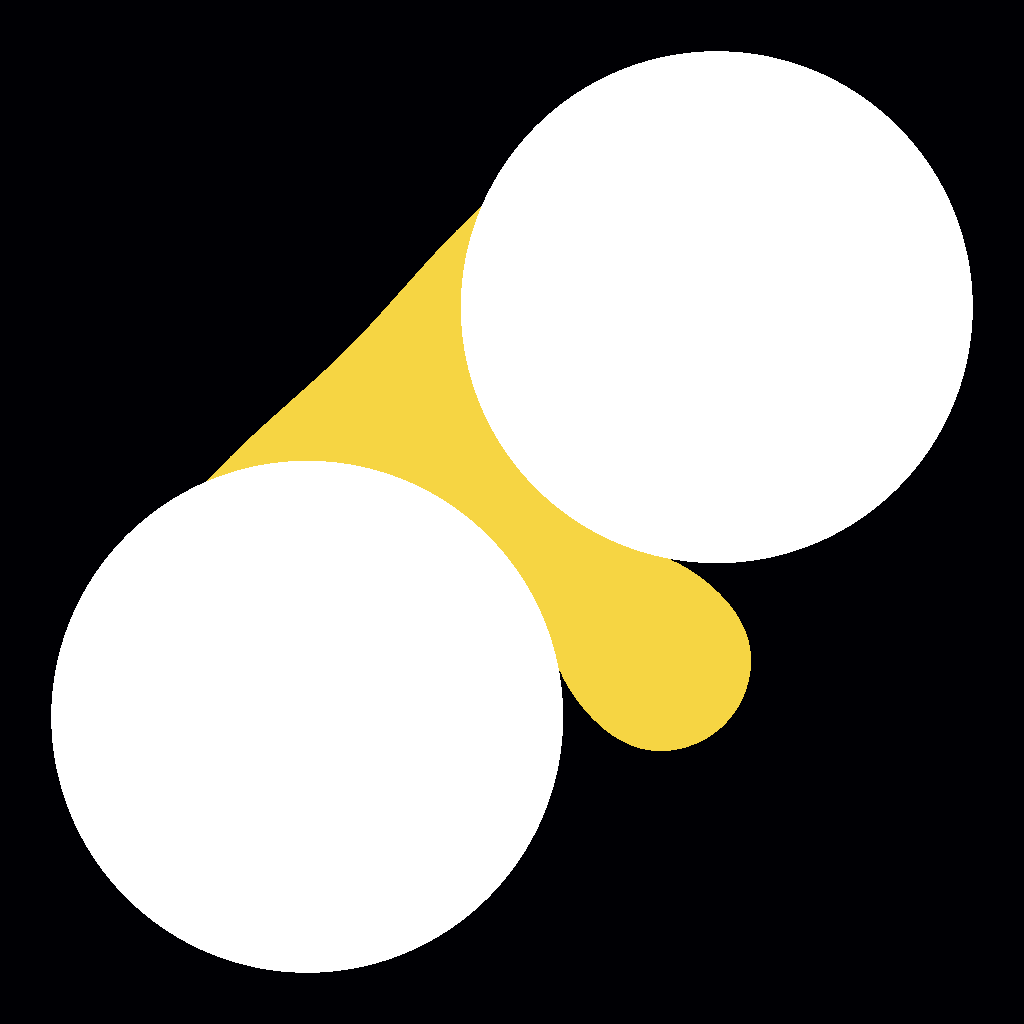}
    \end{minipage}\hfil
    \begin{minipage}[b]{0.2\linewidth}
      \centering
      \includegraphics[width=.99\linewidth]{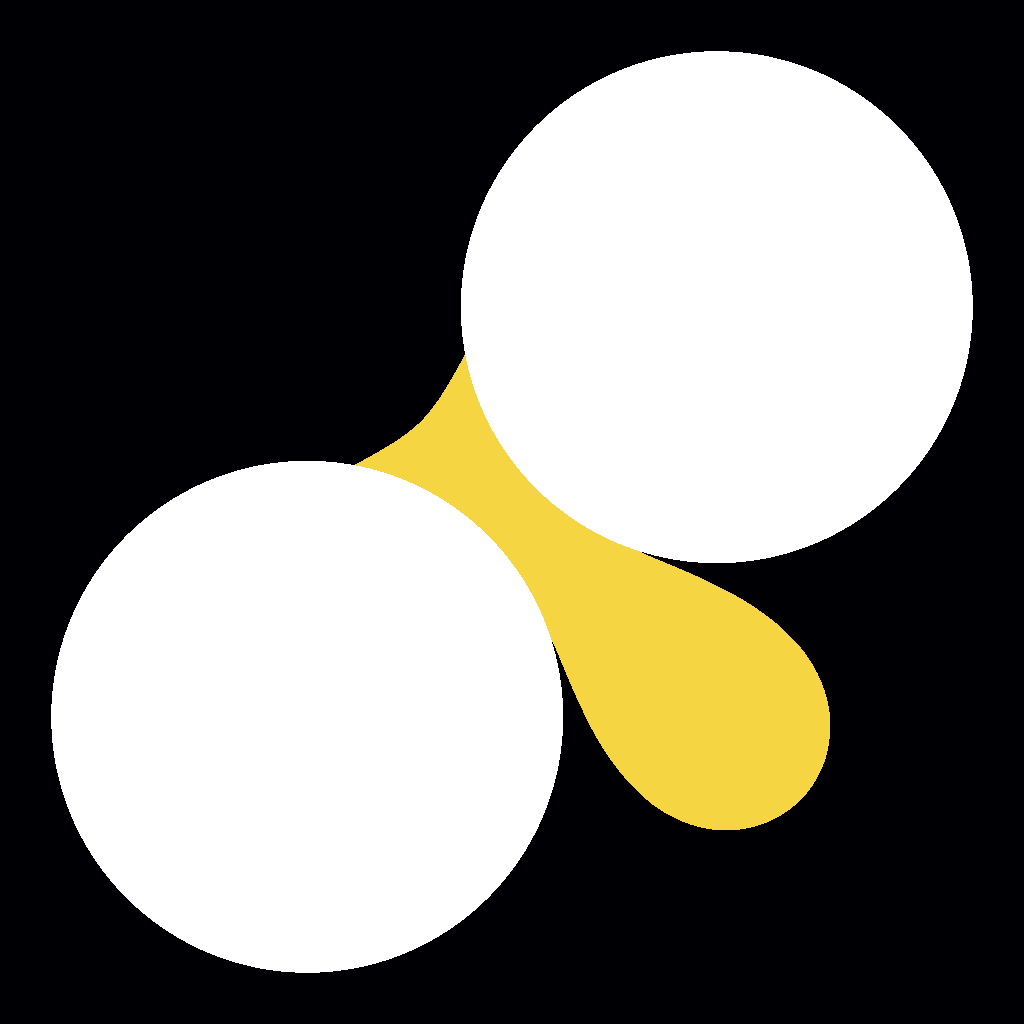}
    \end{minipage}\hfil
    \begin{minipage}[b]{0.2\linewidth}
      \centering
      \includegraphics[width=.99\linewidth]{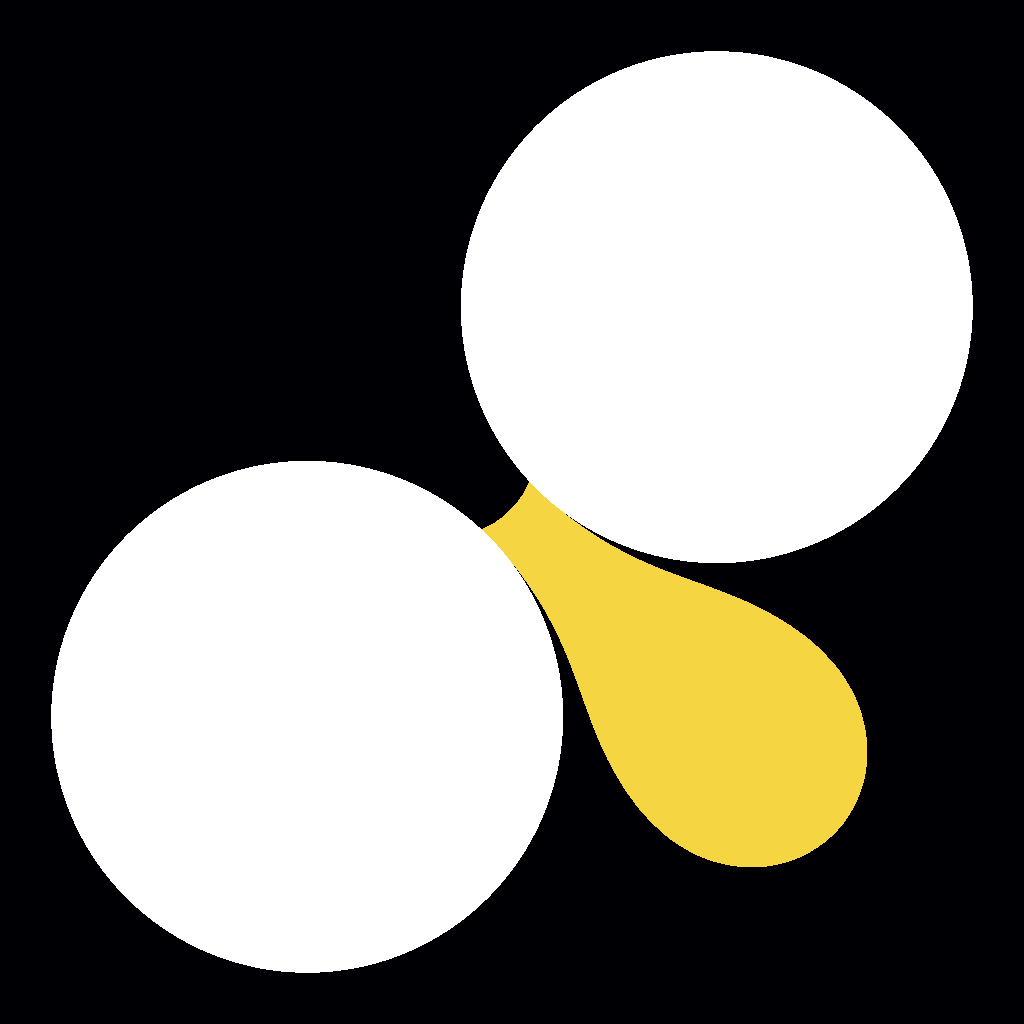}
    \end{minipage}\hfil
    \begin{minipage}[b]{0.2\linewidth}
      \centering
      \includegraphics[width=.99\linewidth]{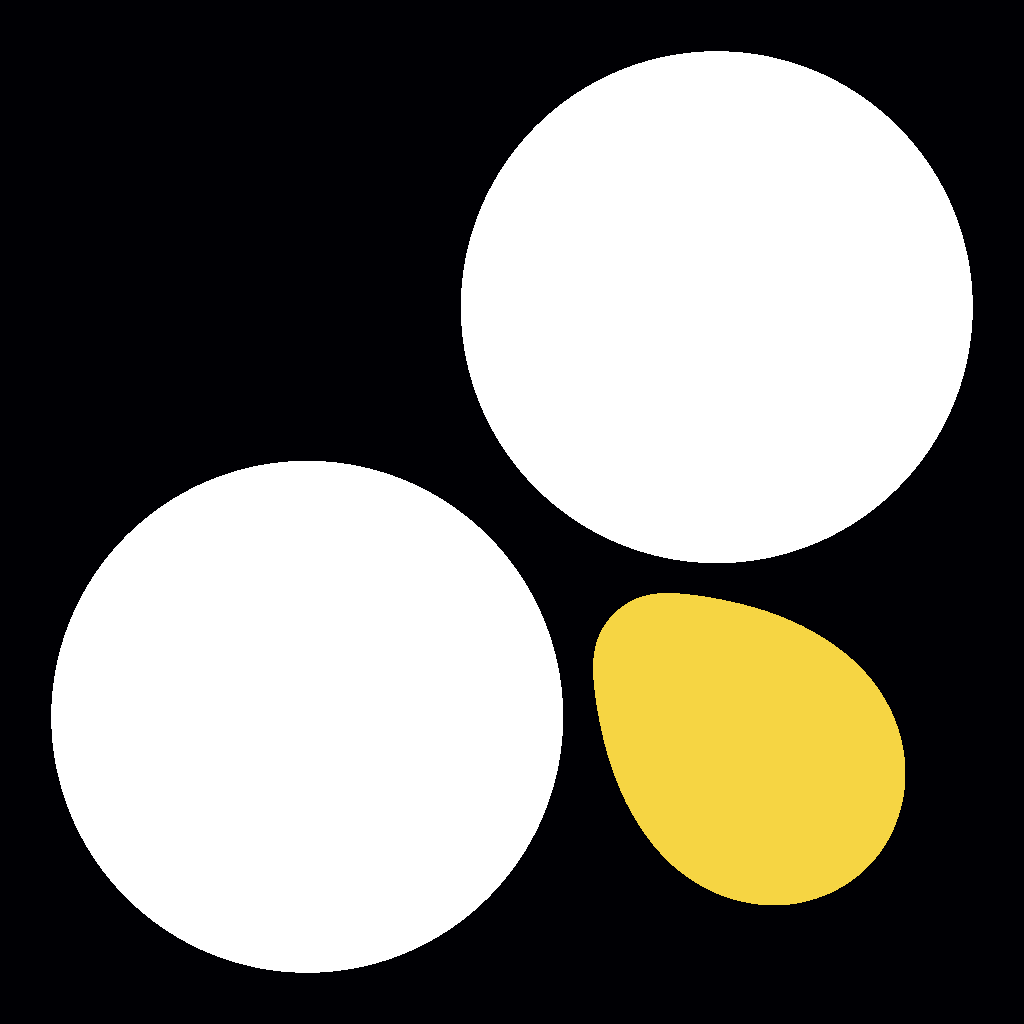}
    \end{minipage}\hfil
    \begin{minipage}[b]{0.2\linewidth}
      \centering
      \includegraphics[width=.99\linewidth]{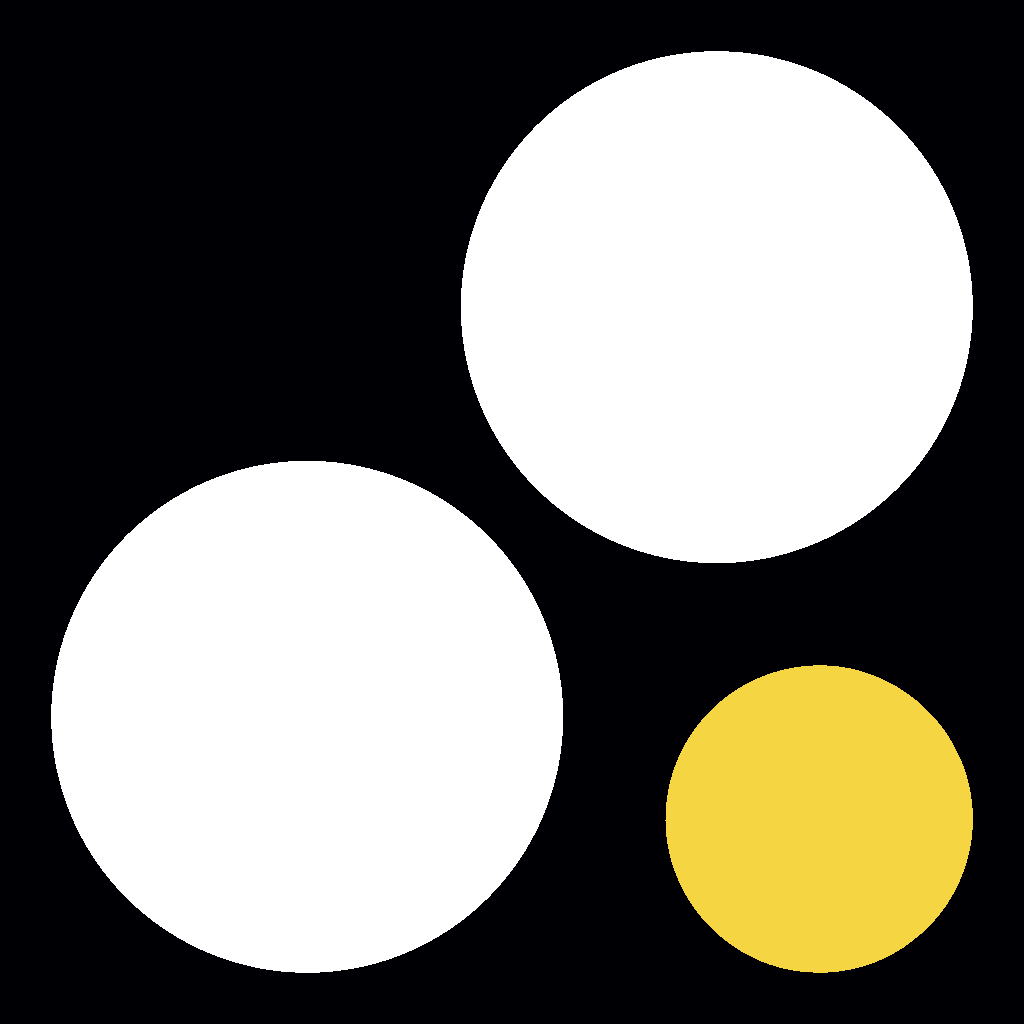}
    \end{minipage}\hfil
    \medskip
    \caption{\footnotesize Incompressible flow with the energy (\ref{eq:incompressible_energy}),  potential (\ref{eq:incompressible_potential}), and obstacle $E_1$.  The images show the evolution from time $t=0$ to $t=20$ (top left to bottom right).  The final image is the approximate steady state. Images are $1024\times 1024$ pixels. Yellow pixels represents the density and white pixels represents the obstacle.}
    \label{fig:incompressible-case}
\end{figure}

\begin{figure}[h]
    \begin{minipage}[b]{0.2\linewidth}
      \centering
      \includegraphics[width=.99\linewidth]{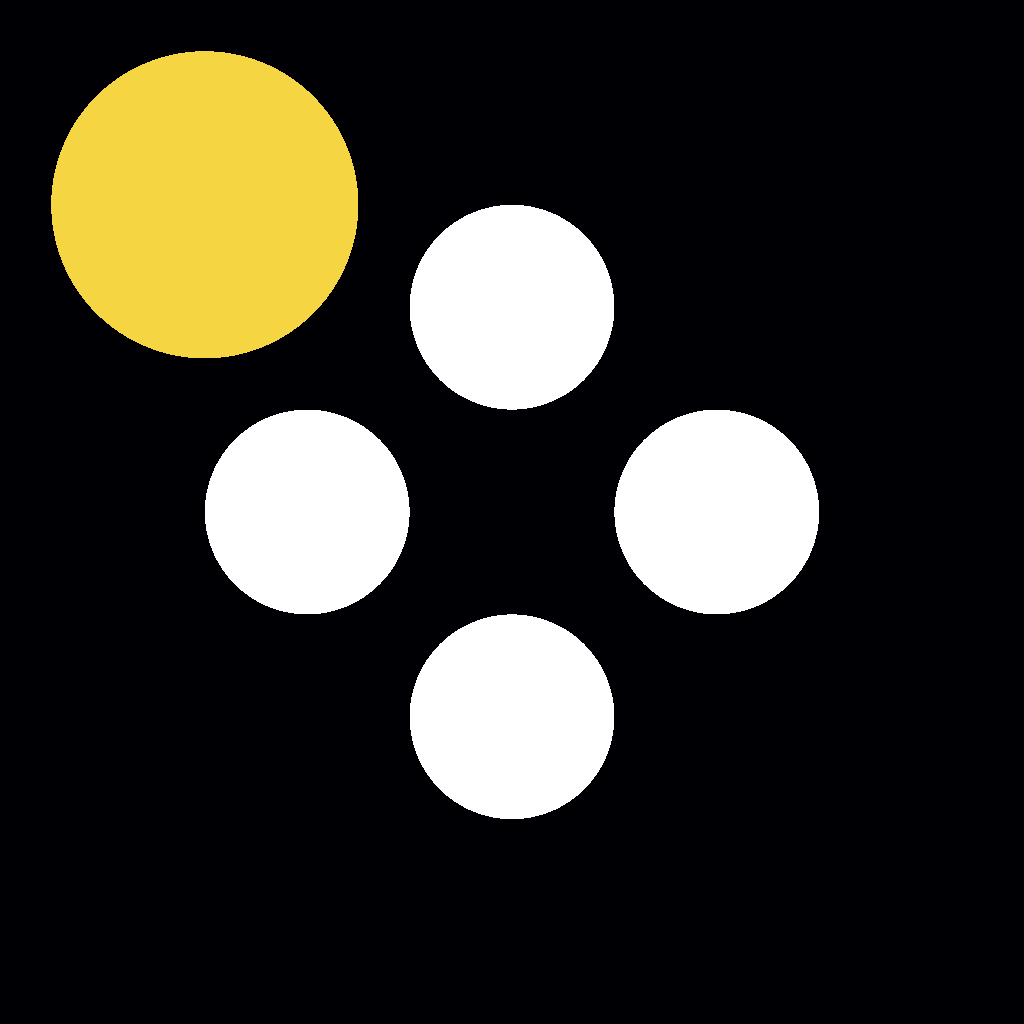}
    \end{minipage}\hfil
    \begin{minipage}[b]{0.2\linewidth}
      \centering
      \includegraphics[width=.99\linewidth]{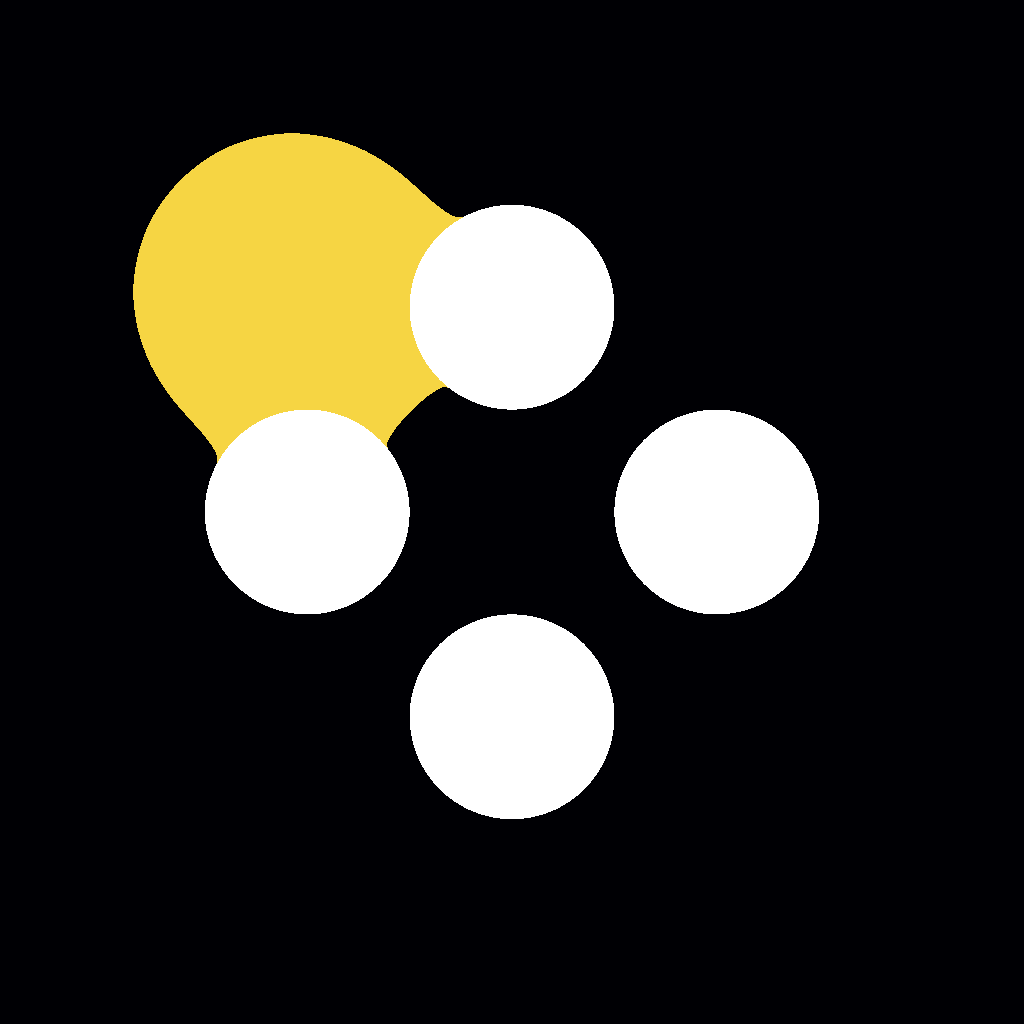}
    \end{minipage}\hfil
    \begin{minipage}[b]{0.2\linewidth}
      \centering
      \includegraphics[width=.99\linewidth]{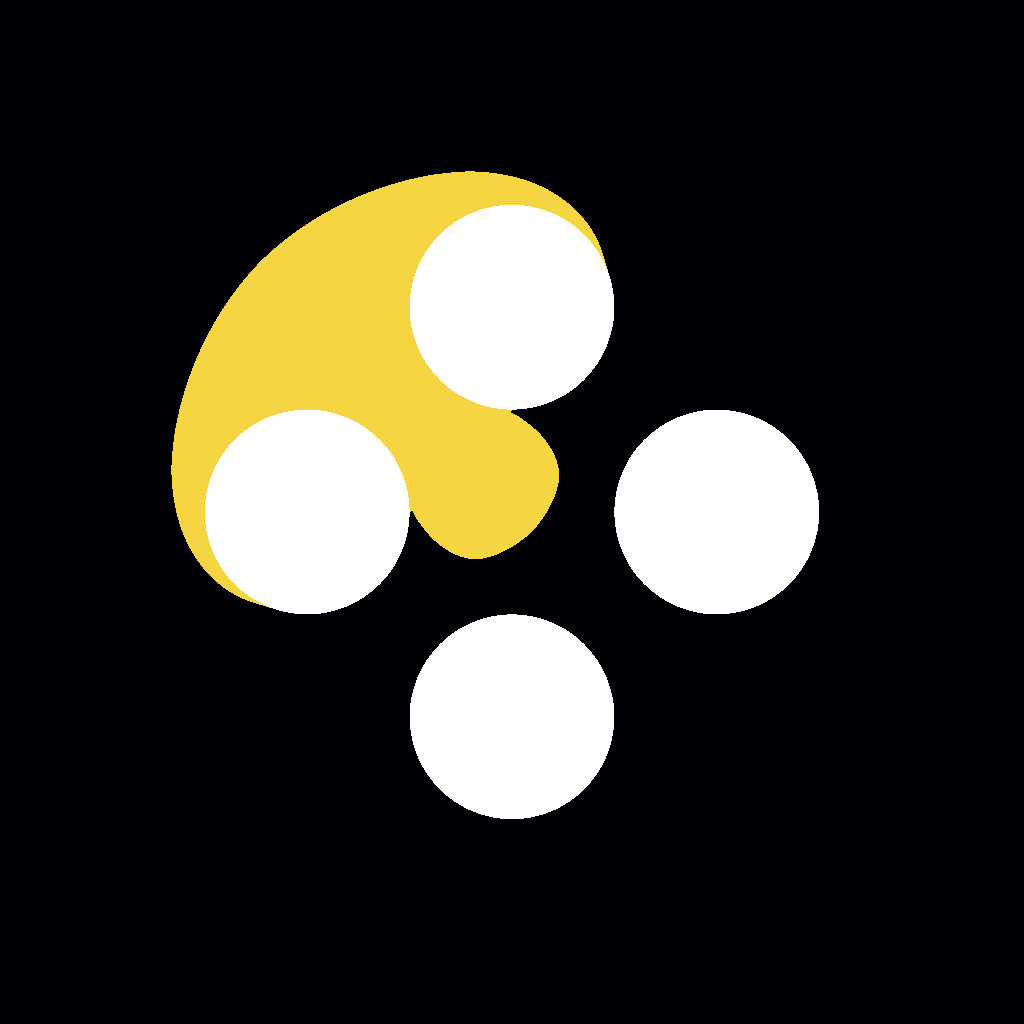}
    \end{minipage}\hfil
    \begin{minipage}[b]{0.2\linewidth}
      \centering
      \includegraphics[width=.99\linewidth]{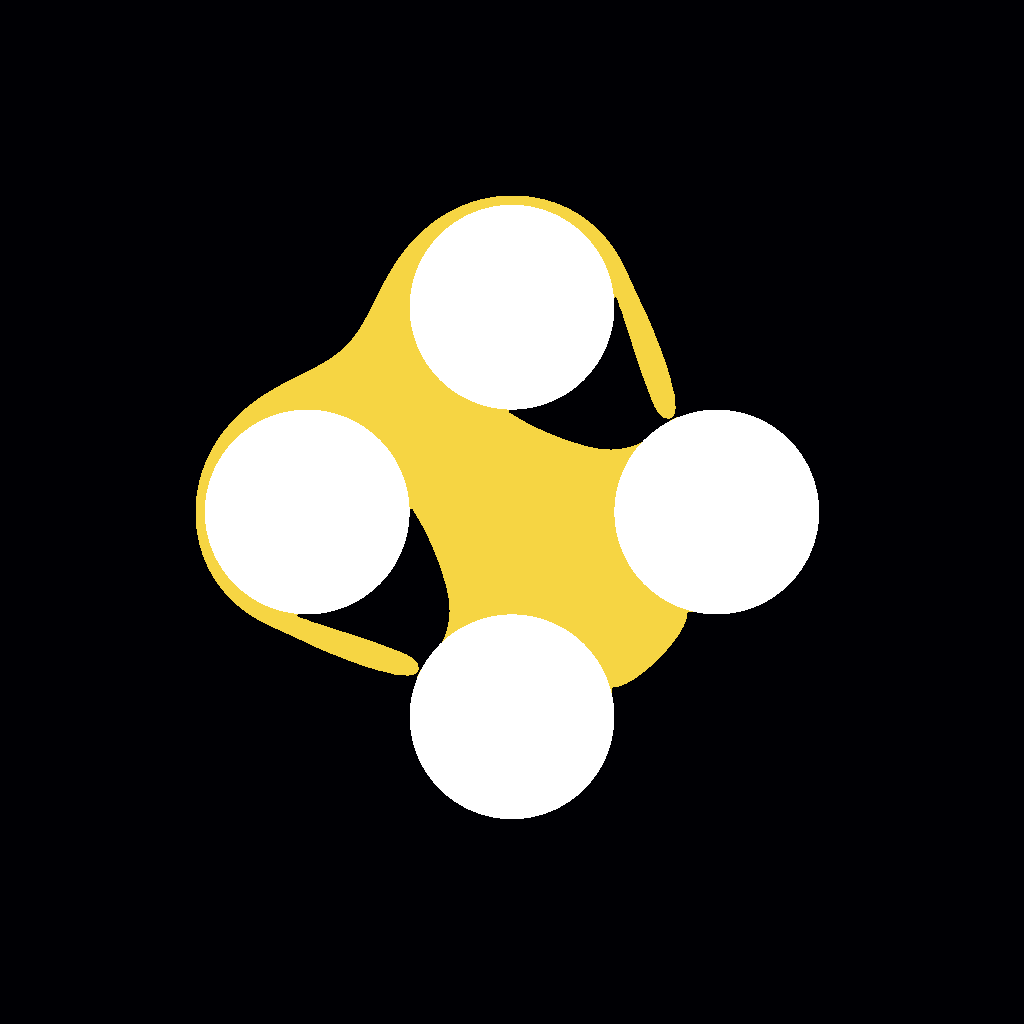}
    \end{minipage}\hfil
    \begin{minipage}[b]{0.2\linewidth}
      \centering
      \includegraphics[width=.99\linewidth]{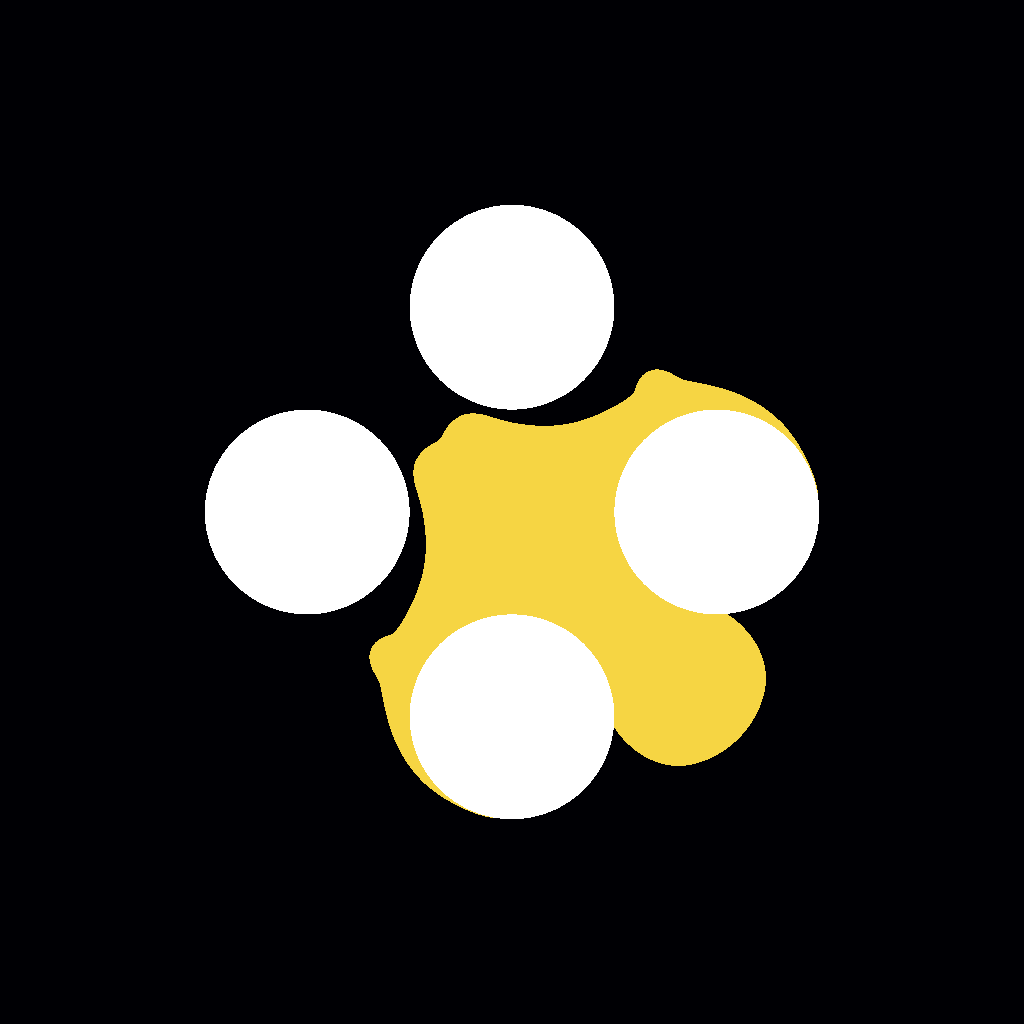}
    \end{minipage}\hfil
    \medskip
    \begin{minipage}[b]{0.2\linewidth}
      \centering
      \includegraphics[width=.99\linewidth]{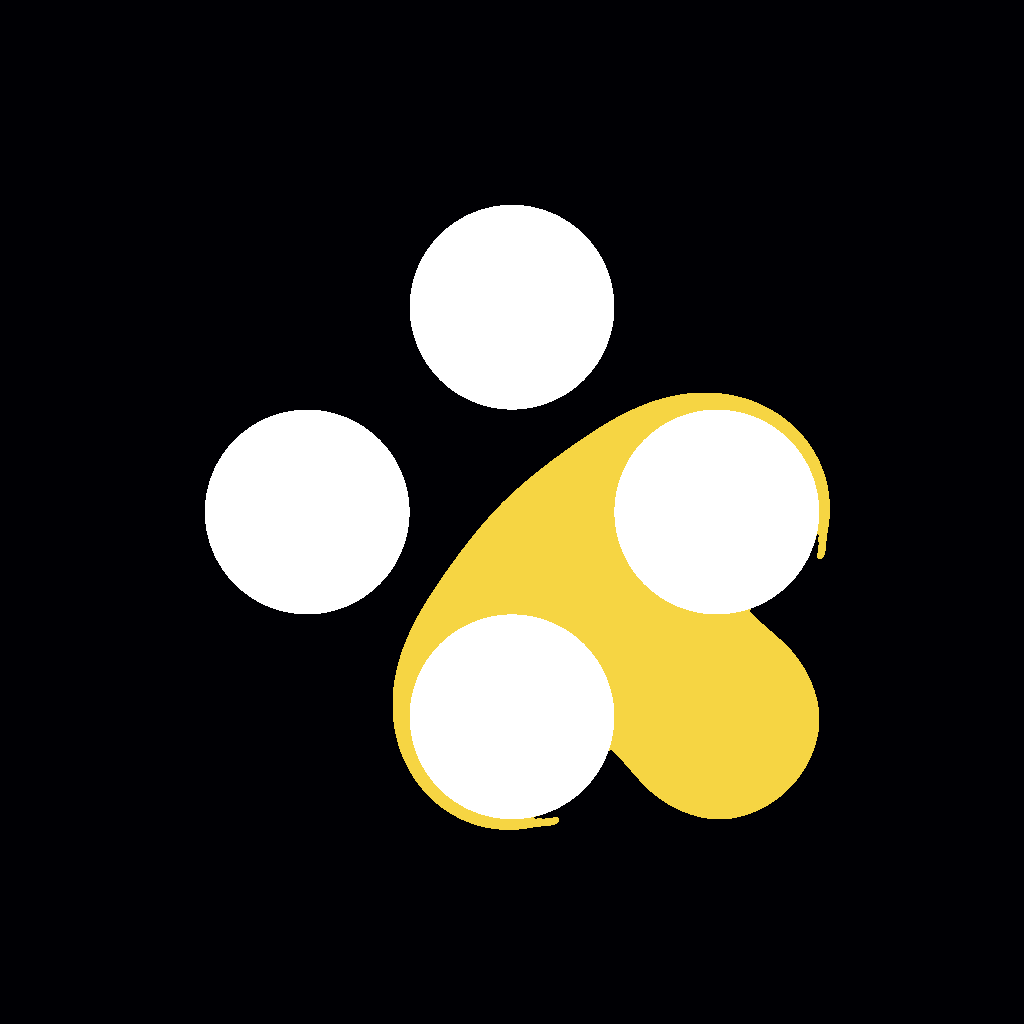}
    \end{minipage}\hfil
    \begin{minipage}[b]{0.2\linewidth}
      \centering
      \includegraphics[width=.99\linewidth]{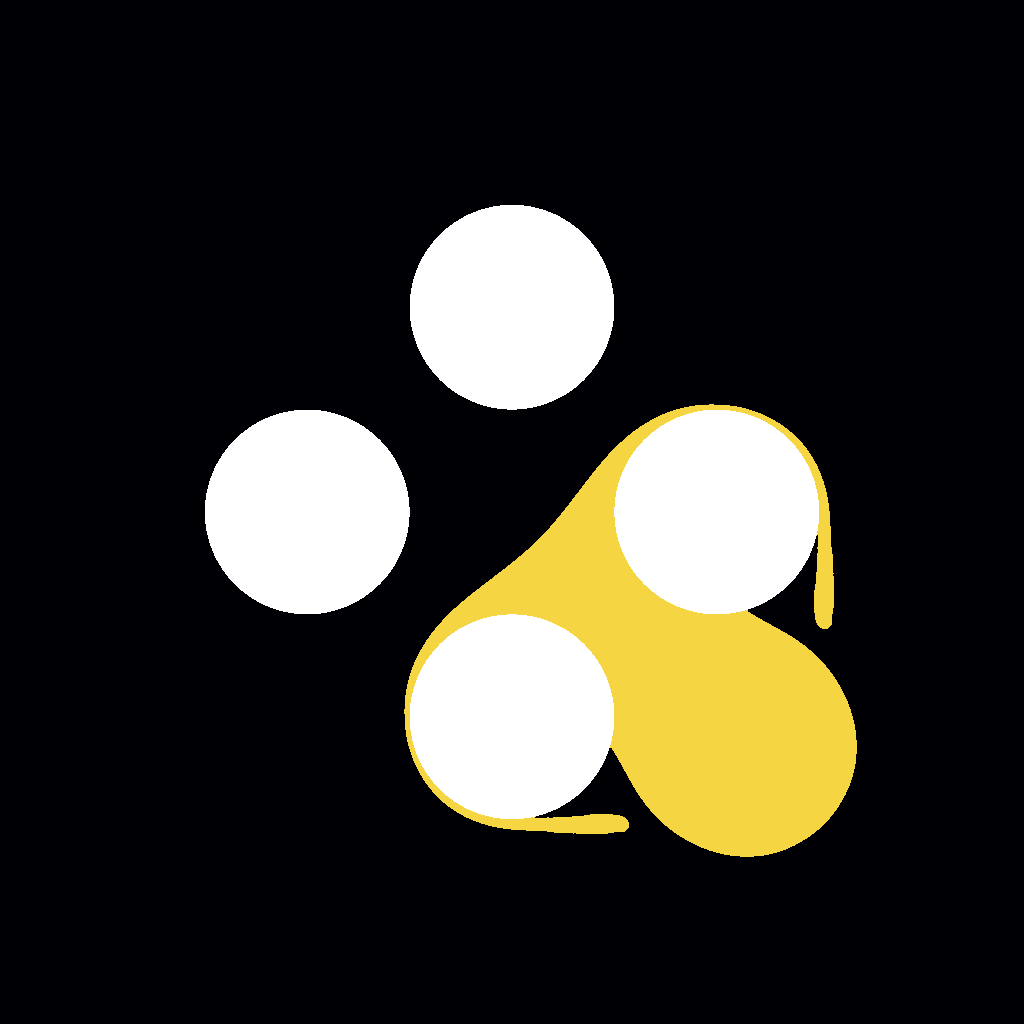}
    \end{minipage}\hfil
    \begin{minipage}[b]{0.2\linewidth}
      \centering
      \includegraphics[width=.99\linewidth]{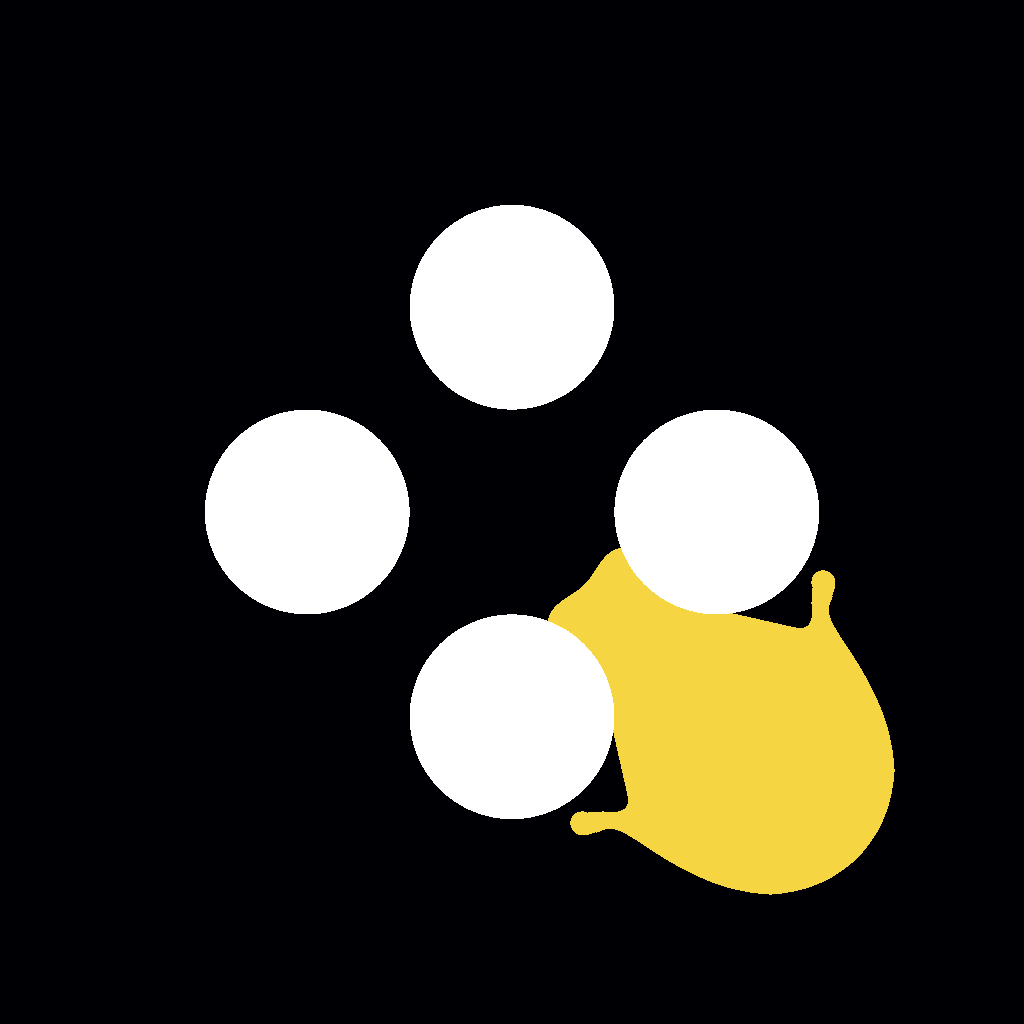}
    \end{minipage}\hfil
    \begin{minipage}[b]{0.2\linewidth}
      \centering
      \includegraphics[width=.99\linewidth]{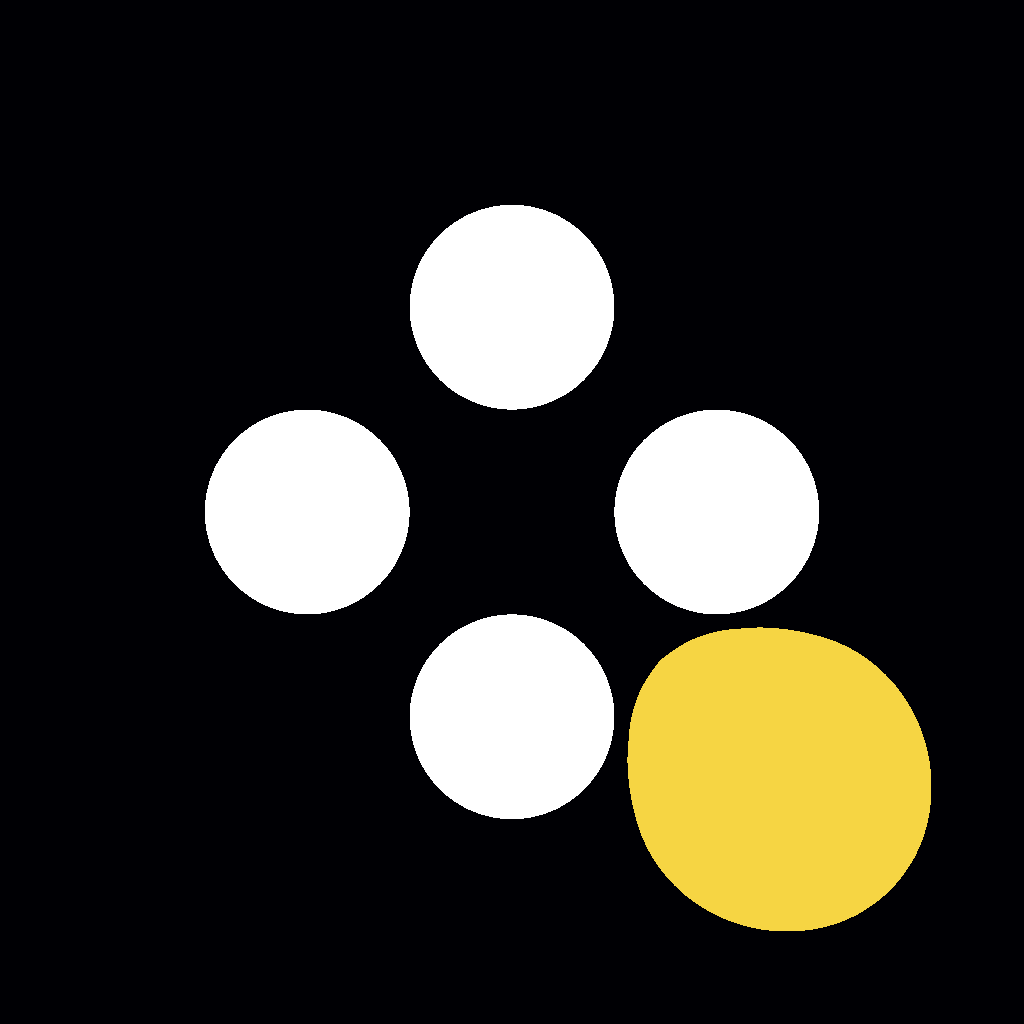}
    \end{minipage}\hfil
    \begin{minipage}[b]{0.2\linewidth}
      \centering
      \includegraphics[width=.99\linewidth]{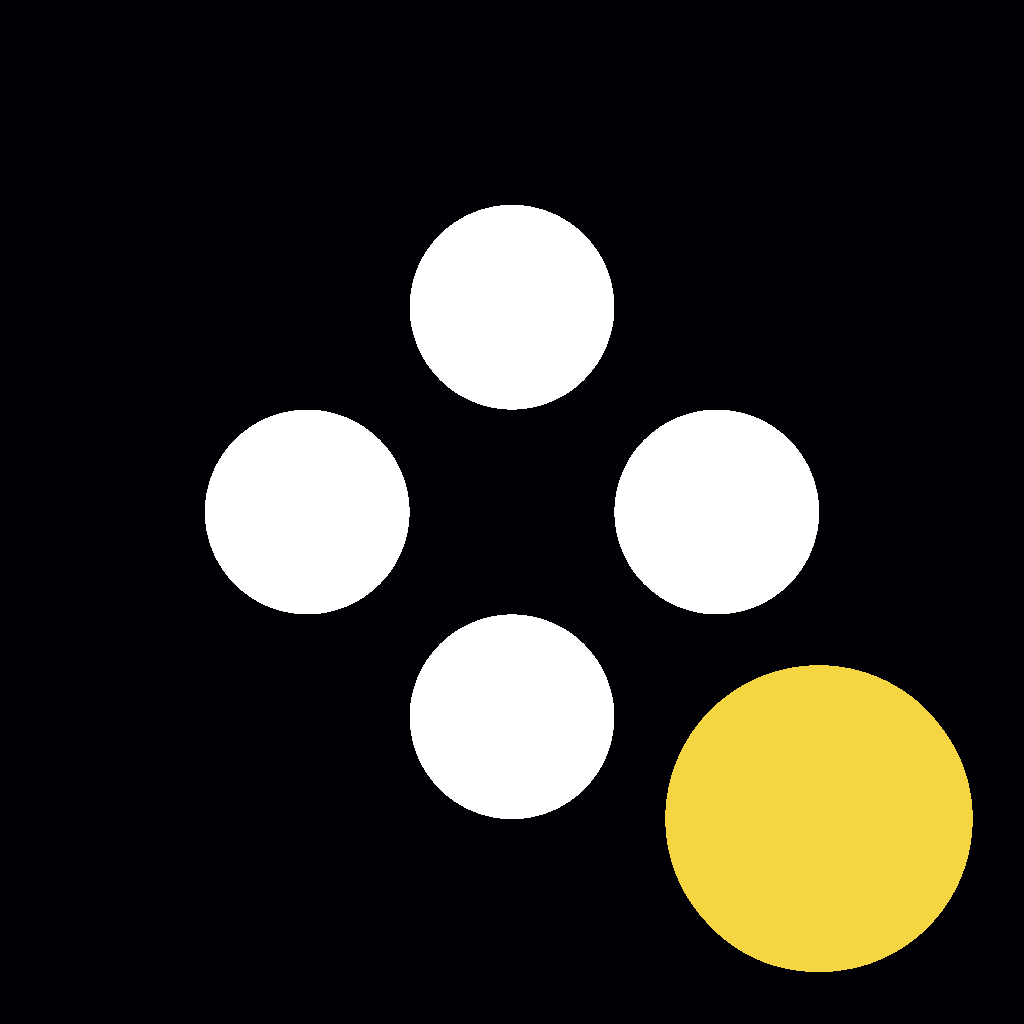}
    \end{minipage}\hfil
    \medskip
    \caption{\footnotesize Incompressible flow with the energy (\ref{eq:incompressible_energy}),  potential (\ref{eq:incompressible_potential}), and obstacle $E_2$.  The images show the evolution from time $t=0$ to $t=20$ (top left to bottom right).  The final image is the approximate steady state. Images are $1024\times 1024$ pixels. Yellow pixels represents the density and white pixels represents the obstacle.}
    \label{fig:incompressible-case-around-ball}
\end{figure}

\appendix

\section{Proofs}

\begin{proof}[Proof of Lemma~\ref{lemma:hessian-F}]
\leavevmode

\textit{Step 1: Derivation of the Hessian.} In order to obtain the Hessian of $F$ let us start with the first derivative. We have
\[
F(\phi+h)-F(\phi)=\int_\Om \big[(\phi+h)^c(x)-\phi^c(x)\big]\,\mu(x)dx.
\]
Assume that $\phi$ is $c$-convex. Then Proposition~\ref{prop:c_transform_variation} tells us how to differentiate the $c$-transform, so that we may write
\[
\int_\Om \big[(\phi+h)^c(x)-\phi^c(x)\big]\,\mu(x)dx = \int_\Om h(T_\phi(x))\,\mu(x)dx + o(h).
\]
Therefore $\delta F(\phi)(h)=\int_\Om h(T_\phi(x))\,\mu(x)dx$. To derive the Hessian of $F$ we similarly compute
\[
\delta F(\phi+h)(h)-\delta F(\phi)(h) = \int_\Om \big[h(T_{\phi+h}(x))-h(T_\phi(x))\big]\,\mu(x)dx.
\]
We must now differentiate the maps $T_\phi$ with respect to $\phi$. By Proposition~\ref{prop:c_transform_variation}  we know that $T_\phi(x)=x-\tau\nabla\phi^c(x)$. As a consequence
\begin{align*}
    T_{\phi+h}(x)-T_\phi(x) &= -\tau \nabla [(\phi+h)^c-\phi^c](x) \\
    &= -\tau \nabla (h\circ T_\phi)(x)+o(h) \\
    &= -\tau DT_\phi(x)^T\nabla h(T_\phi(x)) + o(h).
\end{align*}
Note that $DT_\phi=I_{d\times d}-\tau D^2\phi^c$ is a symmetric matrix. We deduce from the above computations that
\begin{multline*}
   \delta F(\phi+h)(h)-\delta F(\phi)(h) = \\
   \int_\Om \nabla h(T_\phi(x))\cdot (-\tau)DT_\phi(x)\nabla h(T_\phi(x))\,\mu(x)dx + o(h), 
\end{multline*}
from which we conclude that 
\begin{equation*}
    \delta^2F(\phi)(h,h) = -\tau \int_\Om \nabla h(T_\phi(x))\cdot DT_\phi(x)\nabla h(T_\phi(x))\,\mu(x)dx.
\end{equation*}
Since our goal is to bound this Hessian by a norm of $h$ we do the change of variable $y=T_\phi(x)$, or equivalently $x=S_{\phi^c}(y)$ since $S_{\phi^c}$ is the inverse of $T_\phi$, see Proposition~\ref{prop:c_transform_variation}. We obtain 
\[
\delta^2F(\phi)(h,h) = -\tau \int_\Om \nabla h(y)\cdot DT_\phi(S_{\phi^c}(y))\nabla h(y)\,\mu(S_{\phi^c}(y)) \det DS_{\phi^c}(y)dy.
\]
Note that $DS_{\phi^c}$ is a positive semi-definite matrix and therefore no absolute value is needed on the determinant term. Moreover we have $DT_\phi(S_{\phi^c}(y))=DS_{\phi^c}(y)^{-1}$ and putting this term together with the determinant we can form the cofactor matrix $\cof(DS)=\det (DS) DS^{-1}$. As a result we obtain the expression
\[
\delta^2F(\phi)(h,h) = -\tau \int_\Om \nabla h(y)\cdot \cof(DS_{\phi^c}(y))\nabla h(y)\,\mu(S_{\phi^c}(y))dy.
\]

\textit{Step 2: Hessian bounds.} Since $\phi$ is $c$-convex, $\phi=\phi^{c\bar c}$ and therefore $S_{\phi^c}(y)=y+\tau\nabla \phi(y)$. The $c$-convexity of $\phi$ also implies that the symmetric matrix $DS_{\phi^c}(y)=I_{d\times d}+\tau D^2\phi(y)$ is positive semi-definite. Assume now that $I_{d\times d}+\tau D^2\phi(y)\le\Lambda \,I_{d\times d}$ for all $y\in\Om$. Then $I_{d\times d}+\tau D^2\phi(y)$ is a symmetric matrix with eigenvalues between $0$ and $\Lambda$. By general properties of the cofactor matrix the eigenvalues of $\cof(DS_{\phi^c}(y))$ lie between $0$ and $\Lambda^{d-1}$ where $d$ is the space dimension. We immediately deduce 
\[
-\delta^2F(\phi)(h,h) \le \tau \Lambda^{d-1} \normlinf{\mu} \int_\Om \abs{\nabla h(y)}^2\,dy.
\]

\end{proof}

\begin{lemma}\label{lem:trace_dist}
Suppose that $E\subset \R^d$ is a bounded set with $C^2$ boundary and let $R:=\textrm{Reach}(\partial E)$.
Let $u_0:\R^d\to \R$ be a solution to the Eikonal equation $|\nabla u_i|=1$ where $u_0<0$ inside $E$ and $u_0>0$ outside $E$ and set $u_1=-u_0$. Let 
\[
E_r^0=\{x\in \R^d: u_0(x)\in (0, r)\}.
\]
and 
\[
E_r^1=\{x\in \R^d: u_1(x)\in (0, r)\}.
\]
If $g:\R^d\to\R$ is a smooth function, then for $i=0,1$
\[
\int_{\partial E} |g(x)|ds(x)\leq \inf_{0<r<R}
\Big(\int_{E_r^{i}} |\nabla g(x)|\, dx+C_i(E,r)\int_{E_r^{i}} |g(x)|\, dx\Big)
\]
where 
\[
C_i(E,r)=\inf_{0<r'<r} \frac{1}{r'}+\sup_{x\in E_{r'}^i} (\Delta u_i(x))_+
\]
\end{lemma}
\begin{remark}
The reach of $\partial E$ is the largest number $r$ such that the characteristics of $u_0$ do not cross in $E_r^0\cup E_r^1$.  When $\partial E$ is $C^2$, the reach must be strictly positive and the Laplacian $\Delta u$ must be bounded on $E_r^0\cup E_r^1$ for all $r$ smaller than the reach of $\partial E$.    
\end{remark}
\begin{remark}
If $E$ is a convex set, then  $C_1(E,r)=\frac{1}{r}$.
\end{remark}

\begin{proof}
Note that if $x\in \partial E$ and $n(x)$ is the outward facing normal at $x$, then $\nabla u_0(x)=n(x)$.  Therefore,
\[
\int_{\partial E} |g(x)|\,ds(x)=\int_{\partial E} |g(x)|\nabla u_0(x)\cdot n(x)\, ds(x)
\]
For some $r\in (0, R)$ let $\alpha_r:\R\to\R$ be a function such that 
\[
\alpha_r'(t)=\begin{cases}
1 & \textrm{if}\; t\geq 0,\\
1+\frac{t}{r}& \textrm{if} \;t\in (-r,0),\\
0 &\textrm{if} \; t\leq -r.\\
\end{cases}
\]
We then have 
\[
\int_{\partial E} |g(x)|\nabla u_0(x)\cdot n(x)\, ds(x)=\int_{\partial E} |g(x)|\nabla \Big(\alpha_r\big(u_0(x)\big)\Big)\cdot n(x)\, ds(x)=
\]
\[
\int_{ E} \nabla \cdot \Big(|g(x)|\nabla \Big(\alpha_r\big(u_0(x)\big)\Big)\Big)\, dx
\]
where the last equality follows from Stokes Theorem. 
Expanding out the derivatives and noting that $\alpha_r'(t)\in [0,1]$, $\alpha''_r(t)\in [0,\frac{1}{r}]$ and $\alpha'(u_0(x)), \alpha''(u_0(x))$ both vanish for $x$ outside of $E_r^0$, we get
\[
\int_{\partial E} |g(x)|ds(x)\leq  \int_{E_r^0} |\nabla g(x)|\, dx+\int_{E_r^0} |g(x)|\big((\Delta u_0(x))_++\frac{1}{r}\big)\, dx\leq 
\]
\[
\int_{E_r^0} |\nabla g(x)|\, dx+C_0(E,r)\int_{E_r^0} |g(x)|\, dx
\]
Our choice of $r$ was arbitrary, thus we can take an inf over $r\in (0,R)$ to conclude the result when $i=0$.

To tackle the case $i=1$, we will employ a nearly identical argument, except we will use Stokes Theorem to convert the boundary integral into an integral over $\R^d\setminus E$.  Since $\nabla u_1(x)\cdot n(x)=-1$ for $x\in\partial E$, we have
\[
\int_{\partial E} |g(x)|\, ds(x)=-\int_{\partial E} |g(x)|\nabla \Big(\alpha_r\big(u_1(x)\big)\Big)\cdot n(x)\, ds(x)=
\]
\[
\int_{\R^d\setminus E} \nabla \cdot \Big(|g(x)|\nabla \Big(\alpha_r\big(u_1(x)\big)\Big)\Big)\, dx
\]
Now an identical argument to the one above gives the bound for the case $i=1$.
\end{proof}
\begin{corollary}\label{cor:trace}
Suppose that $E\subset \Omega$ is a set with $C^2$ boundary and let $R:=\min(\textrm{Reach}(\partial E),  \textrm{dist}(E,\partial\Omega))$. Define $u_i$, $E_r^i$, and $C_i(E,r)$ as in Lemma~\ref{lem:trace_dist}, and let
\[
C(E,\Omega)=\min_{i\in \{0,1\}} \inf_{0<r<R}C_i(E,r).
\]
If $h:\Omega\to\R$ is an $H^1$ function,
then
\[
\int_{\partial E} |h(x)|^2\, dx\leq \frac{1}{C}\int_{\Omega} |\nabla h(x)|^2+2C\int_{\Omega} |h(x)|^2\, dx,
\]
where 
\[
C=\max(1, C(E,\Omega)).
\]

\end{corollary}
\begin{proof}
Suppose first that $h:\Omega\to\R$ is a smooth function.  By Lemma~\ref{lem:trace_dist}, we have 
\[
\int_{\partial E} |h(x)|^2\, ds(x)\leq \inf_{0<r<R}
\Big(\int_{E_r^{i}} 2|h(x)\nabla h(x) |\, dx+C_i(E,r)\int_{E_r^{i}} |h(x)|^2\, dx\Big)
\]
for $i=0,1$.  
Clearly this is bounded from above by 
\[
 \int_{\Omega} 2|h(x)\nabla h(x) |\, dx+\inf_{0<r<R}C_i(E,r)\int_{\Omega} |h(x)|^2\, dx
\]
Taking a minimum over $i=0,1$, we can conclude that 
\[
\int_{\partial E} |h(x)|^2\, ds(x)\leq \int_{\Omega} 2|h(x)\nabla h(x) |\, dx+C\int_{\Omega} |h(x)|^2\, dx.
\]
We can then use Cauchy-Schwarz to get 
\[
\int_{\partial E} |h(x)|^2\, ds(x)\leq \frac{1}{C}\int_{\Omega} |\nabla h(x) |\, dx+2C\int_{\Omega} |h(x)|^2\, dx.
\]
The result extends to $H^1$ functions thanks to the continuity of the trace operator over $H^1$.
\end{proof}

\begin{proof}[Proof of Theorem~\ref{thm:I-smoothness}]
\leavevmode

Recall that $I(\psi)=\int_\Om \psi(x)\,\mu(x)dx - U^*(\psi^{\bar c})$.

\textit{Step 1: formula for the Hessian of $I$.} The derivation of the Hessian of $I$ is similar to the one of $J$ (see for instance the proof of Lemma~\ref{lemma:hessian-F}). Using the formulas for the first variation of the $c$-transform in Proposition~\ref{prop:c_transform_variation} we can check that
\[
\delta I(\psi)h = -\delta U^*(\psi^{\bar c}) (h\circ S_\psi),
\]
for any test function $h$. To obtain the Hessian of $I$, we need to differentiate $S_\psi$. As in the proof of Lemma~\ref{lemma:hessian-F} we can show that $S_{\psi+h}(y)-S_\psi(y)=\tau DS_\psi(y)^T \nabla h( S_\psi(y))+o(h)$. This implies
\begin{multline*}
    \delta^2 I(\psi)(h,h) = -\delta^2U^*(\psi^{\bar c})(h\circ S_\psi,h\circ S_\psi) - \\ \tau \int_\Om \eta(y) \,\nabla h(S_\psi(y))\cdot DS_\psi(y)\nabla h(S_\psi(y))\,dy,
\end{multline*}
where we set $\eta=\delta U^*(\psi^{\bar c})$. Thus as for $J$, the Hessian of $I$ contains two terms which we can bound separately, $\delta^2I(\psi)(h,h) = -(A)-(B)$. 

\textit{Step 2: Bound on $(B)$}. Do the change of variables $x=S_\psi(y)$, i.e. $y=T_{\psi^{\bar c}}(x)$ in $(B)$. We obtain 
\[
(B) = \tau \int_\Om\eta(T_{\psi^{\bar c}}(x)) \nabla h(x)\cdot \cof DT_{\psi^{\bar c}}(x)\nabla h(x)\,dx,
\]
which can be bounded above by $\tau\normlinf{\eta} \Lambda^{d-1} \normltwo{\nabla h}$ in the same spirit as in the proof of Lemma~\ref{lemma:hessian-F}. Moreover $\normlinf{\eta} \le \rhomax$. Indeed, assuming $V(x)\ge 0$ we have for all $x\in\Om$
\[
\eta(x) = \delta U^*(\psi^{\bar c})(x) = (u_m^*)'(\psi^{\bar c}(x)-V(x)) \le  (u_m^*)'(\psi^{\bar c}(x)) \le \rhomax,
\]
by monotonicity of $(u_m^*)'$ and by definition of $\rhomax$.  As a consequence
\[
(B) \le \tau \rhomax \Lambda^{d-1} \normltwo{\nabla h}^2.
\]

\textit{Step 3: Bound on $(A)$}. 
We have
\[
(A) = \delta U^*(\psi^{\bar c})(h\circ S_\psi,h\circ S_\psi) = \int_\Om (u_m^*)''(\psi^{\bar c}(y)-V(y))\abs{h(S_\psi)}^2\,dy.
\]
Do again the change of variables $y=T_{\psi^{\bar c}}(x)$ to obtain
\[
(A) = \int_\Om (u_m^*)''(p(x)) \abs{h(x)}^2 \det(DT_{\psi^{\bar c}}(x))\,dx,
\]
where we recall that $p(x) = \psi^{\bar c}(T_{\psi^{\bar c}}(x)) - V(T_{\psi^{\bar c}}(x))$. We bound the determinant term by $\Lambda^d$. Then, to go further we must distinguish between the three cases $1\le m\le 2$, $2<m<\infty$ and $m=\infty$. 

When $1\le m\le 2$, the function $(u^*_m)''$ is increasing and therefore
\[
(u^*_m)''(p(x)) \le (u^*_m)''(M) = u_m''(\rhomax)^{-1},
\]
where $M=\sup_x\delta U(\mu)(x)$ (see the maximum principle and the related discussion when $\rhomax$ is defined in equation~\eqref{eq:def-rhomax}). To sum up,
\[
(A) \le u''(\rhomax)^{-1} \Lambda^d \normltwo{h}^2.
\]

When $2<m\le \infty$, one can follow the same line of proof as in the case of $J$, using now the function $p(x)$ instead of $\phi(x)-V(x)$ which modifies the related constants accordingly.

\end{proof}

\bibliographystyle{alpha}  
\bibliography{gf} 

\begin{thebibliography}{DPMSV16}

\bibitem[AKY14]{aky2014}
Damon Alexander, Inwon Kim, and Yao Yao.
\newblock Quasi-static evolution and congested crowd transport.
\newblock {\em Nonlinearity}, 27(4):823--858, mar 2014.

\bibitem[Bar96]{barenblatt1996}
Grigory~Isaakovich Barenblatt.
\newblock {\em Scaling, self-similarity, and intermediate asymptotics},
  volume~14 of {\em Cambridge Texts in Applied Mathematics}.
\newblock Cambridge University Press, Cambridge, 1996.
\newblock With a foreword by Ya. B. Zeldovich.

\bibitem[Bar03]{barenblatt2003}
Grigory~Isaakovich Barenblatt.
\newblock {\em Scaling}.
\newblock Cambridge Texts in Applied Mathematics. Cambridge University Press,
  Cambridge, 2003.
\newblock With a foreword by Alexandre Chorin.

\bibitem[BCL16]{MR3565819}
Jean-David Benamou, Guillaume Carlier, and Maxime Laborde.
\newblock An augmented {L}agrangian approach to {W}asserstein gradient flows
  and applications.
\newblock In {\em Gradient flows: from theory to application}, volume~54 of
  {\em ESAIM Proc. Surveys}, pages 1--17. EDP Sci., Les Ulis, 2016.

\bibitem[BCMO16]{MR3555350}
Jean-David Benamou, Guillaume Carlier, Quentin M\'{e}rigot, and \'{E}douard
  Oudet.
\newblock Discretization of functionals involving the {M}onge--{A}mp\`ere
  operator.
\newblock {\em Numer. Math.}, 134(3):611--636, 2016.

\bibitem[BCW10]{MR2580954}
Martin Burger, Jos\'{e}~A. Carrillo, and Marie-Therese Wolfram.
\newblock A mixed finite element method for nonlinear diffusion equations.
\newblock {\em Kinet. Relat. Models}, 3(1):59--83, 2010.

\bibitem[Bre91]{brenier_polar}
Yann Brenier.
\newblock Polar factorization and monotone rearrangement of vector-valued
  functions.
\newblock {\em Comm. Pure Appl. Math.}, 44(4):375--417, 1991.

\bibitem[CCWW19]{carrillo2019primal}
Jose~A Carrillo, Katy Craig, Li~Wang, and Chaozhen Wei.
\newblock Primal dual methods for {W}asserstein gradient flows.
\newblock {\em arXiv preprint arXiv:1901.08081}, 2019.

\bibitem[CDPS17]{MR3635459}
Guillaume Carlier, Vincent Duval, Gabriel Peyr\'{e}, and Bernhard Schmitzer.
\newblock Convergence of entropic schemes for optimal transport and gradient
  flows.
\newblock {\em SIAM J. Math. Anal.}, 49(2):1385--1418, 2017.

\bibitem[CM10]{MR2566595}
J.~A. Carrillo and J.~S. Moll.
\newblock Numerical simulation of diffusive and aggregation phenomena in
  nonlinear continuity equations by evolving diffeomorphisms.
\newblock {\em SIAM J. Sci. Comput.}, 31(6):4305--4329, 2009/10.

\bibitem[CWXY20]{carrillo2020variational}
Jose~A Carrillo, Li~Wang, Wuzhe Xu, and Ming Yan.
\newblock Variational asymptotic preserving scheme for the
  {V}lasov--{P}oisson--{F}okker--{P}lanck system.
\newblock {\em arXiv preprint arXiv:2007.01969}, 2020.

\bibitem[DPMSV16]{santambrogio_bv}
Guido De~Philippis, Alp\'{a}r~Rich\'{a}rd M\'{e}sz\'{a}ros, Filippo
  Santambrogio, and Bozhidar Velichkov.
\newblock B{V} estimates in optimal transportation and applications.
\newblock {\em Arch. Ration. Mech. Anal.}, 219(2):829--860, 2016.

\bibitem[Eva10]{evans_book}
Lawrence~C. Evans.
\newblock {\em Partial differential equations}, volume~19 of {\em Graduate
  Studies in Mathematics}.
\newblock American Mathematical Society, Providence, RI, second edition, 2010.

\bibitem[Eyr98]{eyre}
David~J. Eyre.
\newblock Unconditionally gradient stable time marching the {C}ahn--{H}illiard
  equation.
\newblock {\em MRS Proceedings}, 529:39, 1998.

\bibitem[Gan94]{gangbo1994elementary}
Wilfrid Gangbo.
\newblock An elementary proof of the polar factorization of vector-valued
  functions.
\newblock {\em Archive for rational mechanics and analysis}, 128(4):381--399,
  1994.

\bibitem[Gan95a]{gangbo_habilitation}
Wilfrid Gangbo.
\newblock {\em Quelques probl{\`e}mes d'analyse non convexe. {H}abilitation
  {\`a} diriger des recherches en math{\'e}matiques.}
\newblock Habilitation, Universit{\'e} de Metz, January 1995.

\bibitem[Gan95b]{gangbo1995quelques}
Wilfrid Gangbo.
\newblock Quelques problemes d’analyse non convexe.
\newblock {\em Habilitation à diriger des recherches en math{\'e}matiques.
  Universit{\'e} de Metz (Janvier 1995)}, 1995.

\bibitem[GM96]{gangbo_mccann}
Wilfrid Gangbo and Robert~J. McCann.
\newblock The geometry of optimal transportation.
\newblock {\em Acta Math.}, 177(2):113--161, 1996.

\bibitem[JKO98]{JKO1998variational}
Richard Jordan, David Kinderlehrer, and Felix Otto.
\newblock The variational formulation of the {F}okker--{P}lanck equation.
\newblock {\em SIAM journal on mathematical analysis}, 29(1):1--17, 1998.

\bibitem[JKT20]{jacobskimtongL1}
Matt Jacobs, Inwon Kim, and Jiajun Tong.
\newblock The ${L}^1$-contraction principle in optimal transport.
\newblock {\em arXiv preprint arXiv:2006.09557}, 2020.

\bibitem[JL20]{jacobslegerbf}
Matt Jacobs and Flavien L{\'e}ger.
\newblock A fast approach to optimal transport: the back-and-forth method.
\newblock {\em Numerische Mathematik}, pages 1--32, Oct 2020.

\bibitem[LMSS20]{leclerc2020lagrangian}
Hugo Leclerc, Quentin M{\'e}rigot, Filippo Santambrogio, and Federico Stra.
\newblock Lagrangian discretization of crowd motion and linear diffusion.
\newblock {\em SIAM Journal on Numerical Analysis}, 58(4):2093--2118, 2020.

\bibitem[Luc97]{DBLP:journals/na/Lucet97}
Yves Lucet.
\newblock Faster than the fast {L}egendre transform, the linear-time {L}egendre
  transform.
\newblock {\em Numer. Algorithms}, 16(2):171--185, 1997.

\bibitem[Nes13]{nesterov2013introductory}
Yurii Nesterov.
\newblock {\em Introductory lectures on convex optimization: A basic course},
  volume~87.
\newblock Springer Science \& Business Media, 2013.

\bibitem[Ott01]{otto2001pme}
Felix Otto.
\newblock The geometry of dissipative evolution equations: the porous medium
  equation.
\newblock {\em Comm. Partial Differential Equations}, 26(1-2):101--174, 2001.

\bibitem[Pey15]{MR3413589}
Gabriel Peyr\'{e}.
\newblock Entropic approximation of {W}asserstein gradient flows.
\newblock {\em SIAM J. Imaging Sci.}, 8(4):2323--2351, 2015.

\bibitem[San15]{otam}
Filippo Santambrogio.
\newblock {\em Optimal transport for applied mathematicians}, volume~87 of {\em
  Progress in Nonlinear Differential Equations and their Applications}.
\newblock Birkh\"{a}user/Springer, Cham, 2015.
\newblock Calculus of variations, PDEs, and modeling.

\bibitem[V{\'a}z07]{vazquez2007porous}
Juan~Luis V{\'a}zquez.
\newblock {\em The porous medium equation: mathematical theory}.
\newblock Oxford University Press, 2007.

\end{thebibliography}
\end{document}